\documentclass[12pt,a4paper]{amsart}
 \usepackage{amsfonts,amssymb,enumerate,color,bm,hhline,makecell,array}
\usepackage{array,mathrsfs}
\usepackage[all,ps,cmtip]{xy}
\usepackage[normalem]{ulem}
\usepackage{bm,mathtools}
\usepackage{hyperref}

\usepackage{enumitem}
\setlist[itemize]{noitemsep,nolistsep}
\setlist[enumerate]{noitemsep,nolistsep}
\usepackage{colortbl}

\allowdisplaybreaks
\let\mathcal\mathscr

\def\Z{{\bf Z}}

\def\F{{\bf F}}
\def\C{{\bf C}}

\def\P{{\bf P}}

\def\hk{hyperk\"ahler}
\def\hK{\hk}

\def\SS{S^{[2]}}

\def\gv{ {\mathsf g}({\mathsf V})}

\def\phi{\varphi}

\def\cI{\mathcal{I}}

\def\cD{\mathcal{D}}
\def\mD{\mathbb{D}}

\def\cE{\mathcal{E}}
\def\cF{\mathcal{F}} 
\def\cG{\mathcal{G}}

\def\cI{\mathcal{I}}

\def\cK{\mathcal{K}}
\def\cL{\mathcal{L}}
\def\cM{\mathcal{M}}
\def\cMDV{\cM_{\rm DV}}
\def\tMDV{\widetilde\cM_{\rm DV}}

\def\cO{\mathcal{O}}
\def\cP{\mathcal{P}}
\def\cQ{\mathcal{Q}}

\def\cS{\mathcal{S}}
\def\cT{\mathcal{T}}

\def\cV{\mathcal{V}}
\def\cY{\mathcal{Y}}
\def\cX{\mathcal{X}}

\def\lra{\longrightarrow}
\def\llra{\hbox to 10mm{\tofill}}
\def\lllra{\hbox to 15mm{\tofill}}

\def\llla{\hbox to 10mm{\leftarrowfill}}
\def\lllla{\hbox to 15mm{\leftarrowfill}}

\def\dra{\dashrightarrow}
\def\thra{\twoheadrightarrow}
\def\hra{\hookrightarrow}
\def\lhra{\ensuremath{\lhook\joinrel\relbar\joinrel\to}}

\def\isom{\simeq}
\def\eps{\varepsilon}

\def\emptyset{\varnothing}

\def\gp{\mathfrak p}
\def\gq{\mathfrak q}

\DeclareMathOperator{\isomdra}{\stackrel{{}_{\scriptstyle\sim}}{\dra}}
\DeclareMathOperator{\isomlra}{\stackrel{{}_{\scriptstyle\sim}}{\lra}}

\DeclareMathOperator{\Aut}{Aut}
\DeclareMathOperator{\Out}{Out}

\def\div{\mathop{\rm div}\nolimits}
\DeclareMathOperator{\End}{End}
\DeclareMathOperator{\disc}{disc}
\DeclareMathOperator{\ev}{ev}

\DeclareMathOperator{\GL}{GL}

\DeclareMathOperator{\DD}{\mathsf{D}}
\DeclareMathOperator{\Gr}{\mathsf{Gr}}

\DeclareMathOperator{\Hilb}{Hilb}
\DeclareMathOperator{\Hom}{Hom}
\DeclareMathOperator{\Id}{Id}
\def\Im{\mathop{\rm Im}\nolimits}

\DeclareMathOperator{\Ker}{Ker}
\DeclareMathOperator{\len}{length}
\DeclareMathOperator{\lin}{\underset{\mathrm lin}{\equiv}}

\DeclareMathOperator{\Mov}{\overline{Mov}}

\DeclareMathOperator{\NS}{NS}

\DeclareMathOperator{\PGL}{PGL}
\DeclareMathOperator{\Pic}{Pic}

\DeclareMathOperator{\pr}{\mathsf{pr}}

\DeclareMathOperator{\SO}{SO}
\DeclareMathOperator{\Sp}{Sp}
\DeclareMathOperator{\Nef}{Nef}

\DeclareMathOperator{\rank}{rank}
\DeclareMathOperator{\Sing}{Sing}

\DeclareMathOperator{\Sym}{Sym}
\DeclareMathOperator{\SL}{SL}
\DeclareMathOperator{\gsp}{\mathfrak{sp}}

\DeclareMathOperator{\gsl}{\mathfrak{sl}}

\DeclareMathOperator{\Tr}{Tr}

\def\llra{\hbox to 10mm{\tofill}}
\def\lllra{\hbox to 15mm{\tofill}}

\def\bw#1{\textstyle{\bigwedge\hskip-0.9mm^{#1}}}
\def\sbw#1#2{\small{\bigwedge\hskip-0.9mm^{#1}}\hskip0.2mm{#2}}

\newtheorem{lemm}{Lemma}[section]
\newtheorem{theo}[lemm]{Theorem}
\newtheorem{coro}[lemm]{Corollary}
\newtheorem{prop}[lemm]{Proposition}

\theoremstyle{definition}
\newtheorem{defi}[lemm]{Definition}
\newtheorem{rema}[lemm]{Remark}

\newtheorem{exam}[lemm]{Example}

\theoremstyle{remark}
\newtheorem*{remark*}{Remark}
\newtheorem*{note*}{Note}

\def\moins{\smallsetminus}

\def\L{{\Lambda}}

\def\Lkkk[#1]{{\Lambda_{\KKK^{[#1]}}}}
\def\kkk[#1]{{\KKK^{[#1]}}}
\DeclareMathOperator{\KKK}{{K3}}

\def\sss[#1]{{S^{[#1]}}}

\newcommand{\aff}{{\Delta}}
\newcommand{\gquot}{/\!\!/}
\newcommand{\wh}{\widehat}
\newcommand{\wt}{\widetilde}

\newcommand{\green}[1]{\leavevmode{\color{green}{#1}}}

\newcommand{\la}{\langle}
\newcommand{\ra}{\rangle}
\newcommand{\ov}{\overline}
\newcommand{\gm}{\mathfrak{m}}

\newcommand{\cZ}{{\mathscr Z}}

\DeclareMathOperator{\sing}{Sing}
\DeclareMathOperator{\diag}{diag}
\DeclareMathOperator{\Bl}{Bl}
\DeclareMathOperator{\exc}{Exc}

\def\setminus{\smallsetminus}

\def\cong{\simeq}

\def\rem{}

\def\OO{\cO}
\def\PP{\P}
\def\CC{\C}
\def\dual{{\vee}}

\def\Qtrois{{Q_3}}  
\def\hilb{{Q_3^{[2]}}} 
\def\chow{{Q_3^{(2)}}} 

\def\L2{L}
\def\sL{\mathsf L}

\def\Tauto#1{\cT_{#1}}
\def\hilbS{S^{[2]}}  
\def\isospecS{\widetilde{S\times S}}  
\def\quotauto{{\cQ}} 
\def\soustauto{{\cS}} 
\title{Hilbert squares of K3 surfaces and Debarre--Voisin varieties}

 \author[O. Debarre]{Olivier Debarre}
 \address{Universit\'e de Paris, Sorbonne Universit\'e, CNRS, Institut de Math\'ematiques de Jussieu-Paris Rive Gauche, 75013 Paris, France
} \email{{\tt  olivier.debarre@imj-prg.fr}}

 \author[F. Han]{Fr\'ed\'eric Han}
 \address{Universit\'e de Paris, Sorbonne Universit\'e, CNRS, Institut de Math\'ematiques de Jussieu-Paris Rive Gauche, 75013 Paris, France
} \email{{\tt  frederic.han@imj-prg.fr}}

 \author[K. O'Grady]{Kieran O'Grady}
 \address{Sapienza Universit\`a di Roma, Dip.to di Matematica, P.le A.~Moro 5, 00185 Italia
} \email{{\tt  ogrady@mat.uniroma1.it}}

 \author[C. Voisin]{Claire Voisin}
 \address{Coll\`ege de France, 3 rue d'Ulm, 75005 Paris, France
} \email{{\tt  claire.voisin@imj-prg.fr}}

\date{\today}

\usepackage{comment,todonotes,multiaudience} 
\SetNewAudience{long} 
\SetNewAudience{short} 

\DefCurrentAudience{short}

 \subjclass[2010]{14J32, 14J35, 14M15, 14J70, 14J28}

\begin{document}

\begin{shownto}{long}
\green{
\tableofcontents. 
}
\end{shownto}

\begin{abstract}   Debarre--Voisin  \hk\  fourfolds are
built  from   alternating $3$-forms on a $10$-dimensional complex vector space, which we call   trivectors.\ They are analogous to the
Beauville--Donagi fourfolds associated with cubic fourfolds.\
In this article, we  study  several trivectors whose associated Debarre--Voisin variety is degenerate, in the sense that it is either reducible or has excessive dimension.\
We show that the Debarre--Voisin
varieties specialize, along  general  $1$-parameter degenerations  to these trivectors, to   varieties isomorphic or birationally isomorphic to the Hilbert square of a
  K3 surface.
\end{abstract}

\maketitle

\section{Introduction}

\rem{
Throughout this article, the notation $U_m$, $V_m$, or $W_m$ means an $m$-dimensional complex vector space.\ Let  $\sigma\in \bw3V_{10}^\dual$  be a nonzero alternating $3$-form (which we call a {\em trivector}).\ The Debarre--Voisin variety  associated with  $\sigma$ is
the scheme 
\begin{equation}  \label{eqDV}
K_\sigma:=\{[W_6]\in \Gr(6,V_{10}) \mid \sigma\vert_{ W_6}=0\} 
\end{equation}
whose points are the  6-dimensional vector subspaces of $V_{10}$ on which  $\sigma$ vanishes identically.\ 

It was proved in \cite{devo} that  for $\sigma$   general, the schemes $K_\sigma$, equipped with the polarization~$\cO_{K_\sigma}(1)$ (of square $22$ and divisibility 2; see Section~\ref{sec21}) induced by the Pl\"ucker polarization on $\Gr(6,V_{10}) $, form a locally complete family
of  smooth  polarized  \hk\ fourfolds which are deformation equivalent to  Hilbert squares  
of   $K3$ surfaces   (one says that $K_\sigma$ is  of $K3^{[2]}$-type).\ This was done by proving that when $\sigma$ specializes to
  a general element of 
  the discriminant hypersurface in $\bw3V_{10}^\dual$ where  the Pl\"ucker hyperplane section
 \begin{equation}\label{eqxs}
X_{\sigma}:=\{[U_3]\in \Gr(3,V_{10}) \mid \sigma\vert_{U_3}=0\} 
\end{equation}
becomes  singular, the scheme $K_\sigma$ becomes  singular along a surface but birationally isomorphic to the   Hilbert square 
of a  $K3$ surface
(the fact that   $K_\sigma$ is  of $K3^{[2]}$-type was reproved in \cite{KLSV}
by a different argument still based on the same specialization of $\sigma$).\ 

The projective 20-dimensional irreducible GIT quotient
\begin{equation*}
\cMDV=\PP(\bw3  V_{10}^\dual) \gquot \SL(V_{10})
\end{equation*}
is a coarse moduli space for 
 trivectors $\sigma$.\ Let~$\cF$ be the quasi-projective 20-dimensional irreducible period domain
for smooth polarized \hk\ varieties that are deformations of $(K_\sigma,\cO_{K_\sigma}(1))$.\ The corresponding period map
$$\gq\colon \cMDV \dra \cF 
$$
  is regular on the  open subset of $ \cMDV$  corresponding to   points $[\sigma]$ such that  $K_{\sigma}$ is a smooth fourfold.\ It is known to be
dominant (hence generically finite)  and was recently shown to be birational (\cite{ognew}).\ Consider the Baily--Borel projective compactification $\cF\subset \overline\cF$ (whose boundary has dimension $1$)  and a resolution 
\begin{equation}\label{diag}
\vcenter{
\xymatrix@R=5mm@M=2mm
{
\tMDV\ar[d]_-\eps\ar@{->>}[r]^-{\bar\gq} &\overline\cF\\
 \cMDV  \ar@{-->}[r]^-{\gq}&\cF\ar@{^(->}[u] 
}}
\end{equation}
 of the indeterminacies of $\gq$, where $\eps$ is birational.\ We define an {\em HLS}  {\em divisor} (for Hassett--Looijenga--Shah) to be an   irreducible hypersurface in $\overline\cF$ which  is the image by $\bar\gq$ of an exceptional divisor of $\eps$ (that is, whose image in $ \cMDV$ has codimension $>1$).\ These divisors reflect some   difference between the GIT   and the Baily--Borel compactifications and there are obviously only finitely many of them.\

 The main result of this article is the following (for the definition of the Heegner divisors $\cD_{2e}\subset\overline\cF $, see Section~\ref{sec21}).
 
 \begin{theo}\label{mainth}
The Heegner divisors $\cD_2$, $\cD_6$, $\cD_{10}$, and $\cD_{18}$ in $\overline\cF $ are HLS divisors.
\end{theo}

This statement puts together the more detailed conclusions of Theorems~\ref{theodue},~\ref{theohas},~\ref{theohasgenus6}, and~\ref{theohasgenus10}.\ These results are in fact more precise: we identify these divisors $\cD_2$, $\cD_6$, $\cD_{10}$, and $\cD_{18}$ with the periods of Hilbert squares of K3 surfaces with a suitable polarization (see Section~\ref{sec11} for more details).\ The {\it singular degenerations} of $\sigma$ discussed above correspond  to a hypersurface in $\tMDV$ mapped by $  \bar\gq$ onto the Heegner divisor $\cD_{22}$, which is  therefore {\em  not} an HLS divisor.

The study of this kind of problems has a long history that started with the work of Horikawa and Shah on polarized K3 surfaces of degree 2 (\cite{hor,sha}) and continued with the work of Hassett, Looijenga, and Laza on cubic fourfolds  (\cite{hassett,loo,looijenga,laza,lazagit}) and O'Grady on double EPW sextics  (\cite{og6,og5}), which   are   \hk\ fourfolds of $K3^{[2]}$-type with a polarization of square $2 $ and divisibility~$1$, associated with Lagrangian subspaces in~$\bw3V_6$.\

Let us describe briefly  the situation in the cubic fourfold case, which inspired the present study.\ One considers hypersurfaces $X_f\subset \P(V_6)$ defined by nonzero cubic polynomials~$f\in\Sym^3\!V_6^\vee$.\ When $f$ is general, 
the  variety 
$$F_f=\{[W_2]\in \Gr(2,V_{6}) \mid f\vert_{W_2}=0\}$$
of lines contained in $X_f$  was shown by Beauville--Donagi in \cite{beaudo} to be a \hk\ fourfold of $K3^{[2]}$-type, with a (Pl\"ucker) polarization of square $6$ and divisibility 2.\ There is again a birational surjective period map
$\widetilde\cM_{\textnormal{Cub}}\to \overline\cG$ which was   completely described by Laza.\ The divisor in $\widetilde\cM_{\textnormal{Cub}}$ that corresponds to singular cubics $X_f$ maps onto the Heegner divisor $\cD_6$.\ The only HLS divisor is $\cD_2$ (\cite{hassett,looijenga,laza}):
it  is obtained by blowing up, in the GIT moduli space $\cM_{\textnormal{Cub}}$, the    semistable point   corresponding to  chordal cubics $X_{f_0} $ (\cite[Section~4.1.1]{laza}).\
 
 O'Grady also proved   that $\cD'_2$, $\cD''_2$, and $\cD_4$ (in the notation of \cite[Corollary~6.3]{dims}; ${\mathbb S}'_2$, ${\mathbb S}''_2$, and~$ {\mathbb S}_4$ in the notation of \cite{og6}) are   HLS divisors  in the period domain of   double EPW sextics and conjectures that there are no others (see Section~\ref{section35}).\ They are also obtained by blowing up points in the GIT moduli space (corresponding to the semistable Lagrangians denoted by $A_k$,~$A_h$, and $A_+$ in \cite{og5}).


 The HLS divisors in Theorem~\ref{mainth} are obtained as follows:  while  general trivectors in $ \P(\bw3V_{10}^\vee)$ have finite stabilizers in $\SL(V_{10})$, we consider instead some
  special trivectors $\sigma_0 $ with positive-dimensional stabilizers and we   blow up their~$\SL(V_{10})$-orbits in $ {\P(\bw3V_{10}^\vee)}$.\ The stabilizers along the exceptional divisors of the resulting blown up space  for the  induced $\SL(V_{10})$-action are generically finite, thus producing  divisors in the quotient (this is  a Kirwan blow up).
  
We  describe  the corresponding Debarre--Voisin varieties $K_{\sigma_0}$.\ 
 In the simplest   cases \rem{(divisors~$\cD_6$ and~$\cD_{18}$)}, they are  still  smooth but of  dimension greater than~$4$.\ There is  an excess vector
  bundle $\cF$ of rank $\dim( K_{\sigma_0})-4$ on $K_{\sigma_0}$ 
  and
  the limit of the varieties~$K_{\sigma_t}$ under a general $1$-parameter degeneration $ (\sigma_t)_{t\in \Delta}$ to $\sigma_0$ 
  is the zero-locus of a general section of~$\cF$.\ In one other case  \rem{(divisor~$\cD_2$)}, the variety $K_{\sigma_0}$ is reducible of dimension~$4$ and  the limit of the   varieties~$K_{\sigma_t}$ is  birationally isomorphic  to the Hilbert square of  a degree-$2$ K3 surface; it is also a degree-$4$ cover of a nonreduced component of  $K_{\sigma_0}$ (very much like what happens for chordal cubics~$X_{f_0}$).

 As mentioned above, there is a relationship between these constructions and $K3$ surfaces; we actually discovered some of these special  trivectors 	and their stabilizers starting from $K3$ surfaces.\
  As   explained in Theorem~\ref{th31},    Hilbert squares of general polarized $K3$
  surfaces of fixed degree~$2e$
 appear as limits
     of Debarre--Voisin varieties   for infinitely many values of $e$, and they form a hypersurface  in $\tMDV$ that maps onto the Heegner divisor $\cD_{2e}$.\ Among these values, the only ones for which there exist  explicit geometric descriptions (Mukai  models for polarized K3 surfaces)  are
  $1$, $3$, $5$, $9$, $11$, and $15$   (\cite{mukai,mukai2,mukai3}).\ This is how we obtain the divisors in Theorem~\ref{mainth} (the case \mbox{$e=11$} corresponds to the singular degenerations of the trivector $\sigma$ mentioned above and does not produce an HLS divisor; our analysis of the case $e=15$ is still incomplete (see Section~\ref{sec115}) and we do not know whether $\cD_{30}$ is an HLS divisor).

      At this point, one may make a couple of general remarks: 
 \begin{itemize}
\item all known HLS divisors are obtained from blowing up single points in the moduli space;
\item all known HLS divisors are Heegner divisors.
\end{itemize}
We have no general explanation for these remarkable facts.\ 
 
Additionally, note that   HLS divisors are by definition uniruled (since they are obtained as images of exceptional divisors of blow  ups).\ They may correspond to periods of  Hilbert squares of   K3 surfaces of degree $2e$ only if the corresponding moduli space of polarized K3 surfaces is uniruled, which, by   \cite{ghs}, may only happen for $e\le 61$ (many thanks to an anonymous referee for making this very interesting remark).\ Adding in the restrictions on $e$ explained in Section~\ref{secrappelonpic}, one finds that   only 7 other Heegner divisors   can be HLS divisors coming from K3 surfaces (Remark~\ref{other}).\ Actually, we expect   $\cD_2$, $\cD_6$, $\cD_{10}$,  $\cD_{18}$, and $\cD_{30}$ to be the only  HLS divisors    (see Section~\ref{section35}).

We  now  describe the
  geometric situations encountered for $e\in\{1,3,5,9,15\}$.}

  \subsection{Stabilizers and $K3$ surfaces}\label{sec11}

We list here the various special trivectors  $[\sigma_0]\in  \P(\bw3V_{10}^\vee)$ that we consider,
their (positive-dimensional) stabilizers for the  $\SL(V_{10})$-action, and the corresponding limits of Debarre--Voisin varieties
   (which are all birationally isomorphic to Hilbert squares of K3 surfaces with suitable polarizations) along general $1$-parameter degenerations to~$\sigma_0$.\ In most cases,  the associated Pl\"ucker hypersurface $X_{\sigma_0}$ is singular and the singular locus of $X_{\sigma_0}$ gives rise to a component of $K_{\sigma_0}$, as explained in Proposition~\ref{remaxsing}(b).

  \subsubsection{The group $\SL(3)$ and  $K3$ surfaces of  degree $2$ \textnormal{(Section \ref{div2})}}  \label{s111}
  
A general degree-$2$ K3 surface $(S,L)$ is   a double cover of $ \P^2$ branched along a smooth sextic curve.\ The Hilbert square $\SS $ is birationally isomorphic to the moduli space  $\cM_S(0,L,1)$ of  sheaves on $S$ defined in Remark~\ref{rema33}.

We take $V_{10}:=\Sym^3 \!W_3$, so that $\bw3 V^{\vee}_{10}$ is an $\SL(W_3)$-representation, and we let $\sigma_0\in \bw3 V^{\vee}_{10}$ be a generator of the $1$-dimensional space  of
$\SL(W_3)$-invariants.\  

The Debarre--Voisin variety $K_{\sigma_0}$ is described in Proposition~\ref{kappapi}: it has two $4$-dimensional irreducible components $K_L$ and $K_M$ and is nonreduced along $K_L$.\ The Pl\"ucker hypersurface $X_{\sigma_0}$ is singular along a surface  (Proposition~\ref{singeffe}) and the
 component $K_L$ of $K_{\sigma_0}$ is  obtained from this surface by the procedure described in Proposition~\ref{remaxsing}(b) (see Proposition~\ref{prop710}(a)).

 Our main result is the following (Theorem~\ref{degenera2}).

{
 \begin{theo} \label{theodue}
Under a general $1$-parameter deformation
 $(\sigma_t)_{t\in \Delta}$,   the Debarre--Voisin fourfolds~$K_{\sigma_t}$ specialize, after a finite base change,  to a scheme which is isomorphic
 to $\cM_S(0,L,1)$, where $S$ is a general K3 surface of degree $2$.
\end{theo}}

This case is the most difficult:  the limit  fourfold $\cM_S(0,L,1)$ does not sit naturally in the Grassmannian $\Gr(6,V_{10})$ but maps 4-to-1  to it.

 {The limit   on $\cM_S(0,L,1)$ of the Pl\"ucker line bundles on  $K_{\sigma_t}$ is the 
 ample line bundle  of square $22$ and divisibility $2$
  described in Table~\ref{enumero1}.\ 
We show  that it is   globally generated for a general degree-$2$ $K3$ surface $S$, but not   very ample  (Remark~\ref{rema33}).}

  \subsubsection{The group  $\Sp(4)$ and  $K3$ surfaces of  degree $6$ \textnormal{(Section \ref{sec51})}}\label{se112}
   Let $V_4$ be a $4$-dimensional vector space equipped with a nondegenerate skew-symmetric form $\omega$.\ The hyperplane
$V_5\subset \bw2V_4 $  defined by $\omega$ is endowed with the nondegenerate quadratic form $q$ defined by wedge product,  and
{$\SO(V_5,q)\cong \Sp(V_4,\omega)$.}\  The form $q$ defines a smooth quadric $Q_3\subset \P(V_5)$ and general degree-$6$ $K3$ surfaces are complete intersections of  $Q_3$  and a cubic in $\P(V_5) $.

There is a natural trivector $\sigma_0$ on
the vector space $V_{10}:=\bw2V_5$:
 view
  elements of $V_{10}$ as   endomorphisms of $V_5$ which are  skew-symmetric  with respect to $q$ and  define
\begin{equation}\label{si0}
\sigma_0(a, b, c)=\Tr(a\circ b\circ c) .
\end{equation}
 The associated Debarre--Voisin variety $K_{\sigma_0}\subset \Gr (6, V_{10} ) $ was described by Hivert in \cite{hivert}: it is isomorphic to $\hilb$.\ In fact, the Pl\"ucker hypersurface $X_{\sigma_0}$ is singular along a copy of $Q_3$ (Lemma~\ref{le11}) and the whole of $K_{\sigma_0}$ is  obtained from $Q_3$ by the procedure described in Proposition~\ref{remaxsing}(b) (see Theorem~\ref{thoehivertDVtzero}).\  

The excess bundle analysis   shows  the following (Theorem~\ref{theoprecis}).

\begin{theo} \label{theohas}
 Under a general $1$-parameter deformation
 $(\sigma_t)_{t\in \Delta}$,   the Debarre--Voisin fourfolds~$K_{\sigma_t}$ specialize to a smooth subscheme of  $K_{\sigma_0}\isom \hilb$ which is   isomorphic to
 $S^{[2]}$, where $S\subset Q_3$ is a general degree-$6$ $K3$ surface.
\end{theo}

The restriction of the Pl\"ucker line bundle to $S^{[2]}\subset \hilb\isom K_{\sigma_0}\subset \Gr (6, V_{10} ) $ is the 
 ample line bundle of square 22 and divisibility~$2$  (see Section
\ref{sec21} for the definition of   divisibility) described in Table~\ref{enumero1}.\
It  is therefore very ample for a general degree-$6$ $K3$ surface $S$.

\subsubsection{The group $\SL(2)$ and  $K3$ surfaces of degree $10$  \textnormal{(Section \ref{secpf6})}}

The subvariety $X\subset \Gr(2,V_5^\vee)\subset \P(\bw2V_5^\vee)$ defined by a general $3$-dimensional
space $W_3\subset \bw2V_5$ of   linear Pl\"ucker equations is a degree-$5$  Fano threefold.\ General degree-$10$ $K3$ surfaces are quadratic sections of~$X$
 (\cite{mukai}).

The spaces $V_5$ and $W_3$ and the variety $X$ carry $\SL(2)$-actions and there is an $\SL(2)$-invariant decomposition $V_{10}:=\bw2V_5= V_7 \oplus W_3$.\ 
Among the  $\SL(2)$-invariant trivectors, there
  is a natural one,~$\sigma_0$,  defined in Proposition~\ref{proexistencesigma}, and the neutral component of its stabilizer is  $\SL(2) $.\
  
  The Debarre--Voisin $K_{\sigma_0}$ has one component  $K_1$ which is generically smooth   and   birationally isomorphic to $X^{[2]}$.\ In fact, the Pl\"ucker hypersurface $X_{\sigma_0}$ is   singular along a copy of the threefold~$X$  and  $K_1$   is  obtained from  $X$ by the procedure described in Proposition~\ref{remaxsing}(b) (see Proposition~\ref{lexsing}).\ 
  
We  obtain the following (Proposition \ref{prop69} and Theorem~\ref{theopourgenre6}).

\begin{theo} \label{theohasgenus6}
Under a general $1$-parameter deformation
 $(\sigma_t)_{t\in \Delta}$,   the Debarre--Voisin fourfolds~$K_{\sigma_t}$ specialize, after finite base change,  to a smooth subscheme of $K_{\sigma_0}$  which is isomorphic
 to $S^{[2]}$,
  where $S\subset X$ is a general K3 surface of degree $10$.
\end{theo}

The limit   on $\SS$ of the Pl\"ucker line bundles on  $K_{\sigma_t}$ is the 
 ample line bundle  of square $22$ and divisibility $2$
  described in Table~\ref{enumero1}.\ 
We show  that it is not  globally generated.

\subsubsection{The group $G_2\times \SL(3)$ and   $K3$ surfaces of  degree $18$ \textnormal{(Section \ref{secpf10})}}\label{se114}

The group~$G_2$ is the subgroup of $\GL(V_7)$ leaving a general $3$-form $\alpha$ invariant.\
There is a   $G_2$-invariant Fano 5-fold $X\subset \Gr(2,V_7)$ which  has index $3$, and general  $K3$ surfaces of degree $18$ are
obtained by intersecting~$X$ with a general $3$-dimensional space $W_3\subset \bw2V_7^\dual$ of linear Pl\"ucker  equations (\cite{mukai}).

The vector space $V_{10}:= V_7\oplus W_3$ is acted on diagonally by the group $G_2\times \SL(W_3)$ and  {we consider    $G_2\times \SL(W_3)$-invariant 
  trivectors $\sigma_0=\alpha+\beta$, where $\beta$ spans $\bw3W_3^\dual$.\ The corresponding points $[\sigma_0]$ of $ \P(\bw3V_{10}^\vee)$ are all in the same $\SL(V_{10})$-orbit and the corresponding 
  Debarre--Voisin variety $K_{\sigma_0}$ splits as a product of a smooth variety of  dimension $8$ and of $\P(W_3^\dual)$ (Corollary~\ref{corokalphabeta}).
 }
  
The excess bundle analysis   shows  the following (Theorem~\ref{theogenre10}).

\begin{theo} \label{theohasgenus10}
 Under a general $1$-parameter deformation
 $(\sigma_t)_{t\in \Delta}$,   the Debarre--Voisin fourfolds~$K_{\sigma_t}$ specialize to a smooth subscheme of $K_{\sigma_0}$  isomorphic 
 to $S^{[2]}$, where $S\subset X$ is a general  $K3$ surface of degree $18$.
\end{theo}

The limit   on $\SS$ of the Pl\"ucker line bundles on  $K_{\sigma_t}$ is the 
 ample line bundle  of square~$22$ and divisibility $2$
  described in Table~\ref{enumero1}.\ 
 It  is therefore very ample for a general  $K3$ surface $S$ of degree 18 (Lemma~\ref{lem513}).
  
\subsubsection{$K3$ surfaces of  degree $30$ \textnormal{(Section \ref{ch8})}}\label{sec115}
This is the last case allowed by the
numerical conditions of Section~\ref{sec33} where a projective model of a
general $K3$ surface $S$ is known.\ It corresponds to the last column of Table~\ref{enumero1}.\ However, the current geometric knowledge for
those $K3$ surfaces (see~\cite{mukai2}) is not as thorough as in the previous cases and we were not able  \rem{to map (nontrivially)~$\hilbS$ to a Debarre--Voisin variety nor to decide whether~$\cD_{30} $ is an HLS divisor.\ 

In some   cases (divisors $\cD_6$ and $\cD_{18}$), we first constructed a     rank-$4$ vector bundle on $\hilbS$ that defined a rational map
  $\hilbS\dra \Gr(6,10)$ and then found a (nonzero) trivector vanishing on the image.\ In Section~\ref{g16Q4}, we complete  the first step by constructing, for $S$ general K3 surface of degree $30$, a canonical rank-$4$ vector bundle on       $\hilbS$ with the same
numerical invariants as the restriction of the tautological quotient bundle of
$\Gr(6,10)$ to a Debarre--Voisin variety.\ We also obtain a geometric
interpretation of the image of the   rational map
  $\hilbS\dra \Gr(6,10)$ that it defines.\ Such a vector bundle is expected to be unique; it
is    modular in the sense  of \cite{ognew}.}


\section{Moduli spaces and period map}

\subsection{Polarized \hk\ fourfolds of degree 22 and divisibility 2 and their period map}\label{sec21}

Let $X$ be a \hk\ fourfold of $K3^{[2]}$-type.\ The abelian group $H^2(X,\Z)$ is free abelian of rank 23 and it carries a nondegenerate integral-valued quadratic form $q_X$ (the   Beauville--Bogomolov--Fujiki  form) that satisfies
$$\forall \alpha\in H^2(X,\Z)\qquad \int_X\alpha^4=3q_X(\alpha)^2
.$$
The lattice $(H^2(X,\Z),q_X)$ is  isomorphic to the lattice
$$(\Lambda,q_\Lambda):=U^{\oplus 3}\oplus E_8(-1)^{\oplus 2}\oplus I_1(-2),$$
where $U$ is the hyperbolic plane, $E_8$ the unique positive definite even rank-$8$ unimodular lattice, and $I_1(-2) $ the rank-$1$ lattice whose generators have square $-2$.\ 

The divisibility $\div(\alpha)$ of a nonzero element $\alpha$ of a lattice $( L,q_L)$ is the positive generator of the subgroup   $q_L(\alpha, L)$ of $ \Z$.\ There is a unique $O(\Lambda)$-orbit of primitive elements $h\in \Lambda$ such that $q_\Lambda(h)=22$ and $\div(h)=2$ (\cite[Corollary~3.7 and Example~3.10]{ghscomp}) and we fix one of these elements $h$.

We consider pairs $(X,H)$, where  $X$ be a \hk\ fourfold of $K3^{[2]}$-type and $H$ is an ample line bundle on $X$ such that $q_X(H)=22$ and $\div(H)=2$.\ It follows from Viehweg's work~\cite{vie} that there is a quasi-projective 20-dimensional coarse moduli space $\cM$ for these pairs and Apostolov proved in \cite{apo} that $\cM$ is irreducible.

The domain
 \begin{equation*}
\mD(h^\bot):=\{[\alpha]\in\PP(\Lambda\otimes\CC) \mid   q_\Lambda(\alpha, \alpha) =q_\Lambda(\alpha, h) =0,\  q_\Lambda(\alpha,\overline \alpha) >0\}
\end{equation*}
has two  connected components, both isomorphic to the   20-dimensional bounded symmetric domain of type IV associated with the lattice $h^\bot\subset \Lambda$.\
It is acted on properly and discontinuously by the {isometry group} $  O(h^\bot)$
and the quotient
$$\cF:= O(h^\bot)\backslash \mD(h^\bot)$$
 is, by Baily--Borel's theory, an irreducible 20-dimensional quasi-projective variety.

The Torelli theorem takes the following form  for our \hk\ fourfolds  (\cite{ver}, \cite[Theorem~3.14]{ghssur}, \cite[Theorem~8.4]{marsur}).

\begin{theo}[Verbitsky, Markman]\label{torth}
The period map
$$\gp\colon  \cM\lra \cF
$$  is an (algebraic) open embedding.
\end{theo}

Let us describe its image.\  Given an element $v\in h^\bot$ of negative square, we define the associated {\em Heegner divisor}
as the image by the quotient map $  \mD(h^\bot)\to\cF$ of the hypersurface
\begin{equation*}
 \{[\alpha]\in \mD(h^\bot) \mid  q_\Lambda(\alpha  , v ) =0\} .
\end{equation*}
It is an irreducible algebraic divisor in $\cF$ that only depends  on the even
negative integer $-2e:=\disc(v^\bot)$ (\cite[Proposition~4.1(2)(c)]{dm}).\  We
denote it by $\cD_{2e}$; it is nonempty if and only if $e$ is  positive and a
square modulo $11$ (see the end of the proof of \cite[Proposition~4.1]{dm}).\ The following result is \cite[Theorem~6.1]{dm}.

 \begin{prop}[Debarre--Macr\`\i ]\label{image}
The image of the period map
 $\gp\colon  \cM\hra \cF
 $  is the complement of the irreducible divisor $\cD_{22}$.
\end{prop}

\subsection{Debarre--Voisin varieties} 

We now relate this material with the  constructions in \cite{devo}.\ Let $V_{10}$ be a $10$-dimensional vector space.\ As in \eqref{eqDV}, one can associate with  a nonzero
 $\sigma\in \bw3V_{10}^\dual$ a subscheme $K_\sigma\subset \Gr(6,V_{10})$ which, for $\sigma$ general, is a
   \hk\  fourfold  of $K3^{[2]}$-type; the polarization $H$ induced by this embedding then satisfies $q_{K_\sigma}(H)=22$ and $\div(H)=2$.

 \rem{We defined in the introduction
the GIT
coarse moduli space
$
\cMDV=\PP(\bw3  V_{10}^\dual) \gquot \SL(V_{10})
$
  for Debarre--Voisin varieties.}

\begin{prop}
Let $[\sigma]\in \PP(\bw3  V_{10}^\dual)$.\ If $K_{\sigma}$ is smooth of dimension $4$,
the point $[\sigma]$ is $\SL(V_{10})$-semistable.
\end{prop}

\begin{proof}
Let $\PP(\bw3  V_{10}^\dual)^{\textnormal{sm}}\subset \PP(\bw3  V_{10}^\dual)$ be the open subset of points $[\sigma]$ such that $K_{\sigma}$ is smooth of dimension $4$.\ The map
\begin{equation*}
\wt{\gp}\colon \PP(\bw3  V_{10}^\dual)^{\textnormal{sm}} \lra \cF 
\end{equation*}
that sends $[\sigma]$ to the period of $K_\sigma$ is  {\em regular}.\
 Let  $[\sigma]\in \PP(\bw3  V_{10}^\dual)^{\textnormal{sm}}$.\
  Let  $D$ be a nonzero effective divisor on the quasi-projective variety $\cF $   such that
$\wt{\gp}([\sigma])\notin D$.\
 The closure of $ {\wt{\gp}^{-1}(D)}$ in $\PP(\bw3  V_{10}^\dual)$ is the divisor of a $\SL(V_{10})$-invariant section of some power of $\cO_{\P(\sbw3  V_{10}^\dual)}(1)$, which does not vanish at $[\sigma]$,   hence $[\sigma]$ is $\SL(V_{10})$-semistable.
\end{proof}

There is a modular map
$$\gm\colon \cMDV \dra \cM ,\qquad [\sigma]\longmapsto [K_\sigma]
$$
which is regular on the  open subset $\cMDV^{\textnormal{sm}}\subset \cMDV$  corresponding to   points $[\sigma]$ such that  $K_{\sigma}$ is a smooth fourfold.\ In the diagram~\eqref{diag} from the introduction, the map $\gq$ is $\gp\circ\gm$.

\section{\rem{Hilbert squares of   $K3$ surfaces as specializations of Debarre--Voisin varieties}}\label{secrappelonpic}

\rem{In this section, we exhibit, in the period domain $\cF$ for Debarre--Voisin varieties, infinitely many Heegner divisors whose  general points are  periods of polarized \hk\ fourfolds   that are birationally isomorphic to  Hilbert squares of polarized K3 surfaces.\ We will prove in the next sections that some of these divisors are HLS divisors.\ The whole section is devoted to the proof of  the following theorem.\ It is based on the results   and techniques of  \cite{bama, dm, hast2}.}

   \begin{theo}\label{th31}
\rem{In the moduli space $\cM$ for \hk\ fourfolds of $K3^{[2]}$-type with a polarization of square $22$ and divisibility $2$, there are countably many irreducible hypersurfaces whose general points   correspond to polarized \hK\ fourfolds that are birationally isomorphic to  Hilbert squares of polarized K3 surfaces.\ Among them, we have}
\begin{itemize}
\item fourfolds that are isomorphic to $(\cM_S(0,L,1),\varpi_*(6\L2-5\delta ))$, where $(S,L)$ is a general polarized K3 surface of degree $2$;\footnote{See Remark~\ref{rema33} for the notation.}
\item  fourfolds that are isomorphic to $(S^{[2]}, 2\L2-(2m+1)\delta)$, where $(S,L)$ is a general polarized K3 surface of degree $2(m^2+m+3)$ (for any $m\ge 0$).
\end{itemize}
In the first case, the periods dominate the Heegner divisor
  $\cD_{2 }$.\ In the second case, the periods dominate the Heegner divisor
  $\cD_{2(m^2+m+3)}$.
\end{theo}

\subsection{The movable   cones of Hilbert squares of very general polarized K3 surfaces}\label{sect31}
Let   $(S,L)$ be a polarized K3 surface with $\Pic(S)=\Z L$ and $L^2=2e$.\ We have
  $$\NS(S^{[2]})\isom \Z \L2 \oplus \Z \delta,
  $$
  where $\L2$ is 
 the line bundle on the  Hilbert square $S^{[2]} $  induced by $L$  and $2\delta$ is the class of the exceptional divisor of the Hilbert--Chow morphism $S^{[2]}\to S^{(2)}$ (see Section~\ref{sectaut}).\ One has
 $$q_{S^{[2]}}(\L2)=2e\ ,\quad q_{S^{[2]}}(\delta)=-2\ ,\quad q_{S^{[2]}}(\L2,\delta)=0
 .$$
Let $(X,H)$ correspond to an element of $\cM$.\ If there is a birational isomorphism $\varpi\colon S^{[2]}\dra X$, one can write $\varpi^*H=2b\L2 -a\delta$, where $a$ and $b$ are positive integers \rem{(the coefficient of $\L2 $ is even because $H$ has divisibility $2$)}.\ Since $q_X(H)=22$, they satisfy the quadratic equation
\begin{equation}\label{p11}
a^2-4eb^2=-11  .
\end{equation}
Moreover, the class $2b\L2-a\delta$ is movable.\

 The closed movable  cone $\Mov(S^{[2]})$   was determined in~\cite{bama} (see also \cite[Example~5.3]{dm}): one extremal ray is spanned by $\L2$ and the other by $\L2-\mu_e\delta$, where the rational number $\mu_e$ is determined as follows:
\begin{itemize}
\item if $e$ is a perfect square, $\mu_e=\sqrt{e}$;
\item if $e$ is not a perfect square, $\mu_e=eb_1/a_1$, where $(a_1,b_1)$ is the minimal positive (integral) solution of the Pell equation $x^2-ey^2=1$.
\end{itemize}

\rem{The next proposition explains for which integers $e$ there is a movable class of square $22$ and divisibility $2$ on~$S^{[2]} $.}

  \begin{prop}\label{pr01}
Let $e$ be a positive integer such that the equation \eqref{p11} has a solution and let $(a_2,b_2)$ be the minimal positive solution.\ The numbers $ e,a,b $
such that the class $2b\L2-a\delta$ is movable on $S^{[2]}$ and of square $22$ are precisely  the following:
\begin{itemize}
\item  $e=1$ and $(a,b)=(5,3)$;
\item  $e=9$ and $(a,b)=(5,1)$;
\item  $e$ is not a perfect square, $b_1$ is even,  and $(a,b)$ is either $(a_2,b_2)$ or $ (2eb_1b_2-a_1a_2,a_1b_2 -\frac12a_2b_1)$ (these pairs are equal if and only if $11\mid e$);
\item  $e$ is not a perfect square,  $b_1$ is odd, and $(a,b)=(a_2,b_2)$.
\end{itemize}
\end{prop}

\begin{proof}
  Assume first that $m:=\sqrt{e}$ is an integer.\ The equation \eqref{p11} is then
  $$(a-2bm)(a+2bm)=-11,$$
  with $a+2bm>|a-2bm|$, hence
 $a+2bm=11$ and $a-2bm=-1$, so that  $a=5$ and $bm=3$.\ The only two possibilities are $e=1$ and $(a,b)=(5,3)$, and $e=9$ and $(a,b)=(5,1)$.\ In both cases, one has indeed $a/2b<\sqrt{e}$, hence the class
 $2b\L2-a\delta$ is movable.

   Assume that $e$ is not a perfect square.\ Set
 $x_2:=a_2+b_2\sqrt{e}\in\Z[\sqrt{e}]$ and $\bar x_2:=a_2-b_2\sqrt{e}$, so that $x_2\bar x_2= -11$ and  $0<-\bar x_2< \sqrt{11}<x_2$.

 We also set $x_1:=a_1+b_1\sqrt{e}$ and $\bar x_1:=a_1-b_1\sqrt{e}$, so that $x_1\bar x_1= 1$ and $0<\bar x_1< 1<x_1$.

 Let $(a'_1,b'_1)$ be the minimal positive solution of the Pell equation $x^2-4ey^2=1$ and set $x'_1:=a'_1+b'_1\sqrt{4e}$.\ If $b_1$ is even, we have $x'_1=x_1$ and  $b'_1=b_1/2$.\ If $b_1$ is odd, we have $x'_1=x_1^2$ and $b'_1=a_1b_1$.


By \cite[Theorem~110]{nage},
 all the  solutions of the equation \eqref{p11} are given by $\pm x_2x_1^{\prime n}$ and $\pm \bar x_2x_1^{\prime n}$, for  $n\in \Z$.\ Since $x'_1>1$, we have $0<x_2x_1^{\prime -1}<x_2$.\ Since $x_2$ corresponds to a minimal solution, this implies $x_2x_1^{\prime -1}< \sqrt{11}$, hence $-\bar x_2x'_1>\sqrt{11}$.\ By minimality of $x_2$ again, we get $-\bar x_2x'_1\ge x_2$.\ It follows that the positive solutions of the equation \eqref{p11} correspond to the following increasing sequence of elements of $\Z[\sqrt{e}]$:
 \begin{equation}\label{sols}
\sqrt{11}<x_2\le -\bar x_2x'_1<x_2x'_1\le -\bar x_2x_1^{\prime 2}<x_2x_1^{\prime 2}<\cdots
\end{equation}
 By \cite[Theorem~110]{nage} again, we have $x_2= -\bar x_2x'_1$ if and only if $11\mid e$.

 Since the function $x\mapsto  x-\frac{11}{x}$ is increasing on the interval $(\sqrt{11},+\infty)$, the corresponding positive solutions $(a,b)$ have increasing $a$ and $b$, hence increasing ``slopes'' $a/2b=\sqrt{\frac{e}{ 1+\frac{11}{a^2}}}$.

 We want to know for which of these positive solutions $a+b\sqrt{4e}$ the corresponding class $2b\L2-a\delta$ is movable, that is, satisfies $\frac{a}{2b} \le \mu_e=\frac{eb_1}{a_1}$.\

{\em Assume first that $b_1$ is even,} so that $x'_1=x_1$.\ The inequality $x_2\le -\bar x_2x'_1 $ translates into  $a_2\le -a_2a_1+2eb_2b_1$, hence
 \begin{equation}\label{slope}
\frac{a_2}{2b_2}\le \frac{eb_1}{a_1+1}.
\end{equation}
The  class corresponding to the solution $-\bar x_2x'_1=
2eb_1b_2 -a_1a_2 +(2a_1b_2-a_2b_1) \sqrt{e}$
 is movable if and only if we have
$$
\begin{array}{crcl}
& \dfrac{2eb_1b_2-a_1a_2}{2a_1b_2-a_2b_1}&\le& \dfrac{eb_1}{a_1}\\
 \Longleftrightarrow& a_1(2eb_1b_2-a_1a_2)&\le& eb_1(2a_1b_2-a_2b_1)\\
 \Longleftrightarrow& a_2(eb_1^2-a_1^2)&\le& 0,
\end{array}
$$
which holds since 
  $eb_1^2-a_1^2=-1$.\ This class is therefore movable, and so is the class corresponding to the minimal solution since it has smaller slope.

The  class corresponding to the next solution $x_2x_1=
 a_1a_2+2eb_2b_1 +(a_2b_1+2a_1b_2) \sqrt{e}$ is movable if and only if we have
 $$
\begin{array}{crcl}
& \dfrac{a_1a_2+2eb_2b_1}{a_2b_1+2a_1b_2}&\le& \dfrac{eb_1}{a_1}\\
 \Longleftrightarrow& a_1(a_1a_2+2eb_2b_1)&\le& eb_1(a_2b_1+2a_1b_2)\\
 \Longleftrightarrow& a_2(a_1^2-eb_1^2)&\le& 0,
\end{array}
$$
which   does not hold since $a_1^2-eb_1^2 =1$.\ This class is therefore not movable.

{\em  Assume now that $b_1$ is odd,} so that $x'_1=x_1^2=2a_1^2-1+2a_1b_1\sqrt{e}$.\ The inequality $x_2\le -\bar x_2x'_1 $ translates into  $a_2\le 4ea_1b_1b_2-a_2(2a_1^2-1)$, hence
\begin{equation}\label{ab}
\frac{a_2}{2b_2}\le \dfrac{eb_1}{a_1},
\end{equation}
which means exactly that the  class corresponding to the minimal solution $ x_2=a_2 +2b_2 \sqrt{e}$
 is movable (and it is on the boundary of the movable cone if and only if $11\mid e$).

The  class corresponding to the next solution $-\bar x_2x'_1=
(-a_2 +2b_2 \sqrt{e})(2a_1^2-1+2a_1b_1\sqrt{e})
=-a_2(2a_1^2-1)+4ea_1b_1b_2+(2b_2(2a_1^2-1)-2a_1a_2b_1)\sqrt{e}
$ is   movable if and only if
  $$
\begin{array}{crcl}
& \dfrac{-a_2(2a_1^2-1)+4ea_1b_1b_2}{2b_2(2a_1^2-1)-2a_1a_2b_1}&\le& \dfrac{eb_1}{a_1}\\
 \Longleftrightarrow& -a_1a_2(2a_1^2-1)+4ea_1^2b_1b_2&\le& 2eb_1b_2(2a_1^2-1)-2ea_1a_2b_1^2\\
 \Longleftrightarrow&2eb_1b_2  &\le& a_1a_2.
\end{array}
$$
It follows that the class is not movable unless there is equality in \eqref{ab}, which happens exactly when
$-\bar x_2x'_1=x_2$.\ Finally, one checks that the next solution $ x_2x'_1$ never corresponds to a movable class.
\end{proof}

\subsection{\rem{The nef cones of Hilbert squares of very general polarized K3 surfaces}}\label{sect32}

Let again $(S,L)$ be a polarized K3 surface  with $\Pic(S)=\Z L$  and $L^2=2e$.\ The nef cone   $\Nef(S^{[2]})$ was determined in \cite{bama} (see also \cite[Example~5.3]{dm}): one extremal ray is spanned by $\L2$,   and $\Nef(S^{[2]})=\Mov(S^{[2]})$, unless the equation $x^2-4ey^2=5$ has  integral solutions; if the minimal positive solution of that equation is $(a_5,b_5)$,
 the other extremal ray of $\Nef(S^{[2]})$ is then  spanned by $\L2-\nu_e\delta$, where $\nu_e=2eb_5/a_5<\mu_e$.

 Furthermore, in the latter case, in the  decomposition
(\cite[Theorem~7]{hast2})
\begin{equation}\label{chamb}
\Mov(S^{[2]})=\overline{\bigcup_{\varpi\colon S^{[2]}\isomdra X\atop X\textnormal{ \hK}}\varpi^*(\Nef(X))}
\end{equation}
into cones  which are either equal or have disjoint interiors, there are only two cones (this means that there is a unique   nontrivial birational map $\varpi\colon S^{[2]}\isomdra X$), unless
      $b_1$ is even and $5\nmid e$,  in which case  there are three cones (\cite[Example~3.18]{survey}).

\subsection{\rem{Movable and nef   classes of square~$22$ and divisibility~$2$}}\label{sec33}

\rem{We put together the results of Sections~\ref{sect31} and~\ref{sect32} and determine all  positive integers $e\le 22$ for which there exist movable or ample classes of square~$22$ and divisibility~$2$ on the Hilbert square of a very general polarized K3 surface of degree $2e$.

For that, the quadratic equation~\eqref{p11} needs to have solutions (and we denote by~$(a_2,b_2)$ its minimal positive solution).\ Table~\ref{enumero1} also indicates the minimal positive solution~$(a_1,b_1)$ of the Pell equation $x^2-ey^2=1$ (which  is used to compute the slope $\mu_e$ of the nef cone) and the slope~$\nu_e $ of the ample cone (computed as explained in Section~\ref{sect32}).
}
 
  \begin{table}[htb!]
\renewcommand\arraystretch{1.5}
 \begin{tabular}{|c|c|c|c|c|c|c|}
 \hline
$e$&$1$&$3$&$5$&$9$&$11$&$15$
 \\
 \hline
$(a_1,b_1)$
 &$-$&$(2,1)$&$(9,4)$ &$-$ &$(10,3)$ &$ (4,1)$
\\
 \hline
 $
\begin{matrix}\mu_e\\\textnormal{(slope of movable cone)}\end{matrix}$
 &$1$&$3/2$&$20/9$&$3$&$33/10$&$ 15/4$
\\
 \hline
$(a_2,b_2)$
 &$(5,3)$&$(1,1)$&$(3,1)$ &$(5,1)$ &$(33,5)$ &$(7,1)$
\\
 \hline
$\begin{matrix}\textnormal{movable classes of}\\  \textnormal{square $22$ and div.\ $2$}\end{matrix}$
 &$6\L2-5\delta$&$ 2\L2-\delta$&$ \begin{matrix}2\L2-3\delta\\ 6\L2-13\delta\end{matrix}$&$2\L2-5\delta$&$10\L2-33\delta$&$ 2\L2-7\delta$
\\
 \hline
$ \begin{matrix}\nu_e\\\textnormal{(slope of ample cone)}\end{matrix}
 $&$2/3$&$3/2$&$2$&$3$&$22/7$&$15/4$
 \\
\hline
$\begin{matrix}\textnormal{ample classes of}\\  \textnormal{square $22$ and div.\ $2$}\end{matrix}
$ &$-$&$2\L2-\delta$&$2\L2-3\delta$&$2\L2-5\delta$&$-$&$ 2\L2-7\delta$ 
\\
 \hline
 \end{tabular}
\vspace{5mm}
\caption{Movable and nef   classes of square~$22$ and divisibility~$2$ in $\SS$ for $e\le 22$}\label{enumero1}
\end{table}


%


 \begin{rema}\label{rema33a}
When $e=5$, the decomposition~\eqref{chamb} has  two cones and  $S^{[2]}$ has a unique nontrivial birational automorphism.\ It is an involution  $\varpi$ which was described geometrically in \cite[Proposition~4.15, Example~4.16]{survey}.\ One has $\varpi^*(2\L2-3\delta)=6\L2-13\delta$ and $S^{[2]}$ has no nontrivial \hK\ birational models.
 \end{rema}

 \begin{rema} 
  A consequence of Proposition~\ref{pr01} is that there are always one or two movable classes of square $22$ and divisibility $2$ as soon as the equation~\eqref{p11} has a solution.\ As Table~\ref{enumero1} shows, it can happen that some of these classes are not ample.\ 
It can also happen that both of these classes are ample (this is the case when $e=45$).
 \end{rema}
 
  \begin{rema}\label{other}
  \rem{We mentioned in the introduction that HLS divisors coming from polarized K3 surfaces of degree $2e$ may only occur if the  corresponding moduli space of polarized K3 surfaces is uniruled.\ This may only happen for $e\in\{1,2,\dots,45,47,48,49,51,53,55,56,59,61\} $ by   \cite{ghs}.\ One can continue Table~\ref{enumero1} for those   values of $e$ and find that only $\cD_{46}$, $ 
\cD_{54}$, $\cD_{66}$, $\cD_{90}$, $\cD_{94}$, $\cD_{106}$, and $\cD_{118}$ may be HLS divisors coming from polarized K3 surfaces.}
 \end{rema}

\subsection{\rem{Proof of Theorem~\ref{th31}}}

\rem{Let again $(S,L)$ be a polarized K3 surface  with $\Pic(S)=\Z L$  and~$L^2=2e$.

  When $e=1$, the decomposition~\eqref{chamb} has  two cones and $ \SS$ has a unique  nontrivial \hK\ birational model; it is  the
 moduli space $X_S:=\cM_S(0,L,1)$ of $L$-semistable pure sheaves on $S$ with Mukai vector $(0,L,1)$.\ As   we see from   Table~\ref{enumero1}, the square-$22$ class $H:=6L-5\delta$  is ample on $X_S$.\ The pair $(X_S,H)$ therefore defines an element of the moduli space $\cM$ and this proves the first item of the theorem.}

\rem{Assume now $e=m^2+m+3$, where
  $m$ is a nonnegative integer, so that $(a_2,b_2)=(2m+1,1)$.\ By Proposition~\ref{pr01}, the class $2\L2-(2m+1)\delta$ is always movable.\ One checks that its slope $(2m+1)/2$ is always smaller than the slope $\nu_e$ of the nef cone, hence this class is in fact  always ample.\ This proves the second item of the theorem.

Finally, in
 the general case, the orthogonal of $\NS(S^{[2]}) $ in the lattice $\Lambda$ is isomorphic to the orthogonal of~$L$ in the (unimodular) K3 lattice $H^2(S,\Z)$.\ Its discriminant is therefore $-2e$ and, whenever $H$ is an ample class of of square~$22$ and divisibility~$2$, the period of $(S^{[2]},H)$ is a general point of the Heegner divisor $\cD_{2e}$.\ Note also that although we only worked with very general polarized K3 surfaces, ampleness being an open condition still holds when $S$ is a general polarized K3 surface.\ This finishes the proof of the theorem.}
 
 \begin{rema}\upshape\label{rema33}
\rem{Going back to the case $e=1$  with the notation introduced in the proof above,} a 
 general element of $X_S$
  corresponds to a sheaf  $\iota_{*}\xi$, where $C\in |L|$, the map  $\iota\colon C\hra S$ is the inclusion, and $\xi$ is a  degree-$2$ invertible sheaf on $C$ (\cite[Example~0.6]{muk0}).\  
  The  birational map $
 \varpi\colon
 S^{[2]}\isomdra X_S$   takes a general $Z\in S^{[2]}$ to the  sheaf $\iota_{*}\cO_C(Z)$, where $C$ is the unique element of~$|L|$ that contains $Z$.\ It is
the Mukai flop of $S^{[2]} $ along the image of the map $\P^2\hra  S^{[2]}$ induced by the canonical double cover $\pi\colon S\to \P^2$.\ 
 
 The line bundle $L-\delta$ is base-point free on $X_S$ and defines the Lagrangian fibration \mbox{$f\colon  X_S\to \P^{2\vee}$} that takes the class in $X_S$ of a   sheaf on $S$ to its support.\
The line bundle $3\L2-2\delta$ is base-point free and not ample on both $\SS$ and $X_S$
 (\cite[Exercise~3.13]{survey}, \cite[Lemma~2.1.12]{vdd}).\ The \rem{ample line bundle $H=6L-5\delta$ is therefore  also base-point free on $X_S$.}\ It restricts to a general fiber $F=\Pic^2(C)$ of $f$ (where $C\in |L|$) as $L\vert_F$, and this is twice the canonical principal polarization on $F$.\ In particular, the morphism that $H$ defines factors through the involution of~$X_S$ induced by the involution of $S$ attached to $\pi$ and $H$ is not very ample.
\end{rema}

 \begin{rema}\label{rema36}
When $\sigma\in \bw3V_{10}^\vee$ is a general trivector such that the hypersurface $X_\sigma$ is singular, the variety $K_\sigma$ becomes singular, but, with its Pl\"ucker line bundle, birationally isomorphic to $(S^{[2]}, 10\L2-33\delta)$, where $(S,L)$ is a general polarized K3 surface of degree $22$ (\cite[Proposition~3.4]{devo}).\ As indicated in Table~\ref{enumero1} above, the line bundle $10\L2-33\delta$ is on the boundary of the movable cone of $S^{[2]}$; it defines the birational map $S^{[2]}\dra K_\sigma \subset \Gr(6,V_{10})\subset \P(\bw6V_{10})$.\ The corresponding ``periods'' cover the Heegner divisor $ \cD_{22}$.
\end{rema}


\subsection{Vectors of minimal norm and HLS divisors}\label{section35}

\rem{The Heegner divisor $\cD_{2e}$ was defined in Section~\ref{sec21} starting from a primitive $v\in h^\bot$ of negative square.\ The relation between $e$ and $v$ was worked out at the end of the proof of \cite[Proposition~4.1]{dm}:
\begin{itemize}
\item either $11\mid e$, $v^2=-2e/11$, and $v$ has divisibility $1$ in $ h^\bot$;
\item or $11\nmid e$, $v^2=-22e $, and $v$ has divisibility $11$ in $ h^\bot$.
\end{itemize}
The discriminant group $\DD(h^\bot) $ is isomorphic to $\Z/11\Z$.\ In the first case, one has $v_*:=v/\div(v)=0$ in~$\DD(h^\bot)$; in the second case, $v_*$  is $a \in \Z/11\Z$, where $a^2\equiv e\pmod{11}$ (recall that $v$ and $-v$ define the same Heegner divisor).

Let us say that a vector $v\in h^\bot$ with divisibility $>1$ (that is, such that $v_*\ne0$)   and negative square has {\em minimal norm} if $-w^2\ge -v^2$ for all vectors $w\in h^\bot$ with $v_*=w_*$ \rem{and $w^2<0$}.\ For each nonzero class $a \in \Z/11\Z$, one can work out the vectors $v$ with minimal norm such that $v_*=a$ (by Eichler's lemma, they form a single $O(h^\bot)$-orbit, characterized by $a$ and $v^2$).\ We obtain  the following table (if $v$ has minimal norm and $v_*=a$, then $-v$ has minimal norm and $(-v)_*=-a$).

  \begin{table}[htb!]
\renewcommand\arraystretch{1.5}
 \begin{tabular}{|c|c|c|c|c|c|}
 \hline
$a$&$\pm 1$&$\pm2$&$\pm3$&$\pm4$&$\pm5$
 \\
 \hline
$e=-v^2/22$&$1$&$15$&$9$&$5$&$3$
\\
 \hline
 \end{tabular}
\vspace{5mm}
\caption{}\label{enumero2}
\end{table}
 The values of $e$ that appear in this table are exactly those for which we prove that the Heegner divisor    $\cD_{2e}$ is an HLS divisor.\ They are also the five smallest values of $e$ for which a general element of    $\cD_{2e}$ comes from  the Hilbert square of a K3 surface (see Table~\ref{enumero1}).\ Of course, there might be other HLS divisors   which we have not found, but, as mentioned in the introduction, in the case of cubic fourfolds, there is a unique HLS divisor and it corresponds to the unique pair  of orbits of vectors  with minimal norm (the discriminant group is $\Z/3\Z$ in this case); in the case of double EPW sextics, there are three known   HLS divisors, and they correspond to the three orbits of vectors  with minimal norm (the discriminant group is $(\Z/2\Z)^2$ in this case).}

\section{Preliminary results}
We collect  in this section a few results that will be used repeatedly in the rest of the article.

\subsection{Tautological bundles on Hilbert squares}\label{sectaut}

Let $X$ be a smooth projective variety.\ Consider the blow up $\tau\colon \widetilde{X\times X}\lra  {X\times X}$   of the diagonal and its  restriction $\tau_E\colon E\to X$ to its exceptional divisor $E$.\ 
 The  (smooth projective) {\em Hilbert square} of $X$ is  the quotient
$$p\colon \widetilde{X\times X} \lra X^{[2]} $$
 by the lift  $\iota$ of the  involution that exchanges the two factors.\ It is simply ramified along $E$ and there is a class $\delta\in \Pic( X^{[2]} )$ such that $p^*\delta =E$.\ 
We will use the composed maps $q_i\colon\widetilde{X\times X}\xrightarrow{\ \tau\ } X\times X\xrightarrow{\ \pr_i\ }  X$.\  

Let $\cF$ be a vector bundle of rank $r$ on $X$.\ We write $\cF\boxplus\cF:=q_1^*\cF\oplus q_2^*\cF$ and $\cF\boxtimes\cF:=q_1^*\cF\otimes q_2^*\cF$; they are vector bundles on $\widetilde{X\times X} $ of respective ranks $2r$ and $r^2$.\  
  If $\cL$ is an invertible sheaf on $X$, the invertible sheaf $\cL\boxtimes\cL$ is $\iota$-invariant  and descends to an invertible sheaf on $X^{[2]} $ that we still denote by $\cL$.\ This gives an injective group morphism
\begin{equation}\label{notpic}
\Pic(X)\oplus \Z \lra \Pic(X^{[2]}),\qquad (\cL,m)\longmapsto \cL+m\delta.
\end{equation}

The {\em tautological bundle}
\begin{equation*}
\cT_\cF:=p_*(q_1^*\cF)
\end{equation*}
  is locally free of rank $2r$ on $X^{[2]}$ and there is an exact sequence (\cite[prop.~2.3]{dan}, \cite[(3)]{wan})
\begin{equation*}\label{deftaut}
0\to p^*\cT_\cF \to \cF\boxplus\cF \to  \tau_E^*\cF\to 0,
\end{equation*}
of sheaves on $\widetilde{X\times X} $.\ In the notation of~\eqref{notpic}, we have
\begin{equation}\label{dete}
\det(\cT_\cF)= \det(\cF)-r {\delta} 
\end{equation}
and there is an isomorphism
%
%
\begin{equation*}
  H^0(X^{[2]},  \cT_\cF) \isomlra H^0(X,\cF).
\end{equation*}

\begin{rema}\label{rem41}
When $X\subset \P(V)$, there is a morphism $f\colon X^{[2]}\to \Gr(2,V)$ that sends a length-2 subscheme of $X$ to the projective line that it spans in $\P(V) $.\ The rank-2 vector bundle $ \cT_{\cO_X(1)}$ is then the pullback by $f$ of the  tautological subbundle $\cS_2$ on $\Gr(2,V) $.\ It is in particular generated by global sections.
\end{rema}

 We now present  an analogous construction that will be used in Section~\ref{secpf6}.\  There is a surjective morphism
\begin{equation*}\label{eqmorph}
\ev^+\colon \cF \boxtimes\cF\lra \tau_E^*\Sym^2\!\cF
\end{equation*}
obtained by evaluating along the exceptional divisor  $E$ and then projecting onto the symmetric part of
$(\cF \boxtimes\cF)\vert_{ E}=\tau_E^*( \cF\otimes\cF)$.\

\begin{lemm}\label{ledescend}
 There is a locally free sheaf $\cK_\cF$ or rank $r^2$ on $X^{[2]}$ and an exact sequence
 \begin{equation}\label{eqcK}
0\to p^*\cK_\cF \to \cF\boxtimes\cF \xrightarrow{\ \ev^+\ }  \tau_E^*\Sym^2\!\cF\to 0.
\end{equation}
 Moreover, $\det(\cK_\cF)=r\det(\cF)- \frac12 r(r+1)\delta $ and $H^0(X^{[2]} ,\cK_\cF)\isom \bw2H^0(X,\cF)$.
\end{lemm}

 \begin{proof}
 Let $\widetilde\cK_\cF $ be the kernel of $\ev^+$.\ It is locally free on $\widetilde{X\times X}$ and we need to show that it descends to a vector bundle on $X^{[2]}$.\ For that, it is enough to prove that the involution $ {\iota}$ on $\widetilde{X\times X}$ lifts to an involution $\tilde{\iota}$ on $\widetilde\cK_\cF $ that acts by $-\Id$ on $\widetilde\cK_\cF\vert_E $.

 The statement is local over the diagonal of $X$.\ We can thus assume that $\cF$ is trivial on~$X$ with basis $(s_1,\dots,s_r)$ and that we have local coordinates
$x_1,\dots,x_n$ on $X$ near $O\in X$.\  On $X\times X$, we have   coordinates
$x_1,\dots,x_n,y_1,\dots,y_n$ and the bundle $\cF \boxtimes\cF$ has basis
$(s_i\otimes s_j)_{1\le i,j\le r}$, where $(s_i\otimes s_j)(x_1,\dots,x_n,y_1,\dots,y_n)=s_i(x_1,\dots,x_n)s_j( y_1,\dots,y_n)$.\
The   involution $\tilde{\iota}$  on $\cF \boxtimes\cF$ maps $s_i\otimes s_j$ to $s_j\otimes s_i$.\

 Consider a point in $  X^{[2]}$ over $(O,O)$.\ Without loss of generality, we can assume that it corresponds to the tangent vector $\frac{\partial}{\partial x_1}$.\
 At   the corresponding point of  the blow up $\widetilde{X\times X}$, there are then local coordinates 
$\widetilde x_1,\dots,\widetilde x_n, \widetilde y_1,u_2,\dots,u_n$ 
in which the morphism $\tau$ is given by
\begin{equation*}
\tau^*x_i=\widetilde x_i,\ \ \tau^*y_1=\widetilde y_1,\ \  \tau^*(y_i-x_i)=u_i(\widetilde y_1-\widetilde x_1)\ \ \textnormal{for } i\ge 2.
\end{equation*}
The equation of the exceptional divisor $E$ is then   $e:=\widetilde y_1-\widetilde x_1$ and the involution
on
$\widetilde{X\times X}$ is given by
\begin{equation*}
\iota^*\widetilde x_1 =\widetilde y_1 ,\ \   \iota^*\widetilde x_i= \widetilde x_i+u_i(\widetilde y_1-\widetilde x_1),\ \   \iota^*u_i=u_i\ \ \textnormal{for } i\ge 2,
\end{equation*}
and satisfies $\iota^*e=-e$.\
The bundle $\widetilde\cK_\cF $ is thus locally generated by the sections
$$s_i\otimes s_j-s_j\otimes s_i,\ \ e(s_i\otimes s_j+s_j\otimes s_i),$$
for all $i\le  j$.\ This shows that $\tilde{\iota}$ acts by $-\Id$ on $\widetilde\cK_\cF \vert_E$.

The vector bundle $\widetilde\cK_\cF $ therefore descends to a vector bundle $\cK_\cF $ on $X^{[2]}$ whose determinant can be computed from the exact sequence \eqref{eqcK}.

Going back to the global situation, we see that
the space of $\widetilde{\iota}$-antiinvariant sections of $\cF\boxtimes\cF$  on $\widetilde{X\times X}$  that
  are sections of
$\widetilde\cK_\cF$ is $\bw2H^0(X,\cF)$.\ These sections correspond exactly to the sections of
$\cK_\cF$ on $X^{[2]}$.\ This proves the lemma.
\end{proof}

\subsection{Zero-loci of excessive dimensions and excess formula}\label{secexcess}
We describe   in a general context an excess computation that we will use in the 
  proofs of Theorems~\ref{theoprecis},~\ref{theogenre10}, and \ref{theopourgenre6}.\
Let~$M$ be a smooth variety of dimension $n$, let $\cE$ be a vector bundle of rank $r$ on $M$, and let
$\sigma_0$ be a section of $\cE$, with zero-locus $Z\subset M$.\ The differential of $\sigma_0$ defines a morphism $d\sigma_0\colon T_M\vert_Z\to \cE\vert_{Z}$.\ If $Z$ is smooth, of codimension  $s\le r $ in $M$,
the  kernel of $d\sigma_0$ is $T_Z$ and we define the {\em excess bundle $\cF$} to be its cokernel.\ It has rank $r-s$ on $Z$ and  is isomorphic to the quotient $\cE\vert_{Z}/N_{Z/M}$.

Assume now that $\cE$ is generated by global sections and let  $(\sigma_t)_{t\in \Delta}$  be a general $1$-parameter deformation of $\sigma_0$.\ For  $t\in \Delta$ general, the zero-locus $Z_t$ of the section $\sigma_t$ is smooth of pure codimension
$r$  {or empty}.\ The bundle $\cF$, as a quotient of $\cE\vert_Z$, is also generated by its sections and
  the zero-locus of the section
$\overline {\sigma'}  $ defined as the image of $\frac{\partial\sigma_t}{\partial t} \big\vert_{t=0}\in H^0(M,\cE)$ in
$H^0(Z,\cF)$ is smooth of pure codimension $r-s$ in $Z$  {or empty}.

Consider   the closed subset
\begin{equation}\label{eqW}
W=\{(x,t)\in M\times \Delta\mid \sigma_t(x)=0\}.
\end{equation}
The general fibers of the second projection $\pi\colon W\to \Delta$ are smooth    of pure dimension
$n-r$  {or empty,} and the central fiber is $Z$.\  {Let $W^0$ be the union of the   components of $W$ that dominate~$\Delta$ and assume that it is nonempty, hence of pure dimension
$n+1-r$}.\ The central fiber of the restricted map $\pi^0\colon W^0\to \Delta$ is   contained in $Z$.

\begin{prop}\label{proexcess}
For a general $1$-parameter deformation $(\sigma_t)_{t\in \Delta}$, the map $\pi^0\colon W^0\to \Delta$
  is smooth  and its central fiber is the zero-locus of $\overline {\sigma'} $ in $Z$.
\end{prop}

\begin{proof}
 We view the family $(\sigma_t)_{t\in \Delta}$ of sections of $\cE$   as a section $\widetilde\sigma$
of the vector bundle $\widetilde{\cE}:=\pr_M^* \cE$
on $M\times \Delta$, defining $W$.\ We can write $\widetilde\sigma=\widetilde\sigma_0+t\widetilde\sigma'+O(t^2)$ as sections of $\widetilde{\cE}$, where
$\widetilde\sigma_0=\pr_M^*\sigma_0$ and
\begin{equation}\label{eqspointdsigma}
\widetilde \sigma'\vert_{ M\times 0}=\frac{\partial\sigma_t}{\partial t}\Big\vert_{t=0}
.
\end{equation}
Along $Z\times \{0\}\subset W $, we   have
\begin{equation}\label{eqdsigma}
d\widetilde\sigma=d \sigma_0+\widetilde\sigma' dt\colon  T_{M\times\Delta}\vert_{Z\times \{0\}} \lra \widetilde{\cE}\vert_{Z\times \{0\}}.
\end{equation}
Let $z\in Z$ be a point where $\overline {\sigma'}  $ does not vanish.\ We deduce from (\ref{eqspointdsigma}) and (\ref{eqdsigma}) that $Z\times \{0\}$ and~$W$ coincide schematically around $(z,0)$.\ Indeed, as $Z\times \{0\}$ is smooth and contained in $W$, this is equivalent to saying that their Zariski tangent spaces
coincide.\ If they do not, since $Z\times \{0\}$ is the fiber of $W$ at $0$, some tangent vector at $W$ at $0$ is of the type $ (v,\frac{\partial}{\partial t})$.\ By \eqref{eqdsigma}, we have
 $d \sigma_{0,z}(v)+\widetilde\sigma'(z)=0$, so that $\widetilde\sigma'(z)$ belongs to $\Im (d \sigma_{0,z})$.\ By \eqref{eqspointdsigma}, this means that the image $\overline {\sigma'} (z) $ of
$\frac{\partial\sigma_t}{\partial t}\big\vert_{t=0}(z)$ vanishes in $\cF$, contradiction.

We thus proved that
the central fiber of $W^0\to \Delta$ is contained set-theoretically in the zero-locus $Z^0$ of $\overline {\sigma'}  $.\ To prove that the inclusion is scheme-theoretic, we proceed as follows.\
Since $Z\subset M$ is smooth of codimension $s$, we can trivialize $\cE$ locally along $Z$ in such a way
that in the corresponding decomposition $\sigma=(\sigma_1,\ldots,\sigma_r)$, the $s$ first functions have independent differentials, hence define $Z\subset M$.\ We can write
$\widetilde\sigma=(\widetilde\sigma_1,\ldots,\widetilde\sigma_r)$ and replace $M\times \Delta$ by the vanishing locus~$M'$ of
$(\widetilde\sigma_1,\ldots,\widetilde\sigma_s)$ which is smooth of codimension $s$ in $M\times \Delta$  and smooth
over~$\Delta$.\ The central fiber of the restricted map $\pi'\colon M'\to \Delta$ is $Z$ (or rather the relevant open set of~$Z$),
which means that the section $\widetilde\sigma\vert_{ M'}$ vanishes along its central fiber.\ We then have
\begin{equation}\label{eqpourMprime}
 \widetilde\sigma\vert_{ M'}=t\overline{\widetilde\sigma'\vert_{ M'}},
\end{equation}
where $\overline{\widetilde\sigma'\vert_{ M'}}$ is the projection of $\widetilde\sigma'\vert_{ M'}$ onto the
$r-s$ remaining components of $\cE$.\
The decomposition of $W$ into irreducible components is (near the given point of $Z$)
$$W=M'_0\cup \{\widetilde\sigma'\vert_{ M'}=0\},$$
so that $W^0$ is locally  the zero-locus of the section $\overline{\widetilde\sigma'\vert_{ M'}}$.\ Finally, we observe that the restriction to $Z\subset M'$ of the locally defined section
$\overline{\widetilde\sigma'\vert_{ M'}}$ is nothing but $\overline { \sigma'} $.\ As we assumed that
$\overline { \sigma'} $ is general, hence transverse, it follows that
$W^0$ is smooth of codimension $r-s$ in $M'$, with central fiber the zero-locus of $\overline { \sigma'} $.
\end{proof}

\subsection{Geometry of singular trivectors}

Given a nonzero trivector $\sigma\in \bw3V_{10}^\vee$, we   relate singular points on the hypersurface $X_\sigma$  to points on the Debarre--Voisin variety $K_\sigma$ (see   \eqref{eqxs} and~\eqref{eqDV} for  definitions).\ This geometric observation will allow us to describe, for the degenerate trivectors $\sigma_0$ considered in the next sections, the Debarre-Voisin varieties 
(or one of their  irreducible components), as
 Hilbert squares of subvarieties of $\Sing(X_{\sigma_0})$.

  \begin{prop} \label{remaxsing}
  Let $\sigma\in \bw3V_{10}^\vee$ be a nonzero trivector and let $[U_3]$ be a singular point of the hypersurface $X_\sigma\subset \Gr(3,V_{10})$.

  \noindent{\rm{(a)}} The variety $\Sigma_{U_3}:=\{[W_6]\in K_\sigma\mid W_6 \supset U_3\}$ is nonempty of  dimension everywhere at least $  2$ and for all $[W_6]\in \Sigma_{U_3}$, one has $\dim(T_{K_\sigma, [W_6]})>4$.\ In particular, if $K_\sigma$ has (expected) dimension~$4$ at $ [W_6]$, it is singular at that point.

  \noindent{\rm{(b)}} If $[U'_3]$ is another  singular point  of  $X_\sigma$ such that $W_6:=U_3+U'_3$ has dimension~$6$, the point~$ [W_6]$ is in  $K_\sigma$.
 \end{prop}

 \begin{proof}
 Let us prove (a).\ Let $[U_3]\in\Sing(X_\sigma)$ and let $[W_6]\in\Sigma_{U_3}$.\
We will show that the differential~$d\tilde{\sigma}$ of the section $\tilde{\sigma}$ of $\bw3\cE_6$ defining
 $K_\sigma$ does not have maximal rank  at $[W_6]$.
 
 As explained in the proof of  \cite[Proposition~3.1]{devo}, this differential 
 $$d\tilde{\sigma}\colon  T_{\Gr(6,V_{10}),[W_6]}=\Hom(W_6,V_{10}/W_6)\lra \bw3W_6^\vee$$ 
  maps $u\in \Hom(W_6,V_{10}/W_6)$ to the 3-form
 $$d\tilde{\sigma}(u)(  w_1, w_2, w_3 )=  \sigma(u(w_1), w_2, w_3)+\sigma(w_1, u(w_2), w_3)+ \sigma(w_1, w_2, u(w_3)).
$$
Since $[U_3]$ is  singular   on $X_\sigma$, the trivector $\sigma$ vanishes on   $\bw2U_3\wedge V_{10}$   (\cite[Proposition~3.1]{devo}), hence~$d\tilde{\sigma}(u)$ vanishes on $\bw3U_3$.\ The composite
\begin{equation}
\label{eqdiff} 
 \Hom(W_6,V_{10}/W_6)\xrightarrow{\ d\tilde{\sigma} \ } \bw3W_6^\vee
\to \bw3U_3^\vee
\end{equation}
is therefore zero, hence $d\tilde{\sigma}$ does not have maximal rank.

It remains to prove that the variety $\Sigma_{U_3} $  is nonempty of dimension everywhere  $\ge 2$.\ This follows from the fact that it is defined in the smooth $12$-dimensional variety $$ \{[W_6]\in \Gr(6,V_{10})\mid W_6 \supset U_3\}\isom \Gr(3, V_{10}/U_3)$$ as the zero-locus of a section of the 
rank-$10$ vector bundle $(U_3^\vee\otimes \bw2\cE_3)\oplus \bw3\cE_3$, whose top Chern class is nonzero.

Let us prove (b).\ Since $[U_3]$ and $[U'_3]$ are singular points of  $X_\sigma$, the trivector $\sigma$ vanishes on $\bw2U_3\wedge V_{10}$ and   
 $\bw2U'_3\wedge V_{10}$, hence also on $\bw3(U_3+U'_3)$.\ In particular, if $U_3+U'_3$ has dimension~$6$, it defines a point of $K_\sigma$.
\end{proof}

The proof above also gives  the following information which will be useful when we compute the excess bundles of Section \ref{secexcess} in our specific situations.

\begin{lemm} \label{lenewpourexces}
 In Proposition
\ref{remaxsing}{{(a)}}, the restriction map
$\bw3 W_6^\vee\thra \bw3 U_3^\vee$ vanishes on $\Im (d\tilde{\sigma})$.

In 
 Proposition
\ref{remaxsing}{{(b)}}, the restriction map
$\bw3 W_6^\vee\thra \bw3 U_3^\vee\oplus \bw3 U_3^{\prime\vee}$ vanishes on $\Im (d\tilde{\sigma})$.
\end{lemm}


\begin{rema}\label{re45}
In Sections~\ref{sec51} and~\ref{se62}, we will work with a generically smooth   component $K_0$ of a Debarre--Voisin variety $K_{\sigma_0}$  of excessive dimension $6$,  so that the image of $d\tilde{\sigma}_0$ has codimension~$2$ along its smooth locus.\   \rem{In each case, we will see that} a general point of  $K_0$ is   of the form
$[U_3\oplus U'_3]$, with   $[U_3],\,[U'_3]$ in some smooth subvariety $W$ of $ \Sing (X_{\sigma_0})$, \rem{so that} there is a  rational dominant map
\begin{eqnarray*}
f\colon W^{[2]}&\dra &K_0\\
([U_3],\,[U'_3])&\longmapsto &[U_3+ U'_3] 
\end{eqnarray*}
\rem{(see Proposition~\ref{remaxsing}(b)).}\ Lemma \ref{lenewpourexces} then
tells us  that the image of $d\tilde{\sigma}_0$ vanishes in the two-dimensional space $\bw3 U_3^\vee\oplus \bw3 U_3^{\prime\vee}$.\ This identifies,  on
a Zariski open subset of $W^{[2]}$, the pullback by $f$ of the excess bundle on $K_0$ 
 with the tautological bundle $\cT_{\cO_W(1)}$, where $\cO_W(1)$ is the Pl\"ucker line bundle
on $W\subset \Gr(3,V_{10})$.\ By Remark~\ref{rem41}, it is generated by its global sections.
\end{rema}


\section{The HLS divisors  $\cD_{6}$ and $\cD_{18}$}
\label{sechasset}

 We   describe in this section two  polystable  (semistable with closed orbit in the semistable locus) trivectors  in the moduli space
$\cMDV=\PP(\bw3  V_{10}^\dual) \gquot \SL(V_{10})
$ whose total images\footnote{\rem{The {\em total image} of a point $p\in X$ by a rational map $f\colon X\dra Y$ is  the projection in $Y$ of the inverse image of $p$ in $\Gamma$, where $\Gamma\subset X\times Y$ is the (closure) of the graph of $f$.}} by the moduli map
$$\gm\colon \cMDV\dra \cM$$
 are the hypersurfaces in $\cM$  whose general points are pairs $(S^{[2]},  2\L2- \delta)$, where $(S,L)$ is a general  polarized K3 surface of degree 6 (resp.~pairs $(S^{[2]},  2\L2- 5\delta)$, where $(S,L)$ is a general polarized K3 surface of degree~$18$) (see Table~\ref{enumero1}).\ As explained in Section~\ref{secrappelonpic}, their total images
by the composition
$$\gp\circ \gm\colon \cMDV\dra \cF$$ 
are therefore the Heegner divisors $\cD_6$ (resp.~$\cD_{18}$).\
A common feature of these two specific trivectors $\sigma_0$, which makes the specialization analysis
quite easy, is that  the corresponding  Debarre--Voisin varieties $K_{\sigma_0}$ are smooth but  of larger-than-expected dimension.\ The limit
of the Debarre--Voisin varieties along a $1$-parameter degeneration to  ${\sigma_0}$ is then
a smooth fourfold obtained as the zero-locus of a general  section of the excess bundle on~$K_{\sigma_0}$ associated with this  situation (see Section~\ref{secexcess}).

\subsection{The HLS divisor   $\cD_{6}$}\label{sec51}
We construct a   trivector $\sigma_0 $ whose 
Debarre--Voisin variety $K_{\sigma_0}$ is smooth but  has excessive dimension~$6$.\
The neutral component of the stabilizer of $\sigma_0 $  is 
$  \Sp(4)$  and   the point $[\sigma_0]$ of $\P(\bw3V_{10}^\vee)$ is polystable for the $\SL(V_{10})$-action (Proposition~\ref{coronouveau}).\ The total image  in  $\cF$ of the point $[\sigma_0]$  is the Heegner divisor $\cD_6$.\   The main result of this section is Theorem~\ref{theoprecis}.\

\subsubsection{The $\Sp(4)$-invariant trivector}\label{sechivert}

Let $V_4$ be a $4$-dimensional vector space equipped with a symplectic form $\omega$ and let
$V_5\subset \bw2V_4 $ be the hyperplane  defined by $\omega$, endowed with the nondegenerate quadratic form  $q$ defined by $q(x,y)=(\omega \wedge\omega)(x\wedge y)$.\ The form $q$ defines a smooth quadric $Q_3\subset \P(V_5)$.

The $10$-dimensional vector space 
$V_{10}:=\bw2V_5\cong \Sym^2\!V_4$
 can be identified with the space of   endomorphisms of $V_5$ which are  skew-symmetric  with respect to $q$ and we define
 a trivector $\sigma_0$ on~$ V_{10}$ as in \eqref{si0} by
$\sigma_0(a,b,c)=\Tr(a\circ b\circ c)
$.\ It is  invariant for the canonical action of the group $\Sp(V_4,\omega) =\SO(V_5,q)$ on $\bw3V_{10}^\dual $.\

This is a particular case of a general situation studied by Hivert, who proved in particular that the Debarre--Voisin variety
 $K_{\sigma_0}$
   is  smooth  of  dimension 6 (\cite[Definition~1.2 and Theorem~4.1]{hivert}).\ He moreover gave a very concrete description of this variety.\ We will use the   hypersurface $X_{\sigma_0}\subset \Gr(3,V_{10})$ defined in \eqref{eqxs}.


\begin{prop} \label{le11}
{\rm{(a)}} The image of the morphism
 \begin{eqnarray*}
j\colon \Qtrois&\lra& \Gr(3,V_{10})\\
x&\longmapsto&[ x\wedge x^{\perp_q}]
\end{eqnarray*}
 is contained in the singular locus of the hypersurface $X_{\sigma_0}\subset \Gr(3,V_{10})$.
 
\noindent{\rm{(b)}} The morphism $j$ is an embedding and $j^*\cO_{\Gr(3,V_{10})}(1)\isom \cO_{\Qtrois}(3)$.
 \end{prop} 

\begin{proof} Let $x\in \Qtrois$.\  If $z\in x^{\perp_q}$, the  skew-symmetric endomorphism $a_z $ of $V_5$ associated with $x\wedge z$  is
$$\forall u\in V_5\qquad  a_z(u)=q(x,u)z-q(z,u)x,$$
and thus, if $z,z'\in x^{\perp_q}$, we have
 \begin{eqnarray*}
 a_{z'}\circ a_z (u)&=&q(x,u)q(x,z)z'-q(z,u)q(x,x)z'-q(x,u)q(z',z )x+q(z,u)q(z',x)x\\
&=&-q(x,u)q(z',z )x,
 \end{eqnarray*}
 which is symmetric in $z$ and $z'$,
proving that $ a_z$ and $ a_{z'}$ commute.\ The endomorphism
$ a_{z'}\circ a_z$ is then symmetric, hence $\Tr( a_{z'}\circ a_z\circ c)=0$ for any skew-symmetric endomorphism $c\in V_{10}$.\ 
 By \cite[Proposition~3.1]{devo}, this implies item (a). 

We now prove (b).\ The morphism $j$ is injective because $x^{\perp_q}$ is the tangent space to $Q_3$ at~$[x]$ and this hyperplane is tangent only at $[x]$.\ Since $j$ is $O(V_5,q)$-equivariant, it is an embedding.\ Consider now the exact sequence 
 \begin{equation*}\label{eqprem2611} 
 0 \to   \cK \to   V_5\otimes \cO_{\Qtrois}\xrightarrow{\ q\ }  \cO_{\Qtrois}(1)\to 0
 \end{equation*}
 defining the rank-$4$ vector bundle $\cK \isom\Omega_{\P(V_5)}(1)\vert_\Qtrois$ with fiber $x^{\perp_q}$ at $[x]$ and the exact sequence 
 \begin{equation*}\label{eqdeux2611} 
 0\to  \cO_{\Qtrois}(-2) \to \cK\otimes  \cO_{\Qtrois}(-1)\xrightarrow{\ \wedge\ }  j^*\cS_3\to 0
,\end{equation*}
which implies $j^*\cS_3\isom \Omega_\Qtrois$.\ We obtain the desired isomorphism $j^*\cO_{\Gr(3,V_{10})}(1)\isom \cO_{\Qtrois}(3)$ by taking determinants.
\end{proof}

By Propositions \ref{remaxsing} and \ref{le11}, we have
a
 rational map $f\colon \hilb\dra K_{\sigma_0}$ which is    $\Sp(4)$-equivariant.\ The following result is \cite[Theorem~6.3]{hivert}.

\begin{theo}[Hivert]  \label{thoehivertDVtzero}
The  map $f\colon \hilb\to K_{\sigma_0}$ is  an isomorphism.
\end{theo}

\begin{proof}
Any point in $\hilb$ spans  a line in $\P(V_5)$, hence defines an element of $\Gr(2,V_5)$.\ The corresponding morphism $\eps\colon\hilb\to \Gr(2,V_5)$  has a rational inverse:   the intersection of a line in~$\P(V_5)$ with $\Qtrois$ is  a subscheme of length $2$ of $\Qtrois$, except when
  the line is
   contained in~$\Qtrois$.\ The morphism~$\eps$ is therefore the blow up of the scheme of lines contained in~$\Qtrois$ (which is the
  image of the Veronese embedding
 $v_2\colon \P(V_4)\hra   \P(\Sym^2\!V_4)=\P(\bw2V_5)
$; see  \cite[Section~6.2]{hivert}).

%
Hivert moreover proved that the linear system $|\cI_{v_2(\P(V_4) )}(3)|$  embeds $\hilb$ into
the  linear
  span of $K_{\sigma_0}$ in the Pl\"ucker embedding of $\Gr(6,V_{10} )$ and that
its image coincides with~$K_{\sigma_0}$.
\end{proof}

\subsubsection{Orbit and stabilizer}\label{orbetstab}

The decomposition of $\bw3V_{10}^\vee$ into irreducible $\Sp(4)$-representations
is  
\begin{equation} \label{eqhivert}
\bw3V_{10}^\vee=V_{4\omega_1}\oplus V_{3\omega_2}\oplus V_{2\omega_1+\omega_2}\oplus V_{2\omega_2}\oplus V_{\omega_2}\oplus \C,
\end{equation}
\rem{where $V_{a_1\omega_1+a_2\omega_2}$ denotes the irreducible representation of $\Sp(4)$ with highest weight $a_1\omega_1+a_2\omega_2$, where  $\omega_1$ and $\omega_2$ are the fundamental weights (\cite[Section 6.2]{hivert}, \cite{bbki})}.\
 The last term is the space of $\Sp(4)$-invariants; it is generated by our trivector
 $\sigma_0$ defined in \eqref{si0}.\
The first  term   is  $ \Sym^4\!V_4$ and the second term is  $ H^0(\Qtrois,\cO_{\Qtrois}(3))$.\ Since $\gsp (4)=\Sym^2\!V_4=V_{2\omega_1}$ and
$$
\End(V_{10})=V_{4\omega_1}\oplus V_{2\omega_1}\oplus V_{2\omega_1+\omega_2}\oplus V_{2\omega_2}\oplus V_{\omega_2}\oplus \C
,$$  there is an exact sequence
\begin{equation*}
0\to \gsp (4)\to \End(V_{10})\to \bw3V_{10}^\vee \to H^0(\Qtrois,\cO_{\Qtrois}(3))\to 0.
\end{equation*}
We prove that the tangent space to the stabilizer   of $\sigma_0$ is  $\gsp (4)$, hence  the normal
space to the $\GL(V_{10})$-orbit of $\sigma_0$
is
$   H^0(\Qtrois,\cO_{\Qtrois}(3))$.

\begin{prop}\label{coronouveau}
The neutral component of the stabilizer of $\sigma_0$ for the $\SL(V_{10})$-action is $\Sp(V_4)=\SO(V_5)$ and the point $[\sigma_0]$ of $\P(\bw3V_{10}^\vee)$ is polystable for the $\SL(V_{10})$-action.
\end{prop}

\begin{proof}
The neutral component of the stabilizer acts on the Debarre--Voisin variety $K_{\sigma_0}$, which is isomorphic to $\hilb$.\ Since it is connected, it acts trivially on the N\'eron--Severi group, hence preserves the exceptional divisor of the Hilbert--Chow morphism $\hilb\to \chow$.\ It therefore acts on $\chow$, hence on $\Qtrois$.\ It is therefore in $\SO(V_5)$.

To show that $[\sigma_0]$  is polystable, we will use a result of Luna.\ By Proposition~\ref{normal} below, the stabilizer $\SO(V_5)$ has finite index in its normalizer in 
 $\SL(V_{10})$.\ By
 \cite[Corollaire~3]{luna} (applied to  the group $ \SL(V_{10})$ acting on $\bw3V_{10}^\vee $), the orbit of $\sigma_0$ is closed in $\bw3V_{10}^\vee $, hence $[\sigma_0]$  is polystable.
\end{proof}

We   prove the    classical result used in the proof above.

\begin{prop}\label{normal}
Let $G$ be a semisimple algebraic group with a faithful irreducible representation $G\hra \SL(V)$.\ The group $G$ has finite index in its normalizer  in $\SL(V)$.
\end{prop}

\begin{proof}
According to the discussion after \cite[Lemma~16.3.8]{spring}, the group of outer automorphisms of~$G$ is finite.\ The kernel of the   \rem{action} $N:=N_{\SL(V)}(G)\to \Aut(G)$ of the normalizer by conjugation is contained in the centralizer $C:=C_{\SL(V)}(G)$ and the kernel of the induced morphism \mbox{$N/G\to \Out(G)$} is contained in the image of $C$ in $N/G$.\ It is therefore sufficient to show that~$C$ is a finite group.\ But this follows from Schur's lemma: any eigenspace of an element of $C$ is stable by $G$, hence equal to~$V$.\ Therefore, $C$ consists of homotheties, hence is finite.
\end{proof}

\subsubsection{Degenerations and excess bundles}\label{seconepartohivert}

Consider a general $1$-parameter deformation $(\sigma_t)_{t\in \Delta}$.\ The
  derivative $ \frac{\partial\sigma_t}{\partial t}\big\vert_{t=0}$   provides, by the discussion in Section~\ref{orbetstab}, a general section of
$\cO_{\Qtrois}(3)$
which defines a \rem{general}
K3 surface $S\subset \Qtrois\subset \P(V_5)$ of degree 6.\

\begin{theo}\label{theoprecis}
Let $(\sigma_t)_{t\in \Delta}$ be a general $1$-parameter deformation.\
Let $\cK\to\Delta$ be the associated family of Debarre--Voisin varieties and let $\cK^0$ be the irreducible component of $\cK$ that dominates~$\Delta$.\
Then $\cK^0\to \Delta$
is smooth and it central fiber   is isomorphic to $S^{[2]}$, embedded in~$\Gr(6,10)$ as $S^{[2]}\subset Q_3^{[2]}\isom K_{\sigma_0}\subset \Gr(6,V_{10})$, where
$S$ is a general    $K3$ surface of degree 6.
\end{theo}

The proof of the theorem will be based on the excess computation presented in Section~\ref{secexcess}:  we want to
  apply  Proposition \ref{proexcess} with $M=\Gr (6,\bw2 V_5)$ and $\cE=\bw3 \cE_6$, where $\cE_6$ is the dual of the tautological rank-6 subbundle on  $\Gr (6,\bw2 V_5)$.\
For this, we   need  to identify the rank-$2$ excess bundle~$\cF$ on $K_{\sigma_0}\isom\hilb$.\   We use the notation of Section~\ref{sectaut}.\

\begin{prop}\label{Festtauto3}
The excess bundle $\cF$ on $\hilb$ is  isomorphic to the tautological
bundle $\cT_{\cO_{\Qtrois}(3)}$.
\end{prop}

\begin{proof}
By definition,  $\cF$ is a rank $2$-quotient bundle of $\bw3\cE_6\vert_{ \hilb}$, hence of $\bw3V_{10}^\vee\otimes \cO_{\hilb}$.

Since $j$ is an embedding (Proposition~\ref{le11}), the      rank-$2$ vector bundle $\cT_{\cO_{\Qtrois}(3)}$   is generated by the space
  $
\bw3 V_{10}^\vee$ of   global sections by Remark~\ref{re45}.\ More precisely,   on the dense open set $U\subset \hilb$ of pairs
  $\{x,\,y\}$ such that $(x\wedge  x^{\perp_q} )\cap( y\wedge y^{\perp_q})=\{0\}$, the evaluation map
\begin{equation}\label{eq17}
\bw3V_{10}^\vee\otimes \cO_{Q_3^{[2]}}\lra  \cT_{L^{\otimes 3}} 
\end{equation}
  factors through the   composite  map
\begin{equation}\label{eq18}
\bw3V_{10}^\vee\otimes \cO_{\hilb}\to \bw3\cE_6\vert_{ \hilb}\to \cF.
\end{equation}
 The   bundles $\cF$ and $\cT_{L^{\otimes 3}}$   therefore coincide as  quotients of $\bw3V_{10}^\vee\otimes \cO_{Q_3^{[2]}}$:  the morphisms $\hilb\to \Gr(2,\bw3 V_{10}^\dual)$ that they define coincide on the dense   set $U$, hence they are the same.
 \end{proof}

\begin{proof}[Proof of Theorem \ref{theoprecis}]
We apply Proposition \ref{proexcess}:
by Theorem \ref{thoehivertDVtzero},
the locus $Z=K_{\sigma_0}$ is smooth of codimension $18$ in $M$, isomorphic to
$\hilb$, and, by Proposition   \ref{Festtauto3}, the rank-$2$ excess
bundle~$\cF $  on $\hilb$ is isomorphic to
$\cT_{\cO_{Q_3}(3)}$.\ The $5$-dimensional variety $\cK^0$ is therefore smooth
with  fiber over $0$ the smooth zero-locus of the
section $\overline {\sigma'} $ of $\cF$.\

More precisely, the proof of
Proposition \ref{Festtauto3} shows that the  composite
map \eqref{eq18} can be
  identified with the map  \eqref{eq17}
 induced by
the (composed) evaluation map
$$\bw3V_{10}^\vee\otimes \cO_{\Qtrois}\xrightarrow{\ a\ } H^0(\Qtrois,\cO_{Q_3}(3))\otimes \cO_{\Qtrois}\to \cO_{\Qtrois}(3).$$
The derivative $\frac{\partial\sigma_t}{\partial t}\big\vert_{t=0} $ provides via the surjective map $a$  a section
   of $ \cO_{\Qtrois}(3) $ that defines a  general  $K3$ surface $S\subset \Qtrois$ of degree $10$ and  the zero-locus of $\overline {\sigma'} $ can be identified with $S^{[2]}\subset Q_3^{[2]}$.
\end{proof}

\subsection{The HLS divisor $\cD_{18}$}\label{secpf10}
We now construct a trivector $\sigma_0$  whose 
Debarre--Voisin variety $K_{\sigma_0}$ is smooth but  has excessive dimension~$10$ (Corollary \ref{corokalphabeta}).\ The space $V_{10}$ decomposes as $V_7\oplus W_3$ and 
  $\sigma_0$ as $\alpha+\beta$, with $\alpha\in \bw3V_7^\vee$ and $\beta\in \bw3W_3^\vee$.\ For   the $\SL(V_{10})$-action, the point  $[\sigma_0]$ of $\P(\bw3V_{10}^\vee)$ has stabilizer
$G_2 {\times} \SL(3)$ and  is polystable  (Corollary \ref{lestabgenus10}).\  The main result of this section is Theorem~\ref{theogenre10}.

\subsubsection{$K3$ surfaces of degree $18$}\label{secgenre10}
A general polarized $K3$ surface $(S,L)$ of degree $18$ carries a \rem{unique} rank-$2$ Lazarsfeld--Mukai
bundle $\cE_2$ \rem{(that is, stable and rigid)} that satisfies  $\det(\cE_2)=L$ and $c_2(\cE_2)=6$.\ The vector space
$V_7:=H^0(S,\cE_2 )^\vee$ has dimension $7$, the sections of $\cE_2 $ embed $S$ into $\Gr(2,V_7)$, and
via this embedding, $S$
can be described as follows (\cite{mukai}).

 Let $\alpha\in \bw3V_7^\dual $ be general.\ The 7-dimensional space $I_X\subset \bw2V_7^\dual $ of
Pl\"ucker linear sections given by $u\,\lrcorner\, \alpha$, for $u\in V_7$,
cuts out a smooth fivefold $X\subset \Gr(2,V_7)$.\ We have  $K_X=\cO_X(-3)$ and
one gets a general $K3$ surface $S$ of degree $18$ by intersecting $X$ with a projective space $\P(W_3^\bot)$ cut out by three extra general Pl\"ucker linear sections.\ The subspace $I_S=I_X\oplus W_3\subset\bw2V_7^\dual $ of Pl\"ucker linear sections vanishing on $S$ has dimension $10$.\

Recall from Section \ref{secrappelonpic} that we are looking for a rank-$6$ vector bundle $\cS_6$ with determinant $-2\L2+5\delta$ on $S^{[2]}$, in order to embed $S^{[2]}$ in a Debarre--Voisin variety in $\Gr(6,10)$.\ We will construct it as a direct sum
$$
\cS_6=\cS_4\oplus \cS_2.$$
We first construct the vector bundle 
$\cS_4$  as follows.\
The  surjective evaluation map
 $V_7^\vee\otimes\cO_S\thra {\cE_2}$ 
  induces, with the notation of Section~\ref{sectaut}, a surjective evaluation map
\begin{equation*}
\label{eqevpourtautonS}
 \ev\colon V_7^\vee\otimes\cO_{S^{[2]}}\thra\cT_{ \cE_2}.
\end{equation*}
  Indeed, 
    the nonsurjectivity of $\ev$ at a point $([V_2],[V_2'])$ of $S^{[2]}$
   means that the subspace $V_3:=\langle V_2,V'_2\rangle$ of
 $ V_7$ has dimension $3$.\ Then, $S\cap \Gr(2,V_3)$ contains a subscheme of length $2$.\ Since $S$ is defined by linear Pl\"ucker equations in $\Gr(2,7)$, it contains a line, which contradicts the fact that it is general.

Set 
\begin{equation}\label{s4}
\cS_4:=\cT_{\cE_2}^\vee \subset  V_7\otimes\cO_{S^{[2]}}.
\end{equation}
The following lemma  will be used later on.

\begin{lemm}\label{lepourtrivecteur}
The morphism $S^{[2]}\to \Gr(4,V_{7})$ associated with
the bundle $\cS_4$ takes value in the set of $4$-dimensional vector subspaces that are totally isotropic for the $3$-form
$\alpha$ on $V_7$.
\end{lemm}

\begin{proof}  It is enough to check the conclusion at  a general point $([V_2],[V_2'])$ of $S^{[2]}$.\ Then $V_2$ and $ V_2'$ are  transverse vector subspaces
of $V_7$ which belong to $X$, hence satisfy $(\bw2V_2 )\lrcorner\,\alpha=(\bw2V_2' )\lrcorner\,\alpha=0$ in $V_7^\vee$.\
The space $V_4:=\langle V_2,V_2'\rangle\subset V_7$ is the fiber of $\cS_4$ at $([V_2],[V_2'])$.\ The restriction $\alpha':=\alpha\vert_{ V_4}$
is a $3$-form which is either decomposable with one-dimensional  kernel or $0$.\ If it is nonzero, all the elements
$[U_2]\in \Gr(2,V_4)$ that satisfy $U_2\,\lrcorner\,\alpha'=0$ must contain the kernel of $\alpha'$ and this contradicts the equality
$V_2\cap V_2'=\{0\}$.
\end{proof}

Turning to the construction of
$\cS_2$, we now show the following.

\begin{lemm} \label{le510}
 Let $z$ be  a point  of $S^{[2]}$ and set
$V_4:=\cS_{4,z}\subset V_7$.\ Consider the   composition 
$$r_z\colon  I_S\hra \bw2V_7^\dual\to \bw2V_4^\dual.$$
Then,
\begin{itemize}
\item [{\rm (a)}] the kernel of $r_z$ intersects $I_X$ along a $4$-dimensional vector
space;
\item [{\rm (b)}]  the  map
$r_z$ has rank $4${\rm;}
\item [{\rm (c)}]  the  cokernel of
$r_z$ can be identified with the fiber $\cT_{L,z}$.
\end{itemize}
 \end{lemm}

\begin{proof} We know from the proof of Lemma \ref{lepourtrivecteur}
that  $\alpha\vert_{V_4}=0$, which implies that the $2$-forms
$u\,\lrcorner\,\alpha$, for $u\in V_4$,  vanish on $V_4$.\ They all belong to $I_X$, so we have
$\dim (\Ker (r_z)\cap I_X)\geq 4$.\ If the inequality is strict,  there is a $5$-dimensional subspace
$V_5$ of~$V_7$, containing $V_4$ such that $u\,\lrcorner\,\alpha$ vanishes on $V_4$ for $u\in V_5$.\ But  
$\alpha$ then vanishes identically on~$V_5$, which contradicts the fact that $\alpha\in \bw3 V_7^\dual$ is general so has no
$5$-dimensional totally  isotropic subspace.\ This proves (a).

  Turning to  the proof of (b) and (c), the image of $r_z$ is contained in the space of  sections of the Pl\"ucker line bundle on $\Gr(2,V_4)$ vanishing on the length-$2$ subscheme $z$, and this space is $4$-dimensional.\ It remains to see that the rank of $r_z$ is at least $4$.\
By (a), the restriction of $r_z$ to $I_X\subset I_S$
has rank $3$.\  The image  $r_z(I_X)$ defines a conic in
$\Gr(2,V_4 )\subset \Gr(2,V_7 )$  which is contained in $X$ by definition.\
If  $r_z$ has rank only $3$,
this conic is contained in $S$, which contradicts the fact that $S$ is general.\
\end{proof}

By Lemma~\ref{le510}, we have an exact  sequence
\begin{equation}\label{eq20}
 0\to \cS'_6\to I_S\otimes \cO_{S^{[2]}}\xrightarrow{\ r\ }  \bw2\cS_4^\vee
 \to \cT_{L}\to 0
\end{equation}
 of vector bundles on $S^{[2]}$.\ The rank-$6$ vector bundle $\cS'_6$ that it defines 
 contains the rank-$4$  bundle
$\cS_4\subset I_X\otimes \cO_{S^{[2]}}$ (see~\eqref{s4}) and we thus get a rank-$2$
 bundle
 $$\cS_2:=\cS'_6/\cS_4\subset W_3\otimes \cO_{S^{[2]}}.$$

\begin{lemm}  \label{lepourrangdeuxpart}
  The  vector bundle
  $\cS_2$  has  determinant $-\L2+3\delta$, the vector bundle
  $\cS_4$ has  determinant $- \L2+2\delta$, and the vector bundle
  $\cS_6'$ has  determinant $-2\L2+5\delta$.
  \end{lemm}

\begin{proof}  By \eqref{dete}, the determinant of $\cS_4^\vee=\cT_{\cE_2}$ equals $\L2-2\delta$, hence $\det (\bw2\cS_4^\vee)=3\L2-6\delta$, while $\det (\cT_{L})=\L2-\delta$.\ Together with the  exact sequence \eqref{eq20}, this implies
\begin{equation}\label{dets6}
\det (\cS_6')=\L2-\delta-(3\L2-6\delta)=-2\L2+5\delta.
\end{equation}
We then get 
$$\det (\cS_2)=\det(\cS'_6)-\det(\cS_4)=-2\L2+5\delta -(-\L2+2\delta)=-\L2+3\delta,$$
which proves the lemma.
\end{proof}

Set    $\cS_6:=\cS_4\oplus \cS_2$.\ It is a subbundle of the trivial rank-$10$ bundle on $\SS$ with fiber $I_X\oplus W_3$, and this defines a morphism
\begin{equation}\label{defphi}
\phi=(\phi_1,\phi_2)\colon  S^{[2]}\lra \Gr(4,V_7)\times \Gr(2,W_3)\subset \Gr(6,V_7\oplus W_3).
\end{equation}

\begin{lemm} \label{lem513}
 If the surface $S$ is general,  {the morphism  $\phi$ is injective} and the Pl\"ucker line bundle restricts to $2\L2-5\delta$ on $S^{[2]}$.
\end{lemm}

\begin{proof} It suffices to show that the first component $\phi_1$ of $\phi$ is injective.\
Let $z\in\SS$ and let $ [V_4]:=\phi_1(z)=\cS_{4,z} \subset V_7$.\
 As we  saw in the proof of Lemma~\ref{le510}, the data   $V_4\subset V_7$ determine  a (possibly singular) conic $C$ in $ \Gr(2,V_4)\subset X$ and
 the image of the map $I_S\to H^0(C,\cO_C(2))$ has rank at least~$1$, as otherwise the rank
 of the map $I_S\to \bw2V_4^\vee$ would be only $3$.\  A nonzero linear form
 on a conic vanishes on a line contained in the conic or along a subscheme of length~$2$.\ Since a general~$S$  contains no lines, there is at most one length-$2$  subscheme   
 of $S$ on this conic.
 
 The pullback of the Pl\"ucker line bundle   to $S^{[2]}$ was computed in Lemma~\ref{lepourrangdeuxpart}.
\end{proof}

 {We will see in Proposition~\ref{propidentexcessgenre10} that $\phi$ is actually an embedding.}

The tautological quotient bundle on the Grassmannian $\Gr(6,V_7\oplus W_3)$ pulls back via $\phi$ to
a rank-$4$ vector bundle 
 on
$S^{[2]}$  generated by $10$ sections and with   determinant
$2\L2-5\delta$ (Lemma~\ref{lepourrangdeuxpart}).\  

\subsubsection{The $G_2\times \SL(3)$-invariant trivector}\label{sec422}
We let $V_{10}:=V_7\oplus W_3$ and we take as before $\alpha\in \bw3V_7^\dual $   general.\ 
If $\beta $ is a generator of
$\bw3W_3^\vee$, we let $\sigma_0:=\alpha+\beta$.

If $S$ is a K3 surface as above, the image  $\phi(S^{ [2]})$ (see \eqref{defphi}) is, by Lemma \ref{lepourtrivecteur} \rem{and the fact that any 2-dimensional subspace of $W_3$ is totally isotropic for~$\beta$)}, contained in  the Debarre--Voisin variety $K_{\sigma_0}$.\ We first determine
this  variety.

%
%

\begin{prop}\label{proisotsubspace}
Let $V_{10} $ and $\sigma_0=\alpha+\beta$ be as above.\
Any $6$-dimensional subspace $W_6\subset V_{10} $ which is totally isotropic for $\sigma_0$ is of the form
$W_4\oplus W_2$, where $W_4\subset V_7 $ is totally isotropic for $\alpha$ and
$W_2\subset W_3$ is of dimension $2$ (hence totally isotropic for~$\beta$).

Conversely, any such space is  totally isotropic for $\sigma_0$.
\end{prop}

\begin{proof} Denote by $p_1\colon  W_6\to V_7 $ and $p_2\colon  W_6\to W_3$ the two projections.\
We first claim   that $\rank (p_1)\le5$.\ Indeed, on $W_6$, we have
$p_1^*\alpha=p_2^*\beta$ and, as $\beta$ is decomposable, $p_2^*\beta$ vanishes on a hyperplane of $W_6$.\ But
$\alpha$ does not vanish on any $5$-dimensional subspace of $V_7$, which shows that~$p_1$ must have a nontrivial kernel.\

We next claim that $p_1$ cannot have rank $5$.\ Indeed, if it does,  $p_1^*\alpha$ is nonzero, so $p_2^*\beta$ is nonzero.\ But   the  kernel of $p_2^*\beta$ is then $\Ker (p_2)$ and it must be equal to  the kernel of $p_1^*\alpha$, that is,
$p_1^{-1}(\Ker (\alpha\vert_{\Im (p_1)}))$.\ As $p_1$ has rank $\leq 5$, it  follows that there is a
nonzero $u$ in $\Ker (p_1)\cap \Ker (p_2)$, which is absurd.\ From these two facts, we conclude that $p_1$ has rank at most~$4$.\ A similar argument  shows that $p_2$ has rank
at most $2$, that is, $p_2^*\beta=0$, and thus $p_1^*\alpha=0$, that is, $\alpha\vert_{\Im (p_1)}=0$.\ Finally, as $W_6\subset p_1(W_6)+p_2(W_6)$, we conclude that we must have equality.
\end{proof}

\begin{coro}\label{corokalphabeta}
The Debarre--Voisin variety
$K_{\sigma_0}$ is smooth of dimension $10$ and splits as a product $K'_\alpha\times \P(W_3^\vee)$.
\end{coro}

\begin{proof} Let  $K'_{\alpha}\subset \Gr(4,V_7)$ be  the variety of  subspaces
$V_4\subset V_7$ that are totally isotropic for $\alpha$.\ It is the zero-locus  of a general section of the globally generated, rank-$4$, bundle $\bw3\cE_4$, hence it  is smooth of dimension $8$.\ Finally, Proposition
\ref{proisotsubspace} implies $K_{\sigma_0}\cong K'_\alpha\times \P(W_3^\vee)$.
\end{proof}

\subsubsection{Stabilizer}\label{orbetstab2}

The computation of the stabilizer of our trivector $\sigma_0$  is a consequence of Proposition~\ref{proisotsubspace}.

\begin{coro}\label{lestabgenus10} 
The stabilizer of the trivector $\sigma_0 \rem{{}=\alpha+\beta}$ in
$\SL(V_{10})$ is
$G_2\times \SL(3)$, where $G_2$ is the stabilizer of $\alpha$ and $\SL(3)$ is the stabilizer of $\beta$, and the point $[\sigma_0]$ of $\P(\bw3V_{10}^\vee)$ is polystable for the $\SL(V_{10})$-action.
\end{coro}

\begin{proof}  The stabilizer $G_{\sigma_0}$ of $[\sigma_0]$  obviously contains $G_2\times \SL(3)$.\ For the reverse inclusion, it suffices to show that  $G_{\sigma_0}$ preserves the decomposition
\begin{equation}
\label{eqdec7juil}
V_{10}=V_7\oplus W_3.
\end{equation}
Now $G_{\sigma_0}$ acts on $\Gr(6,V_{10})$ preserving
the Debarre--Voisin variety $K_{\sigma_0}$, which is a product \mbox{$K'_\alpha\times \P(W_3^\vee)$} by Proposition
\ref{proisotsubspace}.\ But the connected component of the automorphisms group of a product
of projective varieties is the product of the connected components of its factors.\ Thus
  $G_{\sigma_0}$ acts on each factor $K'_\alpha$ and $\P(W_3^\vee)$.\ This  implies that it preserves the direct sum decomposition
(\ref{eqdec7juil}).

 \rem{To prove the polystability of $[\sigma_0]$, we invoke as before Luna's results.\ By \cite[Corollaire~1]{luna}, the $\SL(V_{10})$-orbit of $\sigma_0$ in $\bw3V_{10}^\vee$ is closed if and only if   its orbit under the  normalizer in $\SL(V_{10})$
 of its stabilizer $G_{\sigma_0}=G_2\times \SL(3)$ is closed.\ Any element of   this normalizer   must preserve the direct sum decomposition $V_{10}=V_7\oplus W_3$, hence can be written
 as    $ \lambda g \cdot \lambda' g' $, with $g \in N_{\SL(V_7)}(G_2)$, $g'\in \SL(3)$, and $\lambda^7\lambda^{\prime 3}=1$.\ 
The group $G_2$ having finite index in  its normalizer $N_{\SL(V_7)}(G_2)$ (Proposition~~\ref{normal}), the closedness of 
  the   $\SL(V_{10})$-orbit
  is equivalent to the closedness of 
  the orbit for the $\C^\star$-action $t\cdot (\alpha+\beta)=t^3\alpha+t^{-7}\beta$.\ This holds  because neither $\alpha$ nor $\beta$ is $0$.\ This proves that $[\sigma_0]$ is  polystable.} 
  \end{proof}
  
  \subsubsection{Degenerations and excess bundles}

The Debarre--Voisin variety $K_{\sigma_0}$  is, by Corollary~\ref{corokalphabeta}, smooth   of codimension $14$
in $\Gr(6,V_{10})$ and isomorphic to $K'_\alpha\times \P(W_3^\vee)$.\ It is the zero-locus of a section
  of the rank-$20$ vector bundle $\bw3\cE_6$ on $\Gr(6,V_{10})$, hence it carries
an excess bundle $\cF$ of rank $6$, described in the
 following proposition.

\begin{prop}\label{proexcessalphabeta}
One has an isomorphism $\cF\cong \cQ_2\otimes((\bw2\cE_4)/\cQ_3)$ between vector bundles on  $K_{\sigma_0}\isom K'_\alpha\times \P(W_3^\vee)$, where
\begin{itemize}
\item the bundle $\cQ_2$   is the pullback of
the rank-$2$ quotient bundle on $\P(W_3^\vee)$,
\item  the bundle $\cE_4$   is the pullback of the dual of
the tautological rank-$4$ subbundle on $K'_\alpha\subset \Gr(4,V_7)$, 
\item  the bundle $\cQ_3$   is the pullback of
the rank-$3$ quotient bundle on $K'_\alpha\subset \Gr(4,V_7)$,
\item  the injective map $\cQ_3\hra \bw2\cE_4$ is induced by the composite map
$$V_7\otimes \cO_{K_{\sigma_0}}\xrightarrow{\ \alpha\,\lrcorner\ } \bw2V_7^\vee\otimes \cO_{K_{\sigma_0}}\to \bw2\cE_4
.$$
\end{itemize}
\end{prop}

\begin{proof} The excess bundle $\cF$ is by definition the cokernel of
\begin{equation*}
d\sigma_0\colon T_{\Gr(6,V_{10})}\lra \bw3\cE_6.
\end{equation*}
Along $K_{\sigma_0}$, Proposition \ref{proisotsubspace} tells us that
$\cE_6=\cE_4\oplus\cQ_2$, so that
\begin{equation}\label{eqdecomp12juin}
\bw3\cE_6=\bw3\cE_4\oplus
(\bw2\cE_4\otimes \cQ_2)\oplus (\cE_4\otimes \bw2\cQ_2).
\end{equation}
On the other hand, the tangent bundle $T_{\Gr(6,V_{10})}$ is isomorphic to
$\cE_6\otimes \cE_4$ and
$d\sigma_0$ is the composition
\begin{equation}\label{eqmapcompdiff12juin}
\cE_6\otimes \cE_4\to \cE_6\otimes \bw2\cE_6\to \bw3\cE_6,
\end{equation}
where the second map is the wedge product map and the first one is
induced by the factorization
\begin{equation*}
\cE_4\lra  \bw2\cE_6
\end{equation*}
of $(\sigma_0)\,\lrcorner\,\colon V_{10}\otimes \cO_{K_{\sigma_0}}\to \bw2\cE_6$.\ We now decompose
$T_{\Gr(6,V_{10})}=\cE_6\otimes \cE_4$ along $K_{\sigma_0}$ as
\begin{equation}\label{eqmapcompdiff12juinter}
 T_{\Gr(6,V_{10})}=(\cE_4\oplus \cQ_2)\otimes (\cQ_3\oplus \cE_1)
=(\cE_4\otimes \cQ_3)\oplus( \cQ_2\otimes \cQ_3)\oplus
(\cE_4\otimes\cE_1)\oplus (\cQ_2\otimes \cE_1).
\end{equation}
The composite map (\ref{eqmapcompdiff12juin}) maps (\ref{eqmapcompdiff12juinter}) to
(\ref{eqdecomp12juin}) preserving the decompositions and it is easy to see that the only piece with a nontrivial quotient is
$$\cQ_2\otimes \cQ_3\lra \bw2\cE_4\otimes \cQ_2,$$
where the map is induced by $\alpha\,\lrcorner{}\,$.\ This completes  the proof.
\end{proof}

The following theorem is the main result of this section.

\begin{theo}\label{theogenre10}
Let $(\sigma_t)_{t\in \Delta}$ be a general $1$-parameter deformation.\
Let $\cK\to\Delta$ be the associated family of Debarre--Voisin varieties and let $\cK^0$ be the irreducible component of $\cK$ that dominates~$\Delta$.\
Then $\cK^0\to \Delta$
is smooth and its central fiber is isomorphic to $S^{[2]}$,  embedded in~$\Gr(6,10)$ as in Lemma~\ref{lem513}, where
$S$ is a general  K3 surface of degree 18.
\end{theo}

\begin{proof} The proof follows the same line as the proof of Theorem~\ref{theoprecis}.\ We apply  Proposition~\ref{proexcess} and conclude that
the central fiber is the zero-locus of a general section of the excess bundle
$\cF$
on $K_{\sigma_0}$.\
 It is in particular smooth since the excess bundle is generated by its sections.\
The proof is completed using Proposition~\ref{proexcessalphabeta} and  the following proposition.\end{proof}

\begin{prop} \label{propidentexcessgenre10}
Let $S\subset X\subset \Gr(2,V_7)$ be a general $K3$ surface of degree 18.\ The morphism~$\phi$ from Lemma~\ref{lem513} induces an isomorphism between 
 $ S^{[2]} $ and the    zero-locus in $K_{\sigma_0}$ of a general  section of
the excess bundle  $\cF=\cQ_2\otimes((\bw2\cE_4)/\cQ_3)$.
\end{prop}

\begin{proof} The space of global  sections of $\cF$ is equal to  $W_3^\vee\otimes (\bw2V_7^\vee/V_7)$.\ We  identify  $V_7$ with~$I_X$.\ Choosing a general section~$s$ of $\cF$,  we thus get a $K3$ surface $S\subset X$ defined by the three-dimensional space of sections
  $\Im (W_3\to  H^0(X,\cO_X(1)))$.

 {Lemma~\ref{lem513} and the lemma below imply that $\phi$ is an injective morphism between  $S^{[2]}$ and the smooth zero-locus of $s$.\ By Zariski's Main Theorem, it is an isomorphism, which proves the proposition.}
\end{proof}

\begin{lemm}\label{lem520}
The
zero-locus of $s$ coincides with the image $\phi(S^{[2]} )\subset K_{\sigma_0}$.
\end{lemm}

\begin{proof} Let $[V_4]\in K'_\alpha$ and let $W_2\subset W_3$ be of dimension $2$.\
Assume that the section $s$ of~$\cF$ vanishes at $([V_4],[W_2])$.\ Lifting $s$ to an element of $\Hom (W_3,\bw2V_7^\vee )$, this means by
the description of $\cF$ given in Proposition \ref{proexcessalphabeta} that the image of the two-dimensional space
$s(W_2)\subset \bw2V_7^\vee$ in~$\bw2V_4^\vee$  is contained in the image $V_3\subset \bw2V_4^\vee$ of
the natural map $\alpha\,\lrcorner\,\bullet\colon V_7/V_4\to \bw2V_4^\vee$.

The intersection of $X$ with the Grassmannian $\Gr(2,V_4)$  is   defined
by the three Pl\"ucker equations given by $V_3$.\ The existence of $W_2$ as above is equivalent to saying that $V_3$ and $W_3$ span only a subspace of dimension $4$ of $\bw2V_4^\vee$, or, equivalently, that the length of the subscheme of~$\Gr(2,V_4)$ defined by $V_3$ and $W_3$ is at least~$2$.\ This subscheme is equal to  $S\cap \Gr(2,V_4)$.\ Furthermore, the space $W_2$ is contained in   the subspace of $W_3$  vanishing on the conic defined by $X\cap \Gr(2,V_4)$.\
Looking at the construction of the injective morphism $\phi\colon S^{[2]}\to K_{\sigma_0}$ given in  Lemma~\ref{lem513}, we conclude  that $\phi(S^{[2]})$ is contained in
the vanishing locus of~$s$.\ As both are   fourfolds of the same degree, they must agree.\
This proves the lemma.
\end{proof}


\section{The HLS divisor $\cD_{10}$}
\label{secpf6} 

Let $(S,L)$ be a general K3 surface of degree $10$.\
As we saw in Section~\ref{secrappelonpic},  the Hilbert square~$S^{[2]}$  with the polarization $2\L2-3\delta$ is a limit of Debarre--Voisin varieties.\ We will first construct a
\mbox{rank-$4$} vector bundle on $S^{[2]}$ mapping it to $\Gr(6,10)$ and then construct a trivector
$\sigma_0$ vanishing on the image.\ It turns out that $\sigma_0$ is $\SL(2)$-invariant
and that the Debarre--Voisin variety  $K_{\sigma_0}$ only depends  on a certain
 $\SL(2)$-invariant Fano threefold  
$X\subset\Gr(2,5)$ in which $S$ naturally sits.\ The rank-$4$ vector bundle is not globally
 generated and $K_{\sigma_0}$ is not irreducible in this case, but we nevertheless conclude in Theorem \ref{theopourgenre6} that a $1$-parameter degeneration to $\sigma_0$
 expresses a general pair $(S^{[2]},2\L2-3\delta)$ as a limit of Debarre--Voisin varieties.

\subsection{The Fano threefold $X$ 	and $K3$ surfaces of degree $10$}\label{secXg6}
Let $V_5$ be a $5$-dimensional vector space and let $W_3\subset \bw2 V_5$ be a general  $3$-dimensional vector subspace.\
Let $X\subset \Gr(2,V_5^\dual)$ be the Fano threefold of index $2$ and degree $5$ defined by the   Pl\"ucker equations in $W_3$.\ 
It has no moduli, the variety of lines contained in $X$ is a smooth surface  isomorphic to
$\P^2$ (\cite[Corollary~(6.6)(ii)]{iskovskikh}), and the automorphism group of $X$ is $\PGL(2)$.\ In fact, if $U_2$ is the standard self-dual irreducible representation of $\SL(2)$ and $V_5:=\Sym^4\! U_2$, there is a direct sum decomposition
 \begin{equation}\label{eqdecompsl2} 
 V_{10}:=\bw2V_5=  V_7 \oplus W_3 
\end{equation}
into irreducible representations, with  $V_7=\Sym^6\! U_2$ and  $W_3=\Sym^2\! U_2$, so that   $X$ is the unique  $\SL(2)$-invariant section of $\Gr(2,V_5^\dual)$ by a linear subspace of codimension 3 (\cite[Section~7.1]{chsh}).

A general polarized $K3$ surface $(S,L)$ of degree~$10$ is obtained as a quadratic section of $X$ (\cite{mukai}).\
Let $\cE_2$ be the restriction to $X$ of the dual of the tautological subbundle on $\Gr(2,V_5^\dual)$ \rem{(it is stable and rigid))}.\  Lemma \ref{ledescend} gives us a  rank-$4$ vector bundle   $ \cK_{\cE_2} $ on $ X^{[2]}$ whose restriction~$\cQ_4 $ to $\SS$ satisfies 
 $H^0(S^{[2]} ,\cQ_4)\isom \bw2V_5$ and $\det(\cQ_4)=2\L2-3\delta$.

\begin{rema} \label{chernn}
Using the package
  Schubert2 of Macaulay2 (\cite{macaulay2}; the code can be found in \cite{comp}), one checks that
the vector bundle $\cQ_4$  has the same  Segre numbers 
\begin{equation*}
s_1^4 = 1452,\ s_1^2 s_2 = 825,\
s_1 s_3= 330,\ s_2^2 = 477,\ s_4 = 105 
\end{equation*}
as the rank-$4$ tautological quotient bundle on   Debarre--Voisin varieties
$K_\sigma\subset \Gr(6,10)$, computed in \cite[(11)]{devo}.\ 
The pair
 $(S^{[2]},\cQ_4)$ is therefore a candidate   to be a limit of Debarre--Voisin varieties (as a subvariety of $\Gr(6,10)$).\ One difficulty in the present
  case is  that
the vector bundle $\cQ_4$ is not generated by its sections (Proposition~\ref{propbirat}(b)).\ This explains why in Theorem \ref{theohasgenus6}, the
  central fiber is only birationally isomorphic to $S^{[2]}$.
\end{rema}

Since $W_3$ has no rank-$2$ elements,  for any
$[x]\in \P(V_5)$, the subspace
$$x\wedge W_3\subset \bw3V_5\cong \bw2 V_5^\dual$$
has dimension $3$.\ Set
\begin{equation*}
V_{4,[x]}:=x\wedge V_5\subset  \bw2 V_5.
\end{equation*}
We have
$\langle V_{4,[x]},x\wedge W_3\rangle=0$.\ 
Setting 
\begin{equation*}
 V_{7,[x]} :=(x\wedge W_3)^\perp\subset \bw2 V_5,
\end{equation*}
 we thus have
 $V_{4,[x]}\subset V_{7,[x]}\subset \bw2 V_5$.\
 Finally, we set
\begin{equation}\label{defk1}
K_1:=\{ [W_6]\in   \Gr(6,\bw2 V_5)\mid \exists\, [x]\in \P(V_5)\quad V_{4,[x]}\subset W_6\subset V_{7,[x]}\}.
\end{equation}
We observe that $K_1$ is smooth of dimension  $6$.

\begin{prop} \label{propbirat}
{\rm{(a)}}
 The space $\bw2 V_5$ of global sections of the rank-$4$ vector bundle   $ \cK_{\cE_2} $ on $ X^{[2]}$ induces a
 birational map
\begin{equation*}
\phi\colon X^{[2]}\dra K_1\subset  \Gr(6,\bw2 V_5)
\end{equation*}
which is regular  outside the $4$-dimensional locus in $ X^{[2]}$ consisting of length-$2$ subschemes
 contained in a line contained in $X$.
 
\noindent {\rm{(b)}} If $S$ is general, the restriction of $\phi$ to $\SS$ is the map induced by the global sections of $\cQ_4$ and it is regular  outside a smooth surface isomorphic to the surface of lines in $X$.
\end{prop}

\begin{proof} 
At a point of $ X^{[2]}$ corresponding to different vector subspaces $V_2,V_2' \subset V_5^\dual$, the evaluation map of $ \cK_{\cE_2} $ is the restriction 
$$\bw2V_5\lra V_2^\dual\otimes V_2^{\prime\dual}.$$
It is surjective if and only if 
$V_2\cap V_2'=\{0\}$, which means exactly that the line joining $[V_2]$ and~$[V'_2]$
 is not contained in  $\Gr(2, V_5^\vee)$ or, equivalently, in $X$.
 
\rem{At a} nonreduced point $z=([V_2],u)$,  where $u\in\Hom (V_2,V_5^\vee/V_2)$, the fiber $\cK_{\cE_2,z}$ appears  in an extension
$$
0\to \Sym^2\!V_2^\vee 
 \to \cK_{\cE_2,z} \xrightarrow{\ a\ }  \bw2V_2^\vee\to 0
$$
The composition $r\colon\bw2V_5\to \bw2V_2^\vee$ of the evaluation map $\bw2V_5\to \cK_{\cE_2,z}$ at $z$ with $a$ is given by restriction, hence is surjective, and its kernel maps to $\Sym^2\!V_2^\vee$ via the composite map
$$\Ker(r)\to (V_5^\vee/V_2)^\vee\otimes V_2^\vee \xrightarrow{\ u^\vee\otimes \Id\ } V_2^\vee \otimes V_2^\vee
\to \Sym^2\!V_2^\vee.$$
 This composite  map (hence also the evaluation map   at $z$) is surjective if and only if    $u$ has (maximal) rank $2$, which means exactly that the line spanned by $z$ is contained in  $\Gr(2, V_5^\vee)$ or, equivalently, in $X$.\ This proves the first part of (a), and also (b), since a general $S$ contains no lines.

It remains to prove that $\phi$ is birational onto $K_1$.\ Let $[W_6]=\phi([V_2],[V'_2])$.\
 If $V_2$ and $V_2'$ are complementary, they span a subspace $V_4^\dual\subset V_5^\dual$ of dimension~$4$.\ Denoting  by $x\in V_5$ a linear form defining $V_4^\vee$, one has $V_{4,[x]}\subset W_6$.\  Next, $W_3\vert_{ V_4^\dual}$  vanishes on $\bw2V_2$ and $\bw2{V_2'}$, hence
 $$W_3\vert_{  V_4^\dual}\subset  V_2^\dual\otimes V_2^{\prime\dual}.$$
 The vanishing of $W_6$ in $ V_2^\dual\otimes V_2^{\prime\dual}$ thus implies that
    $W_{6}\vert_{V_4^\dual}$ is  orthogonal to
 $W_3\vert_{ V_4^\dual}$ for the natural pairing on $\bw2V_4$.\ Equivalently,  $W_6$ is orthogonal to
$ x\wedge W_3$ for the pairing between
$\bw2V_5$ and~$ \bw3V_5$.\
 This shows that $\Im (\phi)$ is contained in
$K_1$.

Conversely, let $[W_6]$ be a general element of
$K_1$.\ Then
$$V_{4,[x]}\subset W_6\subset V_{7,[x]}$$
for some  $[x]\in \rem{\P}(V_5)$, so that $W_6\vert_{V_4^\vee}$ has dimension $2$, where $V_4^\vee$ is defined by $x$.

\ Since $W_6 $ is orthogonal to $ x\wedge W_3$, it follows that
$W_{6}\vert_{  V_4^\dual} $ is orthogonal to $ W_3\vert_{ V_4^\dual}$.\ The $3$-dimensional space $W_3\vert_{ V_4^\dual}\subset \bw2V_4$ defines
a conic $X\cap \Gr(2,V_4^\dual)$ in the Grassmannian $\Gr(2,V_4^\dual)$ and it is easy to check that
a $2$-dimensional subspace $W_2' \subset \bw2V_4$ cuts out two points on this conic
if and only $W_2'\perp W_3\vert_{  V_4^\dual}$.\ This shows that $K_1$ is contained in $\Im(\phi)$.

The proof that $\phi$ is birational follows from the last argument.\ Indeed,  pairs of    points in the conic above
correspond bijectively to  two-dimensional subspaces of $W_3\vert_{  V_4^\dual}$, at least if the conic is nonsingular.
\end{proof}


 \subsection{The $\SL(2)$-invariant trivector}\label{se62}

We now construct a trivector $\sigma_0$ on  $V_{10}=\bw2V_5$ such that~$K_1$ is a generically smooth component of the Debarre--Voisin
variety $K_{\sigma_0}$.

\begin{prop}\label{proexistencesigma} 
There exists a unique trivector
$\sigma_0\in \bw3V_{10}^\dual$ such that,
for any $[x]\in \P(V_5)$,
the restriction
$\sigma_{0}\vert_{ V_{7,[x]}}$ comes from a nonzero element of $\bw3(V_{7,[x]}/V_{4,[x]})^\dual$.\ This trivector   is invariant under \rem{the} $\SL(2)$-action described in Section~\ref{secXg6}.
\end{prop}

\begin{proof} 
Let $\cV_4$ be the rank-4 vector bundle on $\P(V_5)$  image of the
bundle map $V_5\otimes \cO_{\P(V_5)}(-1)
\to\bw2V_5\otimes \cO_{\P(V_5)}$ given by   wedge product.\ We define another vector bundle  
$\cV_7$ on $\P(V_5)$ by the exact sequence
\begin{equation} \label{eqexact228juin}
0\to \cV_7\to \bw2V_5\otimes \cO_{\P(V_5)}
\xrightarrow{\ a\ } W^\dual_3\otimes \cO_{\P(V_5)}(1)\to 0,
\end{equation}
where the  map $a$ at the point $[x]$ is the wedge product map with $x$ with value in
$\bw3V_5$, followed by the natural  map $\bw3V_5\cong \bw2 V_5^\dual\to W_3^\dual$.\ The fibers of $\cV_4$ and $\cV_7$ at $[x]\in \P(V_5)$ are the vector subspaces
$$V_{4,[x]}\subset V_{7,[x]}\subset \bw2V_5 $$
defined previously.\ There is an exact sequence
\begin{equation*} 
0\to \cO_{\P(V_5)}(-2)\to V_5\otimes \cO_{\P(V_5)}(-1)
\to \cV_4  \to 0
\end{equation*}
from which, together with  \eqref{eqexact228juin}, we deduce
 $\det (\cV_4)\isom \det ( \cV_7)\isom\cO_{\P(V_5)}(-3) $,
hence
 $$\det(\cV_7/\cV_4)=\cO_{\P(V_5)}.$$
 The line bundle $\bw3(\cV_7/\cV_4)^\dual$ thus has   a nowhere vanishing section
 $\omega$.\

 We set $\cE_7:=\cV_7^\dual$.\ Via  the inclusion $\bw3(\cV_7/\cV_4)^\dual\subset \bw3\cE_7$, the section ${\omega}$ provides a section   of~$\bw3\cE_7$.\ By Lemma \ref{leisosection} below, this section defines a unique trivector $\sigma_0$ with the desired properties, which proves the proposition.\end{proof}

 \begin{lemm}\label{leisosection}
 The restriction map
 \begin{equation*}
 \bw3(\bw2V_5^\dual)\otimes \cO_{\P(V_5)}\lra \bw3\cE_7
 \end{equation*}
 induces an isomorphism on global sections.
 \end{lemm}

 \begin{proof}
The dual
 \begin{equation}\label{eq59}
0\to W_3\otimes \cO_{\P(V_5)}(-1)\to \bw2V_5^\dual\otimes \cO_{\P(V_5)}\to \cE_7\to0
\end{equation}
of the exact sequence
 (\ref{eqexact228juin})  implies that the bundle $\cG$ defined by
   the exact sequence
\begin{equation*}
0\to\cG\to \bw3(\bw2V_5^\dual)\otimes \cO_{\P(V_5)}\to \bw3\cE_7\to 0
\end{equation*}
 has a filtration with graded pieces
$$W_3\otimes \bw2\cE_7(-1),\ \bw2W_3\otimes \cE_7(-1),\ \bw3W_3\otimes \cO_{\P(V_5)}(-3).$$
It thus suffices to show that these three bundles have vanishing
$H^0$ and $H^1$.

 This is obvious for  the last bundle.\ For the second  bundle, this follows   from (\ref{eq59}).\ 
For the first bundle, we take the second exterior power of
 \eqref{eq59}  tensored by
$  \cO_{\P(V_5)}(-1)$ and get
 \begin{equation*}
 0\to\cG'\to \bw2(\bw2V_5^\dual)\otimes \cO_{\P(V_5)}(-1)\to ( \bw2\cE_7)(-1)\to 0,
\end{equation*}
where the bundle
$\cG'$ is an extension
\begin{equation}\label{eq60}
0\to \bw2W_3\otimes \cO_{\P(V_5)}(-3)\to\cG'\to
W_3\otimes \cE_7(-2)\to0.
\end{equation}
We then get the desired vanishing
$$H^0(\P(V_5),\bw2\cE_7(-1))=0=H^1(\P(V_5),\bw2\cE_7(-1))$$
from
the vanishings $H^1(\P(V_5),\cG')=H^2(\P(V_5),\cG')=0$ which follow  from
 \eqref{eq60}  and the similar vanishings for  $\cE_7(-2)$.
\end{proof}

The threefold $X$ discussed in Section~\ref{secXg6}  embeds
in $\Gr(3,\bw2V_5)$ as follows:
a point $[V_2]\in X$ parametrizes  a vector subspace $V_2\subset V_5^\dual$ of dimension
$2$.\ Let $V_3\subset V_5$ be the kernel of the restriction map
$V_5\to   V_2^\dual$.\ Then
$U_3:=\bw2V_3\subset \bw2V_5$  has dimension $3$ and it determines $V_2$.\

\begin{prop}\label{lexsing}
{\rm (a)}  The threefold $X\subset \Gr(3,\bw2V_5)$ is contained
in the singular locus
of the Pl\"ucker hypersurface $X_{\sigma_0}$.

\noindent{\rm (b)}  The rational map $\phi\colon X^{[2]}\dra \Gr(6,\bw2V_5)$ defined in Proposition~\ref{propbirat}
sends a general pair $([V_2],[V_2'])$ to the subspace
$\langle U_3,U'_3\rangle\subset \bw2V_5$.

\noindent{\rm (c)} The variety $K_1$ is contained in the Debarre--Voisin variety  $K_{\sigma_0}$.
\end{prop}

\begin{proof} 
We first observe the following.

\begin{lemm}
\label{sublepourg6} 
Let $[V_2]\in X$
and let $V_3$ and $U_3=\bw2V_3$ be as above.\ For any  $[x]\in \P(V_3)$, we have $U_3\subset V_{7,[x]}$ and $\dim (U_3\cap V_{4,[x]})=2$.
\end{lemm}

\begin{proof} We want to show that $x\wedge W_3$ is orthogonal to
$U_3$, which means that
for any $w\in W_3$ and any $u\in U_3$, one has $x\wedge w\wedge u=0$ in $\bw5V_5$.\ This is clear, \rem{since}
$x\wedge u\in \bw3V_3$ and $w$ vanishes on~$V_2$, hence belongs to $V_3\wedge V_5$.
The second statement is obvious because $U_3\cap V_{4,[x]}=x\wedge V_3$.
\end{proof}

\rem{We now show item  (a) of the proposition.\ Let again~\mbox{$[V_2]\in X$,}   let $V_3$ and $U_3$ be as above, and let    $[x]\in \P(V_3)$.\ As shown in the proof of~\cite[Proposition~3.1]{devo}, the intersection $X_{\sigma_0}\cap \Gr(3,V_{7,[x]})$ is singular at a point $U'_3\subset V_{7,[x]}$ if $\sigma_0$ vanishes on $\bw2U'_3\wedge V_{7,[x]}$.\ This happens if $\dim(U'_3\cap V_{4,[x]})\ge2$ because, by construction, the $3$-form $\sigma_0\vert_{ V_{7,[x]}}$ is  the wedge product of 3 linear forms that vanish  on $V_{4,[x]}$.\ 
 Lemma \ref{sublepourg6} says that $U_3\subset V_{7,[x]}$ satisfies this condition.\ 
 
 We thus proved that
$X_{\sigma_0}\cap \Gr(3,V_{7,[x]})$ is singular at the point $[U_3]$, for any  $[x]\in \P(V_3)$.\
This means that the Zariski tangent space $T_{X_{\sigma_0},[U_3]}$  contains $T_{\Gr(3,V_{7,[x]}),[U_3]}$ for any  $[x]\in \P(V_3)$.\ We then use the following fact to conclude that $X$ is contained in the singular locus of $X_{\sigma_0}$.}

\begin{lemm}\label{lepourpropsingx}
The vector subspaces $T_{\Gr(3,V_{7,[x]}),[U_3]}\subset
T_{\Gr(3,\sbw2V_5),[U_3]}$, for $[x]\in \P(V_3)$, span the tangent space  $T_{\Gr(3,\sbw2V_5),[U_3]}$.
\end{lemm}

\begin{proof} We have $T_{\Gr(3,V_{7,[x]}),[U_3]}=\Hom (U_3,V_{7,[x]}/U_3)$ and
$T_{\Gr(3,\sbw2V_5),[U_3]}=\Hom (U_3,\bw2V_5/U_3)$, so the lemma is equivalent to the fact that the $V_{7,[x]}$, for $[x]\in \P(V_3)$, span $\bw2V_5$.\ As
$V_{7,[x]}=x\wedge W_3^{\perp}$, the statement is equivalent to 
$\bigcap_{x\in V_3}(x\wedge W_3)=0$, which is obvious.
\end{proof}

By Proposition~\ref{remaxsing}(b), there is  a  rational map
$f\colon X^{[2]}\dra K_{\sigma_0}$.\ Let us compare
$\phi$ and $f$.\ The map
$\phi$ sends  $([V_2],[V_2'])$ to the kernel of the
map $\bw2V_5\to   V_2^\dual\otimes V_2^{\prime\dual}$.\ Since $V_3$ vanishes in~$V_2^\dual$,
the image of $U_3=\bw2V_3$ vanishes in $V_2^\dual\otimes V_2^{\prime\dual}$ and similarly for
$U'_3$.\ It follows that
$$U_3+U'_3=\Ker (\bw2V_5\to   V_2^\dual\otimes V_2^{\prime\dual})$$
when both spaces have the same expected dimension $6$.\ This proves items~(b) and~(c).
\end{proof}

\subsection{Stabilizer, degenerations, and excess bundles}

Recall that $X\subset \Gr(2,V_5^\dual)$ is a Fano threefold of index $2$ and degree $5$.\ We have defined a trivector   $\sigma_0$ on  $V_{10}=\bw2V_5$ such that the smooth sixfold $K_1$ defined in \eqref{defk1} is contained in $K_{\sigma_0}$ (Proposition~\ref{lexsing}(c)).

 The birational map
$\phi\colon X^{[2]}\dra K_1$ defined in Proposition~\ref{propbirat}
induces an isomorphism between
  a dense open subset $U\subset X^{[2]}$ and its open image.\ We identify $U$ with 
  $\phi(U)$.

\begin{prop}\label{prop69}
{\rm(a)} The Debarre--Voisin variety $K_{\sigma_0}$ is smooth of dimension $6$ along
$U$, hence~$K_1$ is a generically smooth irreducible component of $K_{\sigma_0}$.

\noindent{\rm(b)} On  $U$, the excess bundle $\cF$  and the tautological bundle
$\cT_{\cO_X(2)}$  coincide as quotients of \mbox{$\bw3V_{10}^\vee\otimes \cO_U$.}
\end{prop}

Before giving the proof, let us note the following consequence.

\begin{coro} \label{corstabgenus6}
The neutral component   of the stabilizer of $\sigma_0$ for the $\SL(V_{10})$-action is the group~$\SL( 2)$.
\end{coro}

{We do not prove that the point $[\sigma_0]$ is polystable.}

\begin{proof}
An element $g$ of this stabilizer acts on the Debarre--Voisin variety
$K_{\sigma_0}$ and the neutral component  acts preserving the irreducible components.\ By Proposition \ref{prop69}, it acts on $K_1$.\ But
$K_1$ is a $\P^2$-bundle over $\P(V_5)$, so $g$ (via its action
on $\Gr(6,\bw2V_5)$)  has to act on
the base $\P(V_5)$ and this action lifts to the projective bundle $K_1$.\
One easily concludes that $g$ defines an automorphism of $\P(V_5)$ whose induced action on $\Gr(2,V_5^\dual)$
preserves  $X$.
\end{proof}

The proof of Proposition \ref{prop69} will use a few more preparatory steps.\
We start with the
 following easy lemma.

 \begin{lemm}\label{lepourV6dansKsigma0}
 For any  $[W_6]\in K_{\sigma_0}$ and  any
 $[x]\in \P(V_5)$,
 the vector space $W_6\subset \bw2V_5$ intersects $V_{4,[x]}$ nontrivially; it  follows that $\dim  (\P(W_6)\cap \Gr(2, \bw2V_5))\geq 3$.
 \end{lemm}

\begin{proof} The assumption is that $\sigma_0$ vanishes on $W_6$.\ The space
$V:=W_{6}\cap V_{7,[x]}$ is  of dimension at least $3$.\ By construction
 (see Proposition \ref{proexistencesigma}), the restriction of $\sigma_0$ to
$V_{7,[x]}$ is a generator of $\bw{3}(V_{7,[x]}/V_{4,[x]})^\dual$, hence the vanishing of
$\sigma_0\vert_{V}$ means   $V\cap V_{4,[x]}\ne \{0\}$.\ Hence $W_6\cap V_{4,[x]}\ne \{0\}$.

For the second statement, observe that the set of
$[x]\in \P(V_5)$ such that $W_6\cap V_{4,[x]}\ne \{0\}$ is the image in
$\P(V_5)$ of the universal $\P^1$-bundle  over $\P(W_6)\cap \Gr(2, \bw2V_5)$.\
Since all $[x]\in \P(V_5)$ have this property, the dimension of this bundle must be at least $4$.
\end{proof}

Let us show the following   consequence.

\begin{coro} \label{coropourKsigma01007}
Let $K'_1$ be an irreducible component of $K_{\sigma_0}$ containing $K_1$.\ For any $[W_6]\in K'_1$, there is
a unique  $[x]\in \P(V_5)$ such that $V_{4,[x]} $ is contained in $W_6$.
\end{coro}

\begin{proof} The uniqueness is clear, as $x\wedge V_5+y\wedge V_5$ has dimension $7$
for nonproportional $x,\,y$.\ For the existence, we observe that for a general  $[V_6]\in K_1$, the intersection $\P(V_6)\cap \Gr(2, \bw2V_5)$ is equal to
$\P(V_{4,[x]})$ with its reduced structure.\
We now deform $[V_6]$ to a general element $[W_6]$ of the   component $K'_1$, say along
 a family $(\cV _{6,t})_{t\in \Delta}\subset \bw2V_5$ of
 $6$-dimensional vector subspaces.\
 By Lemma~\ref{lepourV6dansKsigma0},
we know that for any $t\in \Delta$, the intersection $\P(\cV _{6,t})\cap \Gr(2, \bw2V_5)$ remains of dimension $\geq3$.\ As for $t=0$, it is reduced, of dimension $3$ and  degree $1$,
the same holds for   $t\in \Delta$ general.\ As the only $3$-dimensional projective subspaces
 of $\Gr(2, \bw2V_5)$ are of the form $\P(V_{4,[x]})$, we obtain that $W_6=\cV _{6,t}$, for $t$ general, contains a space $V_{4,[x]}$.
\end{proof}

\begin{proof}[Proof of Proposition \ref{prop69}{(a)}]
Let as above $K'_1$ be an irreducible component of $K_{\sigma_0}$ containing $K_1$ and let $[W_6]\in K'_1$.\ We know by Corollary \ref{coropourKsigma01007} that  there exists
$[x]\in\P(V_5)$ such that $V_{4,[x]}$ is contained in $W_6$.\ We also  note from the proof of
Corollary \ref{coropourKsigma01007} that the point $[x]\in\P(V_5)$ is general.\
There is a short  exact sequence
\begin{equation}\label{eqshort13juill}
0\to   V_{4,[x]}\to   \bw2V_5\to   \bw2V_{4,[x]}\to   0.
\end{equation}
Here, $V_{4,[x]}=x\wedge V_5$ is seen on the left as a subspace of $   \bw2V_5$ and on the right  as the quotient   $V_5/\C x$.
 
The trivector $\sigma_0\in\bw3(\bw2V_5)^\dual$ vanishes in  the first quotient $\bw3V_{4,[x]}^\dual$, hence
it has an image $\overline{\sigma}_{0,x}$ in the next step of the filtration on $\bw3(\bw2V_5)^\dual$ associated with
(\ref{eqshort13juill}), namely $\bw2V_{4,[x]}^\dual\otimes \bw2V_{4,[x]}^\dual$.\
From the construction of $\sigma_0$, we know that
$\sigma_0\vert_{V_{7,[x]}}$ comes from $\bw3(V_{7,[x]}/V_{4,[x]})^\dual$, which implies that
$\overline{\sigma}_{0,[x]}$ vanishes in $(V_{7,[x]}/V_{4,[x]})^\dual\otimes \bw2V_{4,[x]}^\dual$, or equivalently belongs
to $(x\wedge W_3)\otimes \bw2V_{4,[x]}^\dual$, where we identify $x\wedge W_3\subset \bw2V_{5}^\dual$ as defining $V_{7,[x]}$ (so that its image in $\bw2V_{4,[x]}^\dual$ defines
$V_{7,[x]}/V_{4,[x]}$)).\
Let us examine $\overline{\sigma}_{0,x}\in(x\wedge W_3)\otimes \bw2V_{4,[x]}^\dual$.\ We claim
the following.

\begin{lemm} \label{lerankdesigma0x}
 For    $[x]\in\P(V_5)$ general, the rank of $\overline{\sigma}_{0,x}$ is $3$.
\end{lemm}

\begin{proof} Recall that $V_5$ and $W_3$ are  irreducible representations of $\SL(2)$ (Section~\ref{secXg6}).\ The trivector~$\sigma_0$ is invariant under the induced $\SL(2)$-action on  $\bw3 V_{10}^\dual=\bw3(\bw2V_5)^\dual$.

From \eqref{eqexact228juin}, we see that $V_{4,[x]}$, seen as a quotient of $V_5$, is the fiber at $[x]$ of the vector bundle  $\cV'_4:=\cV_4(1)$.\ Since $H^0(\P(V_5),\bw2\cV'_4)\isom  \bw2V_5$ and $W_3\subset  \bw2V_5$ is general, there is an injection
 $$  W_3\otimes \cO_{\P(V_5)}\hra \bw2\cV'_4
 $$
whose dual is a
surjection $\bw2\cV_4^{\prime\dual}\to W_3^\dual\otimes \cO_{\P(V_5)}$.\ The tensors $\overline{\sigma}_{0,x}  $  
globalize to a section $\overline{\sigma}_0$ of the bundle $W_3\otimes \bw2\cV_4^{\prime\dual}\otimes \cO_{\P(V_5)}(1)$.\
Since $\det(\cV'_4)=\cO_{\P(V_5)}(1)$, we have
$$\bw2\cV_4^{\prime\dual}\otimes \cO_{\P(V_5)}(1)\cong \bw2(\cV_4(1)),$$
hence $\overline{\sigma}_0$ is a section of the bundle $W_3\otimes \bw2\cV_4'$.\
We also have
$$H^0(\P(V_5),W_3\otimes \bw2\cV_4'=W_3\otimes H^0(\P(V_5),\bw2\cV_4')=W_3\otimes \bw2V_5.$$
It follows that $\overline{\sigma}_0$ provides an element of $W_3\otimes \bw2V_5$ which must
 be $\SL(2)$-invariant.\ The decomposition (\ref{eqdecompsl2}) tells us that there is exactly
one such element, $\Id_{W_3}$ (we use the isomorphism $W_3\cong W_3^\dual$ given by the $\SL(2)$-action).\
The conclusion of this analysis is that either $\overline{\sigma}_0$ is $0$  or the rank of $\overline{\sigma}_{0,x}$ is $3$.

To finish the proof of  the lemma, we just have to exclude the case
$\overline{\sigma}_0=0$.\ If this vanishing holds, $\sigma_0$ vanishes on any $3$-dimensional subspace of $\bw2 V_5$ that intersects
one $x\wedge V_5$ along a $2$-dimensional space.\ It is  easy to exclude this possibility: the condition says that $\sigma_0\in\bw3(\bw2V_5^\dual)$ vanishes on all elements of the form
\begin{equation}
\label{eqelementspecial}
(x\wedge y)\wedge (x\wedge z)\wedge( v\wedge w)\in \bw3(\bw2V_5)
 \end{equation}
 for $x,\,y,\,z,\,v,\,w\in V_5$.\ But this would force $\sigma_0=0$, because these elements span $\bw3(\bw2V_5)$.\ Indeed,
this space is generated by general
decomposable elements of the form $m=(x\wedge y)\wedge (t\wedge z)\wedge( v\wedge w)$.\ By generality, we have
$v=\alpha x+\beta y+\gamma t+\delta z+\eps w$ and expanding $m$, we get a sum of terms of
  type (\ref{eqelementspecial}).
\end{proof}

Let us go back to the point $[W_6]$ of $K'_1$, where  $W_6$  contains $ V_{4,[x]}$ for some general $[x]\in\P(V_5)$.\ Since $\sigma_{0}\vert_{ W_6}=0$, the
tensor $\overline{\sigma}_{0,x}$ vanishes in $(W_6/V_{4,[x]})^\dual\otimes \bw2V_{4,[x]}^\dual$.\
By Lemma \ref{lerankdesigma0x}, we conclude that $x\wedge W_3$ has to vanish on $W_6$, that is
$W_6\subset V_{7,[x]}$.\ Thus $[W_6]\in K_1$ and we  proved that $K_1$ is an irreducible component
of $K_{\sigma_0}$.

In order to prove that  $K_1$ and $K_{\sigma_0}$ are equal as schemes generically along $K_1$, we  observe that the argument just given is of an
infinitesimal nature, hence proves that $K_1$ and
$K_{\sigma_0}\cap \Gr(6,x,\bw2V_5)$ are equal as schemes  generically along $K_1$, where  $\Gr(6,x,\bw2V_5)\subset \Gr(6,\bw2V_5)$ is the set of $W_6\subset \bw2V_5$ such that
$x\wedge V_5=V_{4,x}\subset W_6$ for some $x\in \P(V_5)$.\
In order to conclude, we thus just need to show that   $K_{\sigma_0}$ is schematically contained in  $\Gr(6,x,\bw2V_5)$ generically along~$K_1$.\ This is a consequence of the following infinitesimal version of  Corollary~\ref{coropourKsigma01007}.\end{proof}

\begin{lemm} Let $[W_6]\in K_1$ be general and let   $x\in \P(V_5)$ be such that
$V_{4,x}\subset W_6$.\ For any  first order deformation $[W_{6,\eps}]$ of $[W_6]$ in $K_{\sigma_0}$, there exists a first order deformation $x_\eps$ of $x$ such that,
at first order,   $V_{4,x_\eps}=x_\eps \wedge V_5\subset W_{6,\eps}$.
\end{lemm}

\begin{proof}  Let $x\wedge y\in \P(V_{4,x})$ be such that \begin{eqnarray}\label{equationajouteepourV4y}
W_6\cap (y\wedge V_5)=\langle x\wedge y\rangle.
\end{eqnarray} 
The proof
of Lemma \ref{lepourV6dansKsigma0}  shows that
there exists a unique first order deformation $ y_\eps \in \P(V_{4,y})\subset \Gr(2,V_5)$ such that
$W_{6,\eps}\cap( y\wedge V_5)=\langle y\wedge y_\eps\ra$.\ Since $[W_6]$ is  a general point of $K_1$, the set of points~$y$ satisfying (\ref{equationajouteepourV4y}) is the complement
of a closed algebraic subset of codimension $\geq2$ in~$\P(V_{4,x})$.\ The collection
of $y_\eps$  thus extends to a first order deformation of $\P(V_{4,x})$ in $\Gr(2,V_5)$.\
But the latter  are in bijection with the first order deformations of $x\in \P(V_5)$.
\end{proof}

\begin{proof}[Proof of Proposition \ref{prop69}(b)]

We are exactly in the setting of Lemma \ref{lenewpourexces} and Remark  \ref{re45}: by  Proposition \ref{lexsing}(a),  there is an embedding
$j\colon X\hra \Sing (X_{\sigma_0})\subset \Gr(3,V_{10})$;
it induces a  map
$\phi\colon X^{[2]}\dra K_1$,
where $K_1$ is a generically reduced $6$-dimensional  component of $K_{\sigma_0}$ (Proposition~\ref{prop69}(a)).\ The map $\phi$ is birational by Proposition~\ref{propbirat} and $j^*\cO_{\Gr(3,V_{10})}(1)=\cO_X(2)$.\

 On   $U$, the vector bundles
$\cT_{j^*\cO_{\Gr(3,V_{10})}(1)}=\cT_{\cO_X(2)}$ and $\cF$ both have  rank $2$ and  are  quotients of $\bw3V_{10}^\vee\otimes \cO_U$; furthermore,
Lemma \ref{lenewpourexces} says that the 
evaluation map 
$$\ev\colon\bw3V_{10}^\vee\otimes \cO_U\to \cT_{\cO_X(2)}$$
factors through $\cF$.\ This proves that they are the same.
\end{proof}

We finally  prove our main result.

\begin{theo} \label{theopourgenre6}
Let $(\sigma_t)_{t\in \Delta}$ be  a very general $1$-parameter deformation.\
Over a finite cover $\Delta'\to \Delta$, there is a family of smooth polarized \hk\ fourfolds
$\cK'\to \Delta'$ such that a general fiber
$\cK'_{t'}$ is isomorphic to $K_{\sigma_t}$ and the central fiber
  is   isomorphic  to
$S^{[2]}$, where $(S,L)$ is a general $K3$ surface of degree $10$, with the polarization $2L-3\delta$.
\end{theo}

\begin{proof}  Let $ \cK\to \Delta$ be the associated  family of Debarre--Voisin varieties, let $\cK^0$ be the irreducible component of $\cK$ that dominates $\Delta$, and let    $U\subset K_{\sigma_0}=\cK_0$ be the Zariski open set   of Proposition~\ref{prop69}.\ Then  $ \cK_0$ is smooth of dimension $6$ along $U$, so that the analysis of Section~\ref{secexcess} 
 applies.

By Proposition~\ref{prop69}(b), on $U$, the excess bundle
$\cF$ can be identified   with  $\cT_{\cO_X(2)}$ as quotients of $\bw3V_{10}^\vee$.\ The element $\overline{\sigma'_0}$
thus gives a section
  $f$ of $ \cO_X(2)$ and we conclude  that  if $\overline {\sigma'_0}$ is general enough, the zero-locus of  $\overline {\sigma_0'}$
 is equal to $S^{[2]}\cap U$, where
$S\subset X$ is the K3 surface defined by $f$.

Moreover, the open subset $S^{[2]}\cap U$ is then dense in $S^{[2]}$
and we thus proved that the central fiber of $\cK^0$ has one reduced component which is birationally isomorphic to $S^{[2]}$.\
By \cite{KLSV}, it  follows that after base change $\Delta'\to \Delta$ and shrinking, there exists a family
$\pi'\colon\cK'\to \Delta'$ that is fiberwise birationally isomorphic to $\cK^0\times_\Delta\Delta'$, all of whose   fibers
are smooth \hk\ fourfolds, with (smooth) central fiber   birationally isomorphic    to $S^{[2]}$.\ Since $\SS$ has no nontrivial \hk\ birational models (Section \ref{sec33}),  the central fiber  is in fact isomorphic    to $S^{[2]}$.

The varieties $\cK_t$, for $t$ very general, have Picard number~$1$, hence no nontrivial smooth \hk\  birational models.\ It follows that $\cK'_t\cong \cK_t$ and this holds for all $t\ne 0$.
\end{proof}

\begin{rema} {\rm From the viewpoint of subvarieties of $\Gr(6,V_{10})$, the situation is not completely explained.\ The varieties
$\cK_t$ are smooth subvarieties of  $\Gr(6,V_{10})$ of degree $1452$.\ The variety $S^{[2]}$ is mapped to $\Gr(6,V_{10})$ via
 the rational map $\phi$  described in Proposition~\ref{propbirat}, but since this map
is not regular, its image $\phi(S^{[2]})\subset \Gr(6,V_{10})$ has degree $<1452$.\ The limit (in the Hilbert scheme) of the subvarieties
 $\cK_t\subset \Gr(6,V_{10})$ must therefore have another irreducible component.
}
\end{rema}

\section{The  HLS divisor $\cD_2$}\label{div2}

We   describe a polystable point in the moduli space
$\cMDV=\PP(\bw3  V_{10}^\dual) \gquot \SL(V_{10})$ whose total image by the moduli map
$$\gm\colon \cMDV\dra \cM$$
is the divisor whose general points are the fourfolds $\cM_S(0,L,1)$ described in Remark~\ref{rema33}, where $(S,L)$ is a general polarized K3 surface of degree~$2$.\ As explained in Section~\ref{secrappelonpic}, this divisor
is therefore the Heegner divisor $\cD_2$.

 \subsection{The $\SL(3)$-invariant trivector}\label{trivela}

 We take  $V_{10}:=\Sym^3 \!W_3$.\ The $\SL(W_3)$-representation
 $\bw3V_{10}^\vee$ decomposes  as
\begin{equation}\label{deco}
\bw3V_{10}^\vee=\bw3(\Sym^3 \!W_3^{\vee}) =  \Gamma_{0,6}\oplus \Gamma_{3,3}\oplus \Gamma_{2,2}\oplus \Gamma_{0,0},
\end{equation}
 where $\Gamma_{a,b}$ is the irreducible representation given by the kernel of the contraction map $\Sym^a\! W_3\otimes \Sym^b\! W_3^{\vee}\to \Sym^{a-1}\! W_3\otimes \Sym^{b-1} \!W_3^{\vee}$.\footnote{\rem{In the standard notation of \cite{bbki} explained in Section~\ref{orbetstab}, the representation  $ \Gamma_{a,b}$ is $ V_{a\omega_1+b\omega_2}$.}}\ The first term is $\Sym^6\!W_3 ^\vee=H^0(\P(W_3),\cO_{\P(W_3)}(6))$.\ The last term is the ($1$-dimensional) space of $\SL(W_3)$-invariants and we pick a generator  $\sigma_0$.

\begin{shownto}{long}
\green{
\subsection{Analogous results  for   the  variety of lines on a cubic fourfold}\label{anacub}
All of the results stated in Section~\ref{div2} have analogues valid for 
 the variety of lines on a cubic fourfold.\   In particular, the analogue of Theorem~\ref{degenera2}
has been  proved by van den Dries~\cite{vdd}, in a more precise form (in particular $m=1$ will do).
   In the present section we will recall those results.\ In particular 
  we will go over a modified version of van den Dries' proof that will be our model for the proof of Theorem~\ref{degenera2}.\ Our version is not as precise as van den Dries' but we manage to avoid some lenghty computations.\ The point is that in proving Theorem~\ref{degenera2}, we wish to avoid similar, and presumably  longer, explicit computations.

If  $X\subset\PP^5$ is a cubic fourfold,  let  $F(X)\subset \Gr(1,\PP^5))$ be the Hilbert scheme of lines in $X$.\  We recall that the scheme structure of $F(X)$ can be defined by viewing it as the subscheme of $\Gr(1,\PP^5)$ defined by the section of $\Sym^3\! U^{\vee}$ ($U$ is the tautological rank-$2$ bundle on   $\Gr(1,\PP^5))$ associated (up to scalars) with $X$.\ If $X$ is a smooth cubic fourfold, $F(X)$ is a hyperk\"ahler fourfold of $K3$-type, 
by Beauville and Donagi.\ As $X$ varies among smooth cubic fourfolds, the  $F(X)$ form   a locally complete family of polarized \hk\ fourfolds (the polarization is given by the Pl\"ucker ample generator) whose primitive $H^2$ lattice is of \emph{nonsplit type}    and discriminant $3$.\ Note that the primitive $H^2$ lattice of Debarre--Voisin fourfolds  is of nonsplit type  and discriminant $11$ (two notches more complex, in the series of nonsplit lattices, than discriminant $3$).

Let $V_6:=\Sym^2\! W_3$, where $W_3$ is a $3$-dimensional complex vector space.\ Let 
\begin{equation*}
{\mathsf V}:=\{[a^2] \mid  a\in W_3\moins\{0\}\},\qquad {\mathsf D}:=\{[ab] \mid 0\ne a,b\in W_3\}
\end{equation*}
be the $\PGL(W_3)$-invariant Veronese surface and the discriminant hypersurface in $\PP(V_6)$.\ We let $f_0\in\Sym^3\! V_6^{\vee}$ be an equation of ${\mathsf D}$ (that is,~$f_0$ is ``the'' discriminant).\ The
 $\SL(W_{3})$-representation  $\Sym^3\! V^{\vee}_6$ decomposes  as
follows:
\begin{equation}\label{decompongo}
\Sym^3(\Sym^2 \!W_{3}^{\vee})= \Gamma_{0,6}\oplus \Gamma_{3,3}\oplus  \Gamma_{0,0}.
\end{equation}
The trivial addend is generated by the discriminant $f_0$.\ Thus $\sigma_0$ is the analogue, in the world of Debarre--Voisin fourfolds, of  the discriminant $f_0$.

Since $\mathsf V$ is the singular locus of $\mathsf D$, the stabilizer of $[f_0]$ is  $\PGL(W_3)$; this is the analogue of Proposition~\ref{stabver}.

One also proves that $[f_0]$ is $\PGL(V_6)$-polystable (see~\cite[Lemma ~4.3]{lazagit}).

Next, let $0\ne f\in\Sym^3\! V_6^{\vee}$ be such that the intersection $C:=V(f)\cap {\mathsf V}$ is transverse.\ 
Identifying ${\mathsf V}$ with $\PP(W_3)$, the curve $C$ gets identified with a smooth sextic.\ Let   
\begin{equation}\label{eccoesse}
S\to\PP(W_3)
\end{equation}
be the double cover with branch curve $C$.\ In~\cite{vdd}, van den Dries proved that the family $\{F(V(f_0+t^2f)\}_{t\ne 0}$ can be filled at $0$ with  $\cM_S(0,h,1)$.

The first step in the proof is the description of $F({\mathsf D})$.

\begin{defi}\label{cordever}
Let $p\in \PP(W_{3}^{\vee})$, and let $H$ be a codimension-$1$ subspace of $\Sym^2(\Omega_{\PP(W_{3}^{\vee})})(p)$, where 
$\Omega_{\PP(W_{3}^{\vee})}(p)$ is the cotangent space of 
$\PP(W_{3}^{\vee})$ at $p$.\   Let $I(p,H)\subset \Sym^2 W_{3}$ be the subspace of elements $\phi$ which  vanish to order at least $2$ at $p$ (that is,~either they vanish with order $2$, or are zero) and belong to $H$.
\end{defi}

\begin{defi}\label{altrette}
Given    $(p,R)\in\PP (W_3^{\vee})\times \PP (W_3)$, let $J(p,R)\subset \Sym^2 W_{3}$ be the  set of  $\phi=ab$  where $a,b\in W_3$, $a(p)=0$ and $V(b)=R$.
\end{defi}

We let
\begin{eqnarray*}
{\bf I} & := & \{\PP(I(p,H)) \mid p\in \PP(W_{3}^{\vee}),\quad H\in \PP(\Sym^2(\Omega_{\PP(W_{3}^{\vee})})(p)^{\vee})\},\\
 {\bf J} & := & \{\PP(J(p,R)) \mid (p,R)\in \PP (W_3^{\vee})\times \PP (W_3)\}.
\end{eqnarray*}
As is easily checked,
\begin{equation}\label{rettedisc}
F({\mathsf D})={\bf I}\cup {\bf J}.
\end{equation}
 A general line in ${\bf J}$ is contained in the smooth locus of ${\mathsf D}$, hence $F({\mathsf D})$ is smooth at such a point.\ On the other hand, $F({\mathsf D})$ is nonreduced along ${\bf I}$.

Hence the central fiber of the degeneration $\{F(V(f_0+t^2f)\}_{t\in\aff}$ is not reduced, and therefore not good.\ We modify it as follows.

Let $\cZ:=\Bl_{{\mathsf V}\times\{0\}}(\PP(V_6)\times\aff)$, and let $\varphi\colon\cZ\to \PP(V_6)\times\aff$ be the structure map.\ 
Let $E:=\exc(\varphi)$; thus $E\to{\mathsf V}$ is a bundle of $3$-dimensional projective spaces.\ We view $\cZ\to\aff$ as a degeneration of $\PP(V_6)$, with central fiber 
$\Bl_{{\mathsf V}}(\PP(V_6))\cup E$.\ Let $\cY\subset\cZ$ be the strict transform of $V(f_0+t^2f)\subset \PP(V_6)\times\aff$ ($t$ is ``the'' affine coordinate on $\aff$).\ We have a projective map $\pi\colon\cY\to\aff$, with
\begin{equation*}
Y_t:=\pi^{-1}(t)\cong
\begin{cases}
V(f_0+t^2f) & \text{if $t\ne 0$,} \\
\Bl_{{\mathsf V}}({\mathsf D})\cup Q & \text{if $t=0$.}
\end{cases}
\end{equation*}
where $Q\subset E$ is a  bundle of quadric surfaces over ${\mathsf V}$, with  smooth fibers over ${\mathsf V}\setminus C$, and fibers of corank $1$ over $C$.

 Let $\Hilb_P(\cY/\aff)$ be the relative Hilbert scheme parametrizing subschemes of fibers $Y_t$ with Hilbert polynomial $P$ (with respect to a relatively ample line bundle on $\cY\to\aff$) equal to that of a line in $Y_t$ for $t\ne 0$.\ Let  $\rho\colon\Hilb_P(\cY/\aff)\to\aff$ be the structure map, and let  $\wt{F}(\cY)\subset \Hilb_P(\cY/\aff)$ be the schematic closure of $\rho^{-1}(\aff\setminus\{0\})$.\ We let  $\wt{F}(Y_0)$ be the fiber of $\wt{F}(\cY)\to\aff$ over $0$.
 
 We claim that there is an irreducible component of  $\wt{F}(Y_0)$ birationally isomorphic to $S^{[2]}$, where $S$ is the double cover in~\eqref{eccoesse}.\  In fact, let $R$ be a  line parametrized by a point of ${\bf I}\setminus{\bf J}$.\ Then $R$  intersects 
 ${\mathsf V}$ in two distinct points $x_1,x_2$.\ Let $\wt{R}\subset \Bl_{{\mathsf V}}({\mathsf D})$ be the strict transform of $R$, and let $\wt{R}\cap Q=\{\wt{x}_1,\wt{x}_2\}$.\ Then every subscheme of $Y_0$ given by 
  \begin{equation*}
\wt{R}\cup R_1\cup R_2,\qquad \wt{x}_i\in R_i\subset Q_{x_i},\quad R_i\in\Gr(1,E_{x_i})
\end{equation*}
belongs to $\wt{F}(Y_0)$.\ Moreover, by~\eqref{rettedisc}, such subschemes are parametrized by an \emph{open} subset $ \wt{\bf I}^0$ of the fiber of $\Hilb_P(\cY/\aff)\to\aff$ over $0$.\ Hence the closure of  $ \wt{\bf I}^0$ in  $\Hilb_P(\cY/\aff)$ (equivalently, in  $\wt{F}(\cY)$) is an irreducible component of $\wt{F}(Y_0)$; we denote it by $\wt{\bf I}$.\ Clearly, $\wt{\bf I}$ is birationally isomorphic to $S^{[2]}$.\ (The set of lines in $({\bf J}\setminus{\bf I})$ gives an open dense subset  of another irreducible component of $\wt{F}(Y_0)$, birationally isomorphic to $\PP(W^{\vee}_3)\times\PP(W_3)$.)

One proves that $\wt{F}(Y_0)$ is smooth at a general point of  $\wt{\bf I}$  as follows.\ Let $R\subset {\mathsf D}$ be a line parametrized by a point of ${\bf I}\setminus{\bf J}$, and keep notation as above.\ A scheme $C:=\wt{R}\cup R_1\cup R_2$ as above is locally a complete intersection in $Y_0$, hence there is a well-defined normal bundle $N_{C/Y_0}$.\ Since $\wt{\bf I}$ is an \emph{open} subset of the fiber of $\Hilb_P(\cY/\aff)\to\aff$ over $0$ in a neighborhood of $C$, it suffices to prove  
\begin{equation}\label{accauno}
H^1(C,N_{C/Y_0})=0.
\end{equation}
 Let $\wt{\mathsf D}:=\Bl_{{\mathsf V}}({\mathsf D})$.\ We have
\begin{equation*}
N_{C/Y_0}\vert_{\wt{R}}\cong N_{\wt{R}/\wt{\mathsf D}},\quad N_{C/Y_0}\vert_{R_i}\cong N_{R_i/Q_{x_i}}.
\end{equation*}
Since $H^1(R_i,N_{R_i/Q_{x_i}}(-1))=0$, in order to prove~\eqref{accauno} it suffices to prove that $H^1(\wt{R},N_{\wt{R}/\wt{\mathsf D}})=0$.\ 
The latter vanishing follows from the exact sequences
\begin{equation*}
0 \lra N_{\wt{R}/\wt{\mathsf D}}\lra N_{\wt{R}/\wt{\PP(V_6)}}\lra \cO_{\wt{\PP(V_6)}}(\wt{\mathsf D})\vert_{\wt{R}}\lra 0,
\end{equation*}
(we let $\wt{\PP(V_6)}:=\Bl_{{\mathsf V}}(\PP(V_6))$) and
\begin{equation*}
0 \lra  N_{\wt{R}/\wt{\PP(V_6)}}\lra \psi^{*}N_{R/\PP(V_6)}\lra \CC^2_{x_1}\oplus\CC_{x_2}^2 \lra 0.
\end{equation*}
(We let $\psi\colon\wt{R}\to R$ be the restriction of the map $\wt{\mathsf D}\to {\mathsf D}$.) In fact, since  $\deg \cO_{\wt{\PP(V_6)}}(\wt{\mathsf D})\vert_{\wt{R}}=-1$, the above two exact sequences  show that it suffices to prove that the map $H^0(\wt{R},\psi^{*}N_{R/\PP(V_6)})\lra\CC^2_{x_1}\oplus\CC_{x_2}^2$ is surjective.\ That is easily verified.

Since $\wt{\bf I}$ is birationally isomorphic to $S^{[2]}$ and  has multiplicity $1$ in $\wt{F}(Y_0)$,   the family $\{F(V(f_0+t^{2m}f)\}_{t\ne 0}$, for a suitable $m$, can be filled at $0$ with a \hk\ fourfold birationally isomorphic to $S^{[2]}$ by (the proof of)~\cite[Theorem~(0.1)]{KLSV}.

\emph{At this point I have a problem proving that the central fiber can be modified to be $X_S$.\ A deeper analysis of $\wt{F}(Y_0)$ should allow to prove the result directly (as in van den Dries), without invoking~\cite{KLSV}, and also to prove that $m=1$ will do.\ Do we want to do it?} 

}

\end{shownto}

This trivector $\sigma_0$   can  be constructed via the ``symbolic method''  as follows (thanks to Claudio Procesi).\ Choose a generator $\eta$ for $\bw3W_3$ and write $a\wedge b\wedge c=:\det(a,b,c)\eta$ for all $a,b,c\in W_3$.\ Then  $\sigma_0$ is the unique  trivector on $V_{10} $ such that 
\begin{equation*}
\forall x,y,z\in W_3\qquad \sigma_0(x^3,y^3,z^3)= \det(x,y,z)^3
\end{equation*}
(it is alternating  and  $\SL(W_3)$-invariant because it is so when the entries are cubes).\ Let $(x,y,z)$ be a basis for $W_3$ and
   write  $\alpha\in \Sym^3\! W_{3}$ as
 \begin{eqnarray}\label{coordsim}
 \alpha&=&\alpha_{300}x^3+\alpha_{030}y^3+\alpha_{003}z^3\nonumber \\
&&{}+3(\alpha_{210}x^2 y+\alpha_{102}x z^2+\alpha_{021}y^2 z+\alpha_{120} x y^2+\alpha_{201}x^2 z
+\alpha_{012}y z^2)+6\alpha_{111}xyz.
 \end{eqnarray}
A straightforward computation (umbral calculus) shows that
\begin{eqnarray}\label{gransomma}\nonumber
 \sigma_0(\alpha,\beta,\gamma)&=&\qquad\sum\limits_{\tau\in\cP}\eps(\tau)\alpha_{\tau(3,0,0)}  \beta_{\tau(0,3,0)}  \gamma_{\tau(0,0,3)}
 - 3\sum\limits_{\tau\in\cP}\eps(\tau)\alpha_{\tau(3,0,0)}  \beta_{\tau(0,2,1)}  \gamma_{\tau(0,1,2)}   \\\nonumber
&&{}- 3\sum\limits_{\tau\in\cP}\eps(\tau)\alpha_{\tau(0,3,0)}  \beta_{\tau(1,0,2)}  \gamma_{\tau(2,0,1)}
-  3\sum\limits_{\tau\in\cP}\eps(\tau)\alpha_{\tau(0,0,3)}  \beta_{\tau(2,1,0)}  \gamma_{\tau(1,2,0)}   \\
&&{}-3\sum\limits_{\tau\in\cP}\eps(\tau)\alpha_{\tau(2,1,0)}  \beta_{\tau(1,0,2)}  \gamma_{\tau(0,2,1)}
-3\sum\limits_{\tau\in\cP}\eps(\tau)\alpha_{\tau(1,2,0)}  \beta_{\tau(2,0,1)}  \gamma_{\tau(0,1,2)}    \\\nonumber
&&{}- 6\sum\limits_{\tau\in\cP}\eps(\tau)\alpha_{\tau(2,1,0)}  \beta_{\tau(0,1,2)}  \gamma_{\tau(1,1,1)}
- 6\sum\limits_{\tau\in\cP}\eps(\tau)\alpha_{\tau(1,0,2)}  \beta_{\tau(1,2,0)}  \gamma_{\tau(1,1,1)}  \\\nonumber
&&{}- 6\sum\limits_{\tau\in\cP}\eps(\tau)\alpha_{\tau(0,2,1)}  \beta_{\tau(2,0,1)}  \gamma_{\tau(1,1,1)}.
 \end{eqnarray}
In each sum above, $\cP$ denotes the permutation group of the relevant subset of the family of indices.\
In particular, we get the following.

\begin{lemm}\label{nonzero}
For each $r\in\{1,2,3\}$, let $x^{i_r}y^{j_r}z^{k_r}$ be a degree-$3$ monomial.\ Then
$$\sigma_0(x^{i_1}y^{j_1}z^{k_1},x^{i_2}y^{j_2}z^{k_2},x^{i_3}y^{j_3}z^{k_3})\ne0$$
 if and only if $i_1+i_2+i_3=j_1+j_2+j_3=k_1+k_2+k_3=3$  and not all monomials are equal to $xyz$.
 \end{lemm}

\subsection{The   hypersurface   $X_{\sigma_0}$.}\label{effezero}

The equation of the hypersurface $X_{\sigma_0}\subset \Gr(3,\Sym^3\! W_{3})$ defined in \eqref{eqxs} is given by~\eqref{gransomma}.\ More precisely, order the multiindices as in  
Table~\ref{enumero}
 \begin{table}[htb!]
\renewcommand\arraystretch{1.5}
\begin{tabular}{| c | c | c | c | c | c | c | c | c | c |}
\hline
 $ 3,0,0$  & $ 0,3,0$ &  $ 0,0,3$ &  $ 2,1,0$ & $ 1,0,2$ &  $ 0,2,1$ &  $ 1,2,0$ &
 $ 2,0,1$ & $ 0,1,2$ & $ 1,1,1$ \\
\hline
$ 0$  & $ 1$   & $ 2$ & $ 3$ & $ 4$ & $ 5$ & $ 6$ & $ 7$ &
$ 8$ & $ 9$ \\
\hline
\end{tabular}
\vspace{5mm}
\caption{Ordering of multiindices}\label{enumero}
\end{table}
and denote the corresponding Pl\"ucker coordinates on $\bw3(\Sym^3\! W_{3})$ by $q_{012},\ldots,q_{789}$; then  $X_{\sigma_0}$ is the intersection of  $\Gr(3,\Sym^3\! W_{3})$ with the hyperplane
\begin{equation}\label{tagliograss}
q_{012}- 3 (q_{058}+q_{147}+ q_{236}+q_{345}+q_{678})-6(q_{389}+q_{469}+q_{579})=0.
\end{equation}
 
 \subsubsection{The singular locus of $X_{\sigma_0}$} 
 
  We   show in this section that the hypersurface $X_{\sigma_0}$  is singular along a surface which we first describe.\
Let
\begin{eqnarray*}
v_3\colon\PP(W_3) & \lhra & \PP(\Sym^3\! W_3) \\
{[}x] & \longmapsto & [x^3] \nonumber
\end{eqnarray*}
be the Veronese embedding and let ${\mathsf V}\subset\PP(\Sym^3\! W_3)$ be its image.\ The projective tangent space to ${\mathsf V}$ at $[x^3]$
 is $\PP(x^2\cdot W_3)$, hence  the Gauss map of ${\mathsf V}$ is
\begin{eqnarray}\label{mappagauss}
{\mathsf g}\colon {\mathsf V} &  \lhra  & \Gr(3,\Sym^3\! W_3) \\
{[}x^3] & \longmapsto &  [x^2\cdot W_3].     \nonumber
\end{eqnarray}
We have ${\mathsf g}^{*}\cO_{\Gr}(1)\cong \cO_{\PP(W_3)}(6)$ and ${\mathsf g}$ induces an isomorphism
\begin{equation}\label{isomo}
H^0(\Gr(3,\Sym^3\! W_3),\cO_{\Gr}(1))\isomlra H^0(\PP(W_3),\cO_{\PP(W_3)}(6)),
\end{equation}
because the left side
 is a nonzero    $\SL(W_3)$-invariant linear subspace of the right side.\

\begin{prop}\label{singeffe}
The singular locus of $X_{\sigma_0}$ is equal to   the  surface $\gv$.
\end{prop}

 \begin{proof}
We first prove one inclusion.

\begin{lemm}\label{monsing}
Let  $(x,y,z)$ be  a basis of $W_{3}$  and let  $U_3\subset \Sym^3 \!W_3$ be a $3$-dimensional subspace spanned by monomials in  $x,y,z$.\ Then $[U_3]$  is a singular point of $X_{\sigma_0}$ if and only if, after possibly renaming   $x, y, z$,  we have $U_3=\la x^3,x^2y,x^2z\ra$, that is,~$[ U_3]\in \gv$.

In particular, the surface $\gv$  is contained in the singular locus of $X_{\sigma_0}$.
\end{lemm}

\begin{proof}
Let $U_3=\la m_1,m_2,m_3\ra$, where $m_1,m_2,m_3$ are  monomials  in $x,y,z$.\  By \cite[Proposition~3.1]{devo},  the point $[U_3]$
 is  singular  on $X_{\sigma_0}$ if and only if  ${\sigma_0}(m_r\wedge m_s\wedge m)=0$  for every   distinct $r,s\in\{1,2,3\}$ and every monomial $m$  in $x,y,z$.\ Since $m$ is arbitrary, it follows from Lemma~\ref{nonzero} that at least one of the
 following inequalities holds
 \begin{equation*}
i_r+i_s>3,\quad j_r+j_s>3,\quad k_r+k_s>3.
\end{equation*}
The above is true for any choice of distinct $r,s\in\{1,2,3\}$.\ It follows  that, after possibly renaming   $x,y,z$, we have $U_3=\la x^3,x^2y,x^2z\ra$.
\end{proof}

We identify   $\P(V_{10})=\P(\Sym^3\! W_3)$ with $|\cO_{\P(W_3^{\vee})}(3)|$, the linear system of cubic curves in $\PP(W^{\vee}_3)$.\ Given $[\phi]\in \PP(\Sym^3\! W_3)$, we denote by $V(\phi)\subset \P(W_3^{\vee})$ its zero-locus
and, given
 a vector subspace $U\subset \Sym^3 \!W_3$, we let
 \begin{equation}\label{sistlin}
 \sL(U):=\{V(\phi) \mid [\phi]\in\PP(U)\}\subset |\cO_{\P(W_3^{\vee})}(3)|
 \end{equation}
be the associated linear system.

\begin{lemm}\label{primopasso}
Let $U_3\subset \Sym^3 \!W_3$ be a $3$-dimensional subspace.\ Suppose that  one of the following holds:
\begin{itemize}
\item[{\rm (a)}]
there exists $[\phi]\in \PP(U_3)$ such that $V(\phi)$ is  singular at a point $p\in \P(W_3^{\vee})$ not contained in the base-locus of $\sL(U_3)$;
\item[{\rm (b)}]
there exists an element of $\sL(U_3)$ with an ordinary node.
\end{itemize}
Then $[U_3]$ is not a singular point of $X_{\sigma_0}$.
\end{lemm}

\begin{proof}
Assume that  $[U_3]$ is  a singular point of $X_{\sigma_0}$.\  We will reach a contradiction  in   both cases.\
Suppose that (a) holds.\  Let $(x,y,z)$ be a basis of $W_3$ such that $p=(0,0,1)$.\ Then
$\phi=f_2(x,y)z+f_3(x,y)$, where $f_2,f_3$ are homogeneous of respective degrees $2$ and $3$, not both zero.\ By assumption, there exists  $[\psi]\in\PP(U_3)$ such that $p\notin V(\psi)$.\ Thus $\psi=z^3+f_1(x,y)z^2+f_2(x,y)z+f_3(x,y)$.\ Let $\lambda$ be the  $1$-parameter subgroup of $\GL(W_3)$ given (in the chosen basis) by
\begin{equation}\label{primops}
\lambda(t)=\diag(t^{n+1},t^{n},1),\quad n\ge 3.
\end{equation}
Let  $\ov{U}_3 :=\lim_{t\to 0}\lambda(t)U_3$.\ The hypersurface  $X_{\sigma_0}$ is mapped to itself by $\SL(W_3)$, hence it is singular  at $\lambda(t)U_3$ for all $t\in\CC^\star$,
  hence also at $\ov{U}_3$.\  A simple computation shows that if $f_2\ne 0$, then
 $\lim_{t\to 0}\lambda(t)[\phi]=[x^i y^jz]$, where $x^i y^j$ is the monomial with highest power of $y$ appearing in $f_2$,  and that if $f_2=0$, then $\lim_{t\to 0}\lambda(t)[\phi]=[x^i y^j]$,  where $x^i y^j$ is the monomial with highest power of $y$ appearing in  $f_3$.\ 
 
 On the other hand,   $\lim_{t\to 0}\lambda(t)[\psi]=[z^3]$.\ The subspace $\ov{U}_3$ is  generated by monomials in $x,y,z$, because the weights of the action of $\lambda$ on $\Sym^3\! W_3$ are pairwise distinct.\ Thus $\ov{U}_3$  is  generated by monomials in $x,y,z$ and  contains $z^3$ and one of $x^i y^jz$, $x^i y^j$.\ By Lemma~\ref{monsing},  $\ov{U}_3$  is not contained in $\sing (X_{\sigma_0})$.\ This is a contradiction.

Suppose now that (b) holds.\ By assumption, there exist a basis $(x,y,z)$  of $W_3$ and $[\phi]\in\PP(U_3)$ such that $\phi=xyz+f_3(x,y)$.\ Let $\lambda$ be the $1$-parameter subgroup in~\eqref{primops} and
set  $\ov{U}_3 :=\lim_{t\to 0}\lambda(t)U_3$.\ Arguing as above, we get that   $X_{\sigma_0}$ is  singular  at  $\ov{U}_3$.\   A simple computation shows that $\lim_{t\to 0}\lambda(t)[xyz+f_3(x,y)]=[xyz]$.\ Since $\ov{U}_3$ is  generated by monomials in $x,y,z$, this contradicts Lemma~\ref{monsing}.
\end{proof}

We now prove the reverse inclusion $\Sing( X_{\sigma_0})\subset \gv$.\
Let  $[U_3]\in \Sing( X_{\sigma_0})$.\ One of the following holds:
\begin{enumerate}
\item[(a)]
there exists $[\phi]\in \PP(U_3)$ such that $V(\phi)$ is  singular at a point not contained in the base-locus of $\sL(U_3)$;
\item[(b)]
the  base-locus of $\sL(U_3)$ is zero-dimensional and all curves in $\sL(U_3)$ are smooth outside the base-locus;
\item[(c)]
the  base-locus of $\sL(U_3)$ is one-dimensional and all curves in $\sL(U_3)$ are smooth outside the base-locus.
\end{enumerate}
If~(a) holds,   $[U_3]$ is not a singular point of $X_{\sigma_0}$  by Lemma~\ref{primopasso}.\ This is a contradiction.

Suppose that (b) holds.\ We claim that there exists $p\in\PP(W_3^{\vee})$ such that all elements of~$\sL(U_3)$ are singular at $p$.\ The set
\begin{equation*}
\Sigma:=\{(p,[\phi])\in \PP(W_3^{\vee})\times\sL(U_3) \mid \textnormal{$p$ is a singular point of $V(\phi)$}\} 
\end{equation*}
  is the intersection of $3$ divisors in $|\cO_{\PP(W_3^{\vee})}(2)\boxtimes\cO_{\sL(U_3)}(1)|$.\ If  $\Sigma$ has (pure) dimension 1, its projection to $\PP(W_3^{\vee})$ is a sextic curve, which contradicts (b).\ Hence $\dim(\Sigma)>1$ and   there exists a point $p$  such that all curves in $\sL(U_3)$ are singular at $p$.\ By  Lemma~\ref{primopasso}(b), no element of  $\sL(U_3)$ has an ordinary node at $p$.\ It follows that there are  linearly independent $[\phi_1],[\phi_2]\in\PP(U_3)$ such that $V(\phi_1)$ and $V(\phi_2)$  have multiplicity $3$ at $p$.\ Thus, there exists a nonzero  linear combination $c_1\phi_1+c_2\phi_2$ such that
 $V(c_1\phi_1+c_2\phi_2)$ is singular along a line.\ This contradicts our assumption~(b).

Lastly, suppose that (c) holds.\ The base-locus of $\sL(U_3)$ is either a line or a conic (possibly degenerate).\ Assume that it is a line $R$.
By Lemma~\ref{primopasso}(b), no element of $\sL(U_3)$ has an ordinary node.\  This forces $\sL(U_3)$ to be $R+\sL_0$, where
 $\sL_0\subset |\cO_{\PP(W_3^{\vee})}(2)|$ is one of the following:
\begin{enumerate}
\item[($\alpha$)]
the linear system of conics tangent to $R$ at a fixed $p\in R$ and containing a fixed
 $q\in \P(W_3^{\vee})\setminus R$;
\item[($\beta$)]
the linear system of conics with multiplicity of intersection at least $3$ with a fixed smooth conic tangent to $R$ at a fixed $p\in R$;
\item[($\gamma$)]
the linear system of conics singular  at a fixed $p\in R$.
\end{enumerate}
 If~($\alpha$) holds, there exists a basis $(x,y,z)$ of $W_3$ such that $U_3=\la x^2y,xy^2,y^2z\ra$.\ This     contradicts
   Lemma~\ref{monsing}.

  If~($\beta$) holds,  there exists a basis $(x,y,z)$ of $W_3$ such that $U_3=\la x^3+y^2z,xy^2,y^3\ra$.\ Let $\lambda$ be the $1$-parameter subgroup of $\GL(W_3)$ given by $\lambda(t)=\diag(t^{-1},t^{-3},1)$.\ Then  $\lim_{t\to 0}U_3=\la x^3,xy^2,y^3\ra$, which contradicts
   Lemma~\ref{monsing}.

  If~($\gamma$) holds,  there exists a basis $(x,y,z)$ of $W_3$ such that $U_3=\la x^2y,xy^2,y^3\ra$  and this contradicts Lemma~\ref{monsing}.\ 
  
  This proves that the base-locus of $\sL(U_3)$ is not  a line, hence it is a conic.\ If the conic has rank at least $2$, there are elements of $\sL(U_3)$ with an ordinary node  and this contradicts Lemma~\ref{primopasso}.\ Hence the base-locus of $\sL(U_3)$ is a double line, that is,~$[U_3]\in\gv$.
\end{proof}

\subsubsection{The   germ of $X_{\sigma_0}$ at its singular points} 

The local structure  of $X_{\sigma_0}$ at its singular points will be needed in the proof of Theorem~\ref{degenera2}.

\begin{lemm}\label{normeffe}
Let $p$ be a singular point of $X_{\sigma_0}$.\ The (analytic) germ $(X_{\sigma_0},p)$ is isomorphic to the germ
$
\bigl(\aff^2\times \bigl(\sum_{i=1}^{19}\xi_i^2=0\bigr),0\bigr)$.
\end{lemm}

\begin{proof}
Let $p:=[U_3]$ and let $(x,y,z)$ be a basis of $W_3$ such that $U_3=\la x^3,x^2y,x^2z\ra$.\ We write a local equation of $X_{\sigma_0}$ in a neighborhood of 
$p$, adopting the notation in Sections~\ref{trivela} and~\ref{effezero}.\ In particular, coordinates on $\Sym^3 \!W_3$ are defined by~\eqref{coordsim} and we order them as in Table~\ref{enumero}.\ 
Now~$p$ has coordinates $q_{037}=1$ and $q_{ijk}=0$ for $\{i,j,k\}\ne \{0,3,7\}$.\ 
Affine coordinates on the open  subset
$$ \Gr(3,\Sym^3\! W_3)_{q_{037}}\subset \Gr(3,\Sym^3\! W_3)$$
  defined by $q_{037}\ne 0$  are given by $q'_{ijk}:=q_{ijk}/q_{037}$ for all $0\le i<j<k\le 9$ such that exactly two of the indices $i,j,k$ belong to 
$\{0,3,7\}$.\  By~\eqref{tagliograss},  $X_{\sigma_0}\cap  \Gr(3,\Sym^3\! W_3)_{Q_{037}}$ has equation 
\begin{eqnarray*}
0&=&q'_{013}q'_{027}-q'_{017}q'_{023}-3(q'_{035}q'_{078}-q'_{038}q'_{057}+q'_{017}q'_{347}+q'_{047}q'_{137}-q'_{023}q'_{367})\\
&&{}+3(q'_{036}q'_{237}-q'_{034}q'_{357}+q'_{035}q'_{347}
+q'_{067}q'_{378}-q'_{078}q'_{367})+\textnormal{cubic term.}
\nonumber
\nonumber
\end{eqnarray*}
The tangent cone of $X_{\sigma_0}$ at $p$ is defined by the vanishing (in $\C^{21}$) of this quadratic term.\ 
A computation gives  
\begin{equation*}
T_{\gv,p}=\left\la \frac{\partial}{\partial q'_{039}}+2\frac{\partial}{\partial q'_{067}},\, 2\frac{\partial}{\partial q'_{034}}+
\frac{\partial}{\partial q'_{079}}\right\ra.
\end{equation*}
Another computation shows  
\begin{equation*}
T_{\gv,p}=\Ker(\phi).
\end{equation*}
This proves the lemma.
\end{proof}

\subsection{The variety $K_{\sigma_0}$}\label{se73n}

We describe in Proposition~\ref{kappapi} the Debarre--Voisin variety $K_{\sigma_0}$ associated with the trivector $\sigma_0  $ on $V_{10}=\Sym^3 \!W_3 $ defined in Section~\ref{trivela}.

\begin{shownto}{long}
 \green{

\begin{lemm}\label{ellepih}
Let $p$ and $H$  be as above.\ We consider $H$ as a hyperplane in the linear system $|\cO_{\PP(T_{\PP(W_3^{\vee}),p})}(2)|$.\ Then $L(p,H)$ is a $6$-dimensional subspace of $\Sym^3 \!W_3$ and there exists a basis
 $(x,y,z)$  of $W_{3}$ such that 
\begin{equation}\label{monelle}
L(p,H)=
\begin{cases}
 \la x y^2,x z^2,y^3,y^2 z, y z^2,z^3\ra=y^2\cdot W_3+z^2\cdot W_3 & \text{if $H$ is base-point free,} \\
  \la x y^2,x y z,y^3,y^2 z, y z^2,z^3\ra  & \text{if $H$ has a base-point.}
  \end{cases}
\end{equation}
Moreover $L(p,H)\in K_{\sigma_0}$.
\end{lemm}
}
\end{shownto} 
%
%

\begin{shownto}{long}
\green{
\begin{rema}
Let $p$ and $H$ be as above.\ A nonzero  $\phi$ is in $ L(p,H)$ if and only if  the associated divisor $V(\phi)\in|\cO_{\PP(W_{3}^{\vee})}(3)|$ vanishes to order at least $2$ at $p$, and  the class of $\phi$ in $\gm_{p}^2/\gm_{p}^3$ (which is determined up to  $\CC^\star$) belongs to $H$.
\end{rema}
}
\end{shownto}

 \subsubsection{Two distinguished subvarieties of $K_{\sigma_0}$}

\begin{defi}\label{eccolph}
(a) Given $[a]\in \PP(W_{3}^{\vee})$ and a codimension $1$ suspace   $H\subset\Sym^2(a^{\bot})$, let
\begin{eqnarray*}
I(a,H) & := & \textnormal{image of $H$ via the inclusion $(\Sym^2(a^{\bot})\hra \Sym^2\! W_{3})$,} \\
L(a,H) & := & (a\cdot I(a,H)^{\bot})^{\bot}\subset \Sym^3\! W_{3}.
\end{eqnarray*}
Note that $\dim (I(a,H))=2$ and $\dim (L(a,H))=6$.

\noindent(b) 
Given $[a]\in \PP(W_{3}^{\vee})$ and  
 $[x]\in \PP (W_3)$, let
 \begin{eqnarray*}
 J(a,x) & := & x\cdot \Ker(a) \subset \Sym^2\! W_{3}, \\
M(a,x) & := & (a\cdot J(a,x)^{\bot})^{\bot}\subset \Sym^3\! W_{3}.
\end{eqnarray*}
Note that $\dim (J(a,x))=2$ and $\dim (M(a,x))=6$.

\noindent (c) Finally, define two irreducible subvarieties of $ \Gr(6,V_{10})$ by setting
\begin{eqnarray*}
K_L & := & \{[L(a,H)] \mid [a]\in \PP(W_3^{\vee}),\   H \subset \Sym^2(a^{\bot})\textnormal{ hyperplane}\},\\
K_M & := & \{[M(a,x)] \mid [a]\in \PP (W_3^{\vee}),\ [x]\in \PP (W_3)\}.
\end{eqnarray*}
\end{defi}
We list the subspaces $a\cdot I(a,H)^{\bot}$ and $a\cdot J(a,x)^{\bot}$ up to isomorphism.\ First notice that there exist linearly independent $x,y\in W_3$ such that  $H=\la x^2,y^2\ra$ or $H=\la x^2,xy\ra$.\ 
As is easily checked, there exists   a basis $(a,b,c)$ of $W_3^{\vee}$ such that
\begin{equation}\label{esplicito}
\begin{aligned}
a\cdot I(a,H)^{\bot}&=
\begin{cases}
a\cdot\la a^2,ab,ac,bc\ra & \text{if $H=\la x^2,y^2\ra$,} \\
a\cdot\la a^2,ab,ac,c^2\ra & \text{if $H=\la x^2,xy\ra$,} 
\end{cases}
\\
a\cdot J(a,H)^{\bot}&=
\begin{cases}
a\cdot\la a^2,b^2,bc,c^2\ra & \text{if $a(x)\not=0$,} \\
a\cdot\la a^2,ab,ac,c^2\ra & \text{if $a(x)=0$.} 
\end{cases}
\end{aligned}
\end{equation}
We now show that the varieties $K_L$ and $K_M$ are both contained in $K_{\sigma_0}$.

\begin{prop}\label{prop710}
{\rm{(a)}} The subvariety of $K_{\sigma_0}$ obtained from the surface $\gv \subset \Sing(X_{\sigma_0} )$ by the procedure described in Proposition~\ref{remaxsing}(b) is $K_L$.

\noindent {\rm{(b)}} The variety $K_M$ is contained in $K_{\sigma_0}$.
\end{prop}

\begin{proof}
By Proposition~\ref{remaxsing}(b), for $x,y\in W_3$ not collinear, the $6$-dimensional subspace  $x^2\cdot W_3+y^2\cdot W_3\subset \Sym^3\! W_{3}$ corresponds to a point of $K_{\sigma_0}$.\ This is exactly $L(a,H)$, where $a^\perp=\la x,y\ra $   and $H=\la x^2,y^2\ra$.\ Since $K_L$ is irreducible of dimension at most $4$, this proves~(a).

By~\eqref{esplicito}, if $a(x)\ne 0$, then  $M(a,x)=
 \la x^2 y,x^2 z,y^3,y^2 z, y z^2,z^3 \ra $  in a suitable basis  $(x,y,z)$  of~$W_{3}$.\ By Lemma~\ref{nonzero}, this is a point of  $K_{\sigma_0}$, which proves (b).
\end{proof}

The rest of Section~\ref{se73n} will be devoted to the proof of the following result.

\begin{prop}\label{kappapi}
One has $(K_{\sigma_0})_{\rm red} =K_L\cup K_M$.
\end{prop}

We also mention as an addition to  this statement that $K_{\sigma_0}$ is nonreduced along its component~$K_L$: this follows from  Propositions~\ref{prop710}(a) and \ref{remaxsing}(a).

 The following remark (which complements the description of $K_L$ in  Proposition~\ref{prop710}(a))
 will be useful in the proof of Theorem~\ref{degenera2}.
 
\begin{rema}\label{interver}
If $[U_6]\in K_{\sigma_0}$, one of the following holds:
\begin{itemize}
\item[{\rm (a)}]
either $[U_6]\in K_L \setminus  K_M$ and  the scheme-theoretic intersection $\Gr(3,U_6)\cap \gv$  is the union of two reduced (distinct) points;  
\item[{\rm (b)}]
or $[U_6]\in K_M \setminus  K_L $ and  $\Gr(3,U_6)\cap  \gv=\emptyset$;
\item[{\rm (c)}]
or $[U_6]\in K_L \cap  K_M $ and   the scheme-theoretic intersection $\Gr(3,U_6)\cap \gv$ has length $2$.
\end{itemize}
\end{rema}

\begin{shownto}{long}
\green{Arguing as in the proof of Lemma~\ref{ellepih}, one gets the following result.
}
\end{shownto}

\begin{shownto}{long}
\green{
\begin{lemm}\label{emmepir}
For every $(p,R)\in \PP(W_{3}^{\vee})\times \PP(W_{3}) $,  the point $[M(p,R)]$ of $\P(\Sym^3\! W_{3})$ is in $ K_{\sigma_0}$ and there exists a basis
 $(x,y,z)$  of $W_{3}$ such that 
\begin{equation}\label{hersch}
M(p,R)=
\begin{cases}
 \la x^2 y,x^2 z,y^3,y^2 z, y z^2,z^3 \ra & \text{if $p\notin R$,} \\
 \la x y^2,x y z,y^3,y^2 z, y z^2,z^3\ra  & \text{if $p\in R$.}
 \end{cases}
\end{equation}
\end{lemm}
}
\end{shownto} 

\begin{shownto}{long}
\green{\begin{rema}
Let $p$ and $R$ be as above.\ Then  $M(p,R)\subset \Sym^3\! W_{3}$ is the set of elements $\phi$  such that the polar of 
$\phi$ at $p$ vanishes on $R$ and on $p$.
\end{rema}

\begin{rema}
Let $p$, $H$, and $R$ be as above, where $p\in R$ and  $H$ 
 is the $1$-dimensional linear subsystem of $|\cO(2)|$ on $\PP(T_{\PP(W_3^{\vee})}(p))$ with base-point $R$.\ Then $L(p,H)=M(p,R)$.\  The intersection ${\bf L}\cap{\bf M}$ consists of the set of $L(p,H)=M(p,R)$ as above.
\end{rema}
}
\end{shownto}

\begin{rema}\label{rema710}
 Let $F_I\subset \Gr(2,\Sym^2\! W_3)$ be the set of all  $I(a,H)$ and let $F_J\subset \Gr(2,\Sym^2\! W_3)$ be 
the set of all  $J(a,x)$.\ The variety of lines on the chordal cubic in $\PP(\Sym^2\! W_{3})$ is equal to $F_I\cup F_J$, both $F_I$ and $F_J$ are smooth of dimension $4$, and their intersection is smooth of dimension~$3$
 (\cite[Proposition~3.2.4]{vdd}).\ Thus, by Proposition~\ref{kappapi},  $K_{\sigma_0}$ is isomorphic to 
 the variety of lines on the chordal cubic.
\end{rema}

\subsubsection{Elements of  $K_{\sigma_0}$ and $2$-jets}
Considering the definition of $K_L$ and $K_M$, we must, in order to prove Proposition~\ref{kappapi}, examine $U_6^{\bot}$ when $[U_6]\in K_{\sigma_0}$.\ We prove in~Proposition~\ref{crit2jet} that  $U_6^{\bot}$ satisfies a very strong  condition.

\begin{lemm}\label{stab2jet}
Let $U_3\subset \Sym^3 \!W_3=H^0(\PP(W_3^{\vee}),\cO_{\PP(W_3^{\vee})}(3))$ be a $3$-dimensional subspace.\ Suppose that  there exists $p\in\PP(W_3^{\vee})$  such that 
 $U_3\subset H^0(\PP(W_3^{\vee}),\gm_p^2(3))$ and   the natural map 
 $U_3\to (\gm_p^2/ \gm_p^3)\otimes \cO_{\PP(W_3^{\vee})}(3)$ is an isomorphism.\
Then $[U_3]\notin X_{\sigma_0}$.
\end{lemm}

\begin{proof}
We proceed by contradiction.\ Assume   $[U_3] \in X_{\sigma_0}$ and let $(x,y,z)$ be a basis of $W_3$ such that  the   coordinates of $p$ are $(0,0,1)$.\ Let $r$ and $s$ be integers such that $ \frac{3}{2}s>r>s>0$ and let~$\lambda$ be the $1$-parameter subgroup of $\GL(W_3)$ given (in the chosen basis) by
\begin{equation*}\label{abps}
\lambda(t)=\diag(t^{r},t^{s},1).
\end{equation*}
Let  $\ov{U}_3 :=\lim_{t\to 0}\lambda(t)U_3$.\ Then  $X_{\sigma_0}$ contains  $[\ov{U}_3]$, because it is mapped to itself by $\GL(W_3)$.\  The representation $\Sym^3\! \lambda\colon\CC^\star\to \Sym^3\!  W_3$ has isotypic components of dimension $1$.\  Generators of the isotypic components, ordered in increasing order, are 
\begin{equation*}
z^3,\ yz^2,\ xz^2,\ y^2z,\ xyz,\ x^2z,\ y^3,\ xy^2,\ x^2y,\ x^3.
\end{equation*}
It follows that $\ov{U}_3 =\la x^2z,\ xyz,\ y^2z\ra$.\ By Lemma~\ref{nonzero}, one gets $[\ov{U}_3]\notin X_{\sigma_0}$, a contradiction.
\end{proof}

\begin{prop}\label{crit2jet}
Let $[U_6]\in K_{\sigma_0}$.\ For every $[a]\in \PP(W_3^{\vee})$, we have
\begin{equation}\label{nontras}
(a\cdot \Sym^2\! W_3^{\vee})\cap U_6^{\bot}\ne \{0\}.
\end{equation}
\end{prop}

\begin{proof}
We view $U_6$ as a subspace of $H^0(\PP(W_3^{\vee}),\cO_{\PP(W_3^{\vee})}(3))$.\ Let $p\in\PP(W_3^{\vee})$.\ If the natural map
\begin{equation}\label{ginrit}
U_6\lra (\cO_{\PP(W_3^{\vee}),p}/\gm_{p}^3 )\otimes \cO_{\PP(W_3^{\vee})}(3)
\end{equation}
is  surjective, or equivalently bijective since both spaces have dimension $6$,  the kernel of the map $U_6\to (\cO_{\PP(W_3^{\vee}),p}/\gm_{p}^2) \otimes \cO_{\PP(W_3^{\vee})}(3)$ is a $3$-dimensional subspace $U_3\subset U_6\cap H^0(\PP(W_3^{\vee}),\gm_{p}^2(3))$ such that the natural map $U_3\to (\gm_p^2/ \gm_p^3)\otimes \cO_{\PP(W_3^{\vee})}(3)$ is an isomorphism.\ By Lemma~\ref{stab2jet},  $[U_3]\notin X_{\sigma_0}$, but this is absurd because $[U_6]\in K_{\sigma_0}$.\ The map \eqref{ginrit} is therefore not surjective.

Assume first that $p=[a]$ is not in the base-locus of the linear system $\PP(U_6)$.\   The  map $\PP(W_3^{\vee})\dra  \PP(U_6^{\vee})$ defined by  $\PP(U_6)$ is the composition
\begin{equation*}
\PP(W_3^{\vee}) \overset{v_3}{\lra}
 \PP(\Sym^3\! W^{\vee}_3) \dra \PP(U_6^{\vee})
\end{equation*}
 of the Veronese map  $v_3$ and the projection with center $\PP(U_6^{\bot})$.\ If~\eqref{nontras} does not hold, the second-order osculating plane $\P(a\cdot \Sym^2\! W_3^{\vee})$ to the Veronese surface $v_3(\PP(W_3^{\vee}))$ does not meet the center of projection $\PP(U_6^{\bot})$, hence~\eqref{ginrit} is bijective, which we just prove does not hold.\ It follows that~\eqref{nontras} holds if 
$[a]$ is not in the base-locus of~$\PP(U_6)$.\ Since the property~\eqref{nontras} is closed, it  holds for all $[a]\in \P(W_3^{\vee})$.
\end{proof}

\subsubsection{Three-dimensional linear system of plane cubics containing many reducible cubics}
Let $[U_6]\in K_{\sigma_0}$.\ Then $\P(U^{\bot}_6)\subset \P(\Sym^3\! W_3^{\vee})$ is a $3$-dimensional linear systems of cubics in $\PP(W_3)$.\  By Proposition~\ref{crit2jet}, given any line $R\subset \PP(W_3)$, there exists a cubic in $\PP(U^{\bot}_6)$ containing~$R$.\ We prove here the following result.

\begin{prop}\label{disdop}
Let $\Lambda\subset|\cO_{\PP^2}(3)|$ be a $3$-dimensional linear system such that,  for each  line $R\subset \PP^2$, there exists a cubic in $\Lambda$ containing $R$.\ One of the following holds:  
\begin{itemize}
\item[{\rm (a)}]
the base-locus of $\Lambda$ contains a line;
\item[{\rm (b)}]
there exists a (possibly degenerate) conic $C\subset \PP^2$ such that $\Lambda$ contains $C+|\cO_{\PP^2}(1)|$, 
\item[{\rm (c)}]
in a suitable basis $(x,y,z)$ of $H^0(\PP^2,\cO_{\PP^2}(1))$, one of the following holds: 
\begin{itemize}
\item[{\rm (c1)}]
$\Lambda=\PP(\la x^3,y^3,z^3,xyz\ra)$,
\item[{\rm (c2)}]
$\Lambda\subset\PP(\la xz^2,yz^2,x^3,x^2y,xy^2,y^3\ra)$,
\item[{\rm (c3)}]
 $\Lambda=\PP(\la xyz,x^2y+yz^2,x^2z+y^2z,xy^2+xz^2\ra)$,
\item[{\rm (c4)}]
$\Lambda\subset\PP(\la xyz,x^2y+xz^2,xy^2+yz^2,x^2z,y^2z\ra)$,
\item[{\rm (c5)}]
$\Lambda\subset\PP(\la x^2z, xyz, xy^2+xz^2 ,y^2z, y^3+yz^2 \ra)$, 
\item[{\rm (c6)}]
$\Lambda\subset\PP(\la x^2z-xy^2, y^3, y^2z, yz^2, z^3 \ra)$. 
\end{itemize}
\end{itemize}
\end{prop}

Here is  the corollary   of interest to us.

\begin{coro}\label{tutti2jet}
Let $[U_6]\in K_{\sigma_0}$.\  One of the following holds:
\begin{itemize}
\item[\rm ($\alpha$)]
 $U_6^{\bot}=f_1\cdot U_4$, where  $ f_1\in W_3^{\vee}$ and   $U_4\subset\Sym^2\! W_3^{\vee}$ is a $4$-dimensional subspace;
\item[\rm ($\beta$)]
 $U_6^{\bot}\supseteq f_2\cdot W_3^{\vee}$, where $  f_2\in\Sym^2\! W_3^{\vee}$.

\end{itemize}
\end{coro}

\begin{proof}
As noted above, $\Lambda:=\PP(U^{\bot}_6)$ is a $3$-dimensional linear system of cubics satisfying the hypothesis of Proposition~\ref{disdop}.\ Hence one of items~(a), (b), (c) of that proposition holds.\ If (a) holds, then ($\alpha$) holds; if 
(b) holds, then ($\beta$) holds.\ One checks that if~(c) holds,  $[U_6]$ is not in $K_{\sigma_0}$.\ For example, suppose that~(c6) holds and let $(a,b,c)$ be the basis of $W_3$ dual to the basis $(x,y,z)$ of~$W_3^{\vee}$.\ Then $U_6\supset\la a^2c+ab^2,a^3,a^2b,abc,ac^2\ra$ and this is absurd, because 
$\sigma_0(a^2c+ab^2,abc,ac^2)\not=0$ by Lemma~\ref{nonzero}.
\end{proof}

Before proving Proposition~\ref{disdop}, we go through two preliminary results.\ The  first  is an easy exercise which we leave to the reader.

\begin{lemm}\label{tuttired}
Let $\Lambda\subset|\cO_{\PP^2}(3)|$ be a linear system all of whose elements are reducible.\ Then, either $\Lambda$ has a $1$-dimensional base-locus or  all cubics in $\Lambda$ have multiplicity $3$ at a fixed point.
\end{lemm}

\begin{prop}\label{quadsig}
Let $\Lambda\subset|\cO_{\PP^2}(3)|$ be  a $2$-dimensional linear system.\ 
Suppose that,  given an arbitrary  line $R\subset\PP^2$, there exists a cubic in $\Lambda$ containing $R$.\ Then, there exists a conic $C$ such that  $\Lambda=C+|\cO_{\PP^2}(1)|$.
\end{prop}

\begin{proof}
By our hypothesis, the variety of reducible cubics in $\Lambda$ has dimension $2$, hence every cubic in $\Lambda$ is reducible.\ Since all cubics in $\Lambda$ cannot have multiplicity $3$ at a fixed point, 
 Lemma~\ref{tuttired} implies that the   base-locus of $\Lambda$ contains a line $R$ or a conic $C$.\ If the latter holds, we are done because $\dim(\Lambda)=2$.\ If the former holds, 
$\Lambda=R+\Lambda'$, where $\Lambda'$ is a $2$-dimensional linear system of conics such that,  given any  line $R\subset\PP^2$, there exists a conic in $\Lambda'$ containing $R$.\ In particular all conics in $\Lambda'$ are reducible.\ It follows that  there exists a line $R'$ such that $\Lambda'=R'+|\cO_{\PP^2}(1)|$.\ Thus $\Lambda=(R+R')+|\cO_{\PP^2}(1)|$.
\end{proof}

\begin{proof}[Proof of Proposition~\ref{disdop}]
 If the base-locus of $\Lambda$ has dimension~$1$,  item~(a) holds.\ From  now on, we  assume that the base-locus of $\Lambda$ is finite.\ 
Let  $f\colon\PP^2\dra \Lambda^{\vee}\cong\PP^3$ be the natural map.\  Let $ B\subset\PP^2$ be  the (schematic) base-locus of $\Lambda$, so that $\Lambda\subset|\cI_{B}(3)|$.\ Let $f_{B}\colon\PP^2\dra  |\cI_{B}(3)|^{\vee}$ be the natural rational map.\ Then $f$ is the composition $\pi\circ f_{B}$, where  $\pi\colon |\cI_{B}(3)|^{\vee}\dra \Lambda^{\vee}$ is a projection whose center  does not intersect the (closed) image  $f_{B}(\P^2)$.

The (closed) image  $f(\P^2)$ is either a curve or  a surface.
If  it is a curve,   $\Lambda$ is the linear system of  cubics in $\PP^2$ which have multiplicity $3$ at a fixed point.\ This contradicts our hypothesis.\ 
 Hence~$f$ has finite positive degree onto the surface  
 $\Sigma:=f(\P^2)$.\ 
As one  easily checks,
\begin{enumerate}
\item[(i)]
either ${B}$ is  the complete intersection of a  (possibly degenerate) conic $C$ and a cubic,  
\item[(ii)]
or the restriction of  $f_{B}$ to a subscheme $Z\subset \P^2\setminus B$ of length $2$ is \emph{not} injective 
 if and only if the schematic intersection    $\la Z\ra \cap{B}$ has length $3$.
\end{enumerate}
If~(i) holds,  $\Lambda=|\cI_{B}(3)|$, hence $\Lambda\supset C+|\cO_{\PP^2}(1)|$.\ Thus item~(b) of  Proposition~\ref{disdop} holds.\ From now on, we assume that~(ii)  holds.

Assume first that $f$ has degree $1$ onto its image.\ If $R\subset\PP^2\setminus {B}$ is a line,   $f_{B}(R)$ is a twisted cubic by item~(ii).\ A dimension count shows that 
\begin{itemize}
\item[($\alpha$)]
either $f(R)$ is also a twisted cubic for a general line   $R\subset \PP^2\setminus {B}$,
\item[($\beta$)]
or
the projection $\pi\colon |\cI_{B}(3)|^{\vee}\dra \Lambda^{\vee}$ maps to the same point $f_{B}(R_1\setminus{B})$ and $f_{B}(R_2\setminus{B})$, where $R_1,R_2\subset\PP^2$ are distinct lines such that $\len(R_i\cap{B})=3$, 
\item[($\gamma$)]
or
the differential of $f$ vanishes at all points of $R\setminus{B}$, where $R\subset\PP^2$ is a line such that $\len(R\cap{B})=3$. 
\end{itemize}
If ($\alpha$) holds, no cubics in $\Lambda$ contain  a general line $R\subset\PP^2$, because $f(R)\subset\Lambda^{\vee}$ is a twisted cubic, and this contradicts the hypothesis of Proposition~\ref{disdop}.\ If ($\beta$) holds, $\dim(\Lambda)=4$,  $\len({B})=5$, and ${B}$ is a subscheme of $R_1\cup R_2$.\ It follows that $\Lambda\supset R_1+R_2+|\cO_{\PP^2}(1)|$, hence  
 item~(b) of  Proposition~\ref{disdop} holds.\ If ($\gamma$) holds,  $\Lambda\supset 2R+|\cO_{\PP^2}(1)|$ and  item~(b)  holds again.

Assume now that $f$ has degree greater than $1$ onto its image.\ Suppose that the surface $\Sigma\subset\Lambda^{\vee}$ has  degree $2$.\ Let $\wh{\PP}^2\to\PP^2$ be a smooth blow up such that $\wh{f}\colon \wh{\PP}^2\to\PP^2\stackrel{f}\dra \Sigma$ is a morphism.\ 
Let $V\subset\Sigma$ be the union of the set of singular points of the branch divisor of 
 $\wh{f} $ (this includes the points over which the fiber is not finite) and the vertex  of $\Sigma$ if $\Sigma$ is a cone.\ 
 
The linear system $\Lambda$  contains a $2$-dimensional   family of reducible cubics that contain a general line and these cubics correspond to planes in $ \Lambda^{\vee}\isom \P^3$ that either meet $V$ or are tangent to $\Sigma$ at a smooth point of $\Sigma$.\ If these planes all pass through a point of $V$, we can apply Proposition~\ref{quadsig} and item~(b)  holds.\ Otherwise,  given a general line $R\subset\PP^2$, there exists a plane   tangent to $\Sigma$ at a smooth point  such that the corresponding cubic   contains $R$.\ If $\Sigma$ is smooth, the cubics corresponding to tangent planes are of the form $C_1+C_2$, where $C_1$ and $C_2$ belong to two fixed pencils of curves corresponding to the two pencils of lines on $\Sigma$  and this is absurd because they do not contain a general line.\ If $\Sigma$ is a cone, the set of tangent planes is the linear system of planes through the vertex and we are reduced to the first case.\

We may therefore assume $\deg(\Sigma)\ge 3$.\ We claim that the (schematic) base-locus $B$ of $\Lambda$ is curvilinear.\ It is not,  there is a (single) point $p$ in the support of ${B}$ such that, in  a neighborhood of $p$, we have $\cI_{B}=\gm_p^2$.\ This implies $\deg( f)\deg(\Sigma)\le 5$, hence $\deg(\Sigma)=2$, which is absurd.

 Since $B$ is curvilinear, it is locally a complete intersection; therefore, $\deg( f)\deg (\Sigma)+\len(\cO_{B})=9$.\ Since $\deg (f)\ge 2$ and $\deg( \Sigma)\ge 3$, one of the following holds:
\begin{enumerate}
\item[(I)]
$ B$ is empty and $\deg( f)=\deg(\Sigma)=3$;
\item[(II)]
$ B$ is a single reduced point and $\deg( f)=2$;
\item[(III)]
$ B$ has length $3$ and $\deg( f)=2$.
\end{enumerate}
Suppose that~(I) holds.\ In particular, $f\colon \PP^2 \to\Sigma$ is regular.\ Let us show that item~(c1) of Proposition~\ref{disdop} holds.\ First, we claim that $\Sigma$ has isolated singularities.\ In fact,  if $\Sigma$ is a cone, one gets a contradiction arguing as in the proof that $\Sigma$ cannot be a quadric.\ If $\Sigma$ is a nonnormal cubic (and not a cone), its normalization $\wt{\Sigma}$ is the Hirzebruch surface $\F_1$ and we get a contradiction because the dominant map $\PP^2\to\Sigma$ lifts to a dominant map $\PP^2\to \F_1$, and $\rho(\F_1)>\rho(\PP^2)$.
 We have proved that $\Sigma$ has isolated singularities.\ 
 
 The map $f\colon\PP^2\to\Sigma$ is finite and $f^{*}\omega_{\Sigma}\equiv\omega_{\PP^2}$, hence $f$ is unramified in codimension $1$.\ Hence, if $C\in\Lambda$ is general, the map  $C\to f(C)$ is the quotient map for the action of a subgroup of $\Pic^0(C)$ of  order $3$.\ This action is the restriction of an automorphism $\varphi_C$ of $\PP^2$ of order $3$.\ We prove that  $\varphi_C$ does not depend on $C$.\ Let $C'\in\Lambda$ be another general cubic and let $H,H'\subset\Lambda^{\vee}$ be the planes corresponding to $C,C'$.\ The $9$ points in $C\cap C'$ are partitioned into the union of the three fibers (each of cardinality $3$) of the three points of intersection of the line $H\cap H'$ with~$\Sigma$.\ It follows that $\varphi_C$ and $\varphi_{C'}$ agree on the $9$ points in $C\cap C'$, hence are equal.\ The upshot is that there exists an order $3$ automorphism $\varphi$ of $\PP^2$ such that $f\colon\PP^2\to\Sigma$ is the corresponding quotient map  and $f^{*}\cO_{\Sigma}(1)\cong\cO_{\PP^2}(3)$.\ It follows that~(c1) holds.

Suppose that~(II) holds.\  Let $\wh{\PP}^2\to\PP^2$ be the blow up of the base point of $\Lambda$.\ Then $f$ induces a regular finite map 
$\wh{f}\colon \wh{\PP}^2\to \Sigma$ of degree $2$.\ Since the exceptional divisor of $\wh{\PP}^2\to\PP^2$ is the unique $(-1)$-curve of $\wh{\PP}$,  the covering involution  of $\wh{f}$  descends to an involution $\iota\colon\PP^2\to\PP^2$ leaving invariant the cubics in $\Lambda$.\  In suitable  coordinates, we have 
$\iota(x,y,z)=(x,y,-z)$.\ Since the cubics in $\Lambda$ are ${\iota}$-invariant, we have
 $\Lambda\subset\PP(\la xz^2,yz^2,x^3,x^2y,xy^2,y^3\ra)$ and~(c2) holds.

Suppose that~(III) holds.\  The blow up $\Bl_{B}\PP^2$  of $B$  is a weak Del Pezzo surface (the anticanonical bundle is big and nef) with DuVal singularities.\ The anticanonical system 
$|\cI_{B}(3)|$ defines a map $\Bl_{B}\PP^2\to |\cI_{B}(3)|^{\vee}\cong \P^6$ whose image is a Del Pezzo surface $ S$ with DuVal singularities.\ The rational map $f\colon\P^2\dra \Lambda^{\vee}$ is the composition of the natural rational map $\PP^2\dra { S}$ and the restriction to $ S$ of a projection $ |\cI_{B}(3)|^{\vee}\dra   \Lambda^{\vee}$ with center disjoint from ${ S}$.\ The latter is  
a map $\wh{f}\colon { S}\to \Sigma$ which is  finite, of degree $2$.\ If $\wh{\iota}\colon { S}\to { S}$ is its covering involution, $\Lambda$ is contained in the projectivization of the $\wh{\iota}$-invariant subspace of $H^0({ S},\omega^{-1}_{{ S}})$. 

If the involution $\wh{\iota}$ descends to a regular involution of $\PP^2$, item~(c2) holds by the argument   given above.\  Thus we assume that $\wh{\iota}$ is a birational nonregular involution of $\PP^2$; in particular,
$B$ is not contained in a line and there exist  coordinates $x,y,z$ such that 
\begin{itemize}
\item[(a)]
either
$|\cI_{B}(3)|=\PP(\la x^2y, x^2z, xy^2,xyz, xz^2, y^2z, yz^2 \ra)$ and ${B}=\{(1,0,0),(0,1,0),(0,0,1)\}$, 
\item[(b)]
or
$|\cI_{B}(3)|=\PP(\la  x^2z, xy^2,xyz, xz^2, y^3, y^2z, yz^2 \ra)$ and $B$ is supported  at $(1,0,0)$ and $(0,0,1)$, and    has length $2$ at $(1,0,0)$ with tangent  line $z=0$, 
\item[(c)]
or 
$|\cI_{B}(3)|=\PP(\la  x^2z-xy^2,xyz, xz^2, y^3, y^2z, yz^2,z^3 \ra)$ and $B$ is curvilinear (nonlinear) supported  at $(1,0,0)$  with tangent  line $z=0$. 
\end{itemize}
 The standard Cremona quadratic map and the first and second standard degenerate quadratic maps  (see~\cite[Example~7.1.9]{dolga})  provide   examples of such an involution in each of these cases
\begin{equation*}
\tau_a(x,y,z)=(yz,xz,yz),\quad \tau_b(x,y,z)=(xz,yz,y^2),\quad \tau_c(x,y,z)=(-xz+y^2,yz,z^2).
\end{equation*}
Suppose that (a) holds.\ Every involution $\tau$ of ${ S}$ which does not descend to
  $\PP^2$  
is given by $\tau_a\circ h$, where $h\in\PGL(3)$ permutes  the points of ${B}$.\ If $h$ fixes the points of $B$, we get  $\tau=\tau_a$ (after rescaling $x,y,z$), while if $h$ defines a transposition  of $B$, we have  $\tau([x,y,z])=[xz,yz,xy]$  in suitable coordinates.\ 
The $\tau$-invariant subspace of $H^0({ S},\omega^{-1}_{{ S}})$ is equal to $\la xyz,x^2y+yz^2,x^2z+y^2z,xy^2+xz^2\ra$ if the former holds, and to  $\la xyz,x^2y+xz^2,xy^2+yz^2,x^2z,y^2z\ra$ 
if the latter holds.\ 
Hence if the former holds,~(c3) holds;  if the latter holds,~(c4) holds. 

Suppose that (b) holds.\ The relevant birational involutions of $\PP^2$ are given by $\tau_b\circ h$, where 
 $h\in\PGL(3)$ is given by $h(x,y,z)=(\alpha x+\beta y,-\alpha y,-\alpha^{-2} z)$ or  by $h(x,y,z)=(\alpha x,\alpha y,\alpha^{-2} z)$ with $\alpha\in\CC^\star$ and $\beta\in\CC$.\ In a suitable coordinate system, $\tau$ is   $\tau_b$.\  
The $\tau_b$-invariant subspace of $H^0({ S},\omega^{-1}_{{ S}})$ is $\la  x^2z, xyz, xy^2+xz^2 ,y^2z, y^3+yz^2  \ra$, hence~(c5)   holds.

Lastly, suppose that (c) holds.\  The relevant birational involutions of $\PP^2$ are  $\tau_c\circ h$, where 
  $h([x,y,z])=[x+\beta y+\gamma z,y, z]$.\ In a suitable coordinate system, such a birational involution  is equal to $\tau_c$.\ The $\tau_c$-invariant subspace of $H^0({ S},\omega^{-1}_{{ S}})$ is $\la  x^2z-xy^2, y^3, y^2z, yz^2, z^3   \ra$, hence~(c6) holds.
\end{proof}

\begin{shownto}{long}
\green{
\begin{exam} 
Let $L(p,H),M(p,R)\subset \Sym^3\!  W_3$ be as in Definition~\ref{eccolph}.\ 
 By definition, item~(a) of Corollary~\ref{tutti2jet} holds for $L(p,H)^{\bot}$ and for $M(p,R)^{\bot}$.\ In addition,  item~(b)  of Corollary~\ref{tutti2jet} holds for $L(p,H)^{\bot}$ with $f_2=a^2$, where $p=[a]$.
\end{exam}
}
\end{shownto}

\subsubsection{Description of $K_{\sigma_0}$}\label{dimokappa}

Let $[U_6]\in K_{\sigma_0}$ and let $T_4:=U_6^{\bot}$.\ By Corollary~\ref{tutti2jet}, either $T_4=f_1\cdot U_4$, where $U_4\subset\Sym^2\! W_3^{\vee}$ is a $4$-dimensional subspace, or $T_4\supset f_2\cdot W_3^{\vee}$, where 
$f_2\in \Sym^2 \!W_3^{\vee}$.\ Hence, by~\eqref{esplicito}, Propositions~\ref{romain} and~\ref{gary} below finish  the proof of Proposition~\ref{kappapi}.

\begin{prop}\label{romain}
Let $T_4\subset \Sym^3 \!W^{\vee}_3$ be a $4$-dimensional subspace such that
 $T_4=f_1\cdot U_4$, where  $0\ne  f_1\in W_3^{\vee}$ and   $U_4\subset\Sym^2\!  W_3^{\vee}$ is a $4$-dimensional subspace.\
Then $T_4^{\bot}\in K_{\sigma_0}$ if and only if there exists a basis $(a,b,c)$ of $W^{\vee}_3$ such that 
\begin{equation*}
T_4=
\begin{cases}
a\cdot\la a^2, ab,ac,bc\ra,& \text{or}\\
a\cdot\la a^2,ab,ac,c^2\ra, &\text{or}\\
a\cdot\la a^2,b^2,bc,c^2\ra.
\end{cases}
\end{equation*}
\end{prop}

\begin{proof}
Let $R_2:=U_4^{\bot}\subset \Sym^2 \!W_3$.\ Up to the action of $\GL(W_3)$, there are $8$ possibilities for $R_2$, described  as follows in a basis $(x,y,z)$ of $W_3$.\ In  the case where the general conic  (in $\PP(W^{\vee}_3)$)  defined by $\PP(R_2)$ is smooth,  hence the base-locus is a $0$-dimensional curvilinear scheme, we have
\begin{enumerate}
\item
$R_2=\la xy,(x+y+z)z\ra$, that is, the  base-locus of the pencil of conics defined by $\PP(R_2)$ consists of $4$ distinct points;
\item
$R_2=\la xy,(x+z)z\ra$,  that is, the  base-locus of the pencil of conics defined by $\PP(R_2)$ consists of two reduced points and a  point of multiplicity $2$;
\item
$R_2=\la xy,z^2\ra$,  that is, the  base-locus of the pencil of conics defined by $\PP(R_2)$ consists of two   points of multiplicity $2$;
\item
$R_2=\la xy,x^2+yz\ra$,  that is, the  base-locus of the pencil of conics defined by $\PP(R_2)$ consists of one   point of multiplicity $3$ and a reduced point;
\item
$R_2=\la y^2,x^2+yz\ra$,  that is, the  base-locus of the pencil of conics defined by $\PP(R_2)$ consists of one   point of multiplicity $4$.
\end{enumerate}
The remaining $R_2$ are those for which all the conics parametrized by $\PP(R_2)$ are singular:
\begin{enumerate}
\item[(a)]
$R_2=\la y^2,z^2\ra$;
\item[(b)]
$R_2=\la y^2,y z\ra$;
\item[(c)]
$R_2=\la xy,xz\ra$. 
\end{enumerate}
Correspondingly, we get the following  lists of $4$-dimensional subspaces $U_4\subset \Sym^2 \!W^{\vee}_3$: 
\begin{equation}\label{primocaso}
U_4=
\begin{cases}
\la a^2,b^2,ac-c^2,bc-c^2\ra, \\
\la a^2,b^2,ac-c^2,bc\ra, \\
\la a^2,b^2,ac,bc\ra, \\
\la ac,b^2,c^2,a^2-bc\ra, \\
\la ab,ac,c^2,a^2-bc\ra,
\end{cases}
\end{equation}
and
\begin{equation}\label{secondocaso}
U_4=
\begin{cases}
\la a^2,ab,ac,bc\ra, \\
\la a^2,ab,ac,c^2\ra, \\
\la a^2,b^2,bc,c^2\ra.\\
\end{cases}
\end{equation}
Every $4$-dimensional $U_4\subset \Sym^2 \!W^{\vee}_3$ is equivalent modulo  $\GL(W_3)$ to one and only one of the spaces $U_4$ given above.\ Let  $ f_1\in W_3^{\vee}$ be nonzero and let $U_4$ be one of the subspaces in~\eqref{primocaso}.\

We claim that $(f_1\cdot U_4)^{\bot}$ does not belong to $K_{\sigma_0}$.\ To see this, first note that there exists a $1$-parameter subgroup of $\GL(W_3)$ such that $\lim_{t\to 0}\lambda(t)U_4$ is equal to the subspace in the last line of~\eqref{primocaso} (this is clear since $U_4=R_2^{\bot}$).\ Hence it suffices 
to prove that for $U_4$  as in the last line of~\eqref{primocaso}, $(f_1\cdot U_4)^{\bot}$ does not belong to~$K_{\sigma_0}$.\ Next, by acting with a $1$-parameter subgroup of~$\GL(W_3)$ given by $\diag(t^q,t^r,t^s)$ (in the given basis), with $2q=r+s$, we may assume  $f_1\in\{a,b,c\}$.\ An explicit computation then gives
\begin{eqnarray*}
(a\cdot \la ab,ac,c^2,a^2-bc\ra)^{\bot} & = & \la ab^2,b^3,b^2c,bc^2,c^3,a^3+abc\ra, \\
(b\cdot \la ab,ac,c^2,a^2-bc\ra)^{\bot} & = & \la a^3,a^2 c,ac^2,b^3,c^3,a^2b+b^2c\ra, \\
(c\cdot \la ab,ac,c^2,a^2-bc\ra)^{\bot} & = & \la a^3,a^2 b,ab^2,b^3,b^2c,a^2c+bc^2\ra.
\end{eqnarray*}
By Lemma~\ref{nonzero}, we have $\sigma_0(b^3,c^3,a^3+abc)\ne 0$, $\sigma_0(a^3,b^3,c^3)\ne 0$, and 
$\sigma_0(a^3,b^2c,ac^2+bc^2)\ne 0$.\ It follows  that the first, second, and third spaces are   not in $K_{\sigma_0}$.

We are left with $U_4$ as in~\eqref{secondocaso}.\ We know that $(a\cdot U_4)^{\bot}\in K_{\sigma_0}$.\  It remains to prove that if $f_1\notin\la x\ra$, then $(f_1\cdot U_4)^{\bot}\notin K_{\sigma_0}$.\ Acting with a suitable $1$-parameter subgroup  of 
$\GL(W_3)$, we may assume  $f_1\in\{b,c\}$.\ An explicit computation similar to the one presented above finishes the proof.
\end{proof}

\begin{prop}\label{gary}
Let $T_4\subset \Sym^3 \!W^{\vee}_3$ be a $4$-dimensional subspace.\ Suppose that  there exists a nonzero  $f_2\in\Sym^2\!  W_3^{\vee}$ such that $T_4\supset  (f_2\cdot W_3^{\vee})$.\ Then $[T_4^{\bot}]\in K_{\sigma_0}$ if and only if there exists a basis $(a,b,c)$ of $W^{\vee}_3$ such that 
\begin{equation}
T_4=
\begin{cases}
a\cdot\la a^2, ab,ac,bc\ra,& \text{or}\\
a\cdot\la a^2,ab,ac,c^2\ra.
\end{cases}
\end{equation}
\end{prop}

\begin{proof}
There exists a basis $(a,b,c)$ of $W^{\vee}_3$ and $g\in\Sym^3 \!W^{\vee}_3$ such that (according to the rank of $f_2$) 
\begin{equation}
T_4=
\begin{cases}
\la a^2b +ac^2,ab^2+bc^2,abc+c^3,g\ra,& \text{or}\\
\la a^2b,ab^2,abc,g\ra,& \text{or}\\
\la a^3,a^2b,a^2c,g\ra. 
\end{cases}
\end{equation}
Suppose that $T_4$ is as in the first line.\ Let $\lambda$ be the  $1$-parameter subgroup, diagonal in the basis  $(a,b,c)$, given by $\diag(1,t^r,t^s)$.\  Then 
$\lim_{t\to 0}\lambda(t)T_4$ is as in the second line.\ We show that for~$T_4$  as in the second line, the orthogonal $T_4^{\bot}$ is \emph{not} in $K_{\sigma_0}$.\ Let  $\lambda$ be any   $1$-parameter subgroup
diagonal in the basis  $(a,b,c)$, with pairwise distinct weights of the action on $\Sym^3 \!W^{\vee}_3$.\ Then $\ov{T}_4:=\lim_{t\to 0}\lambda(t)T_4$  is monomial and it contains $a^2b$, $ab^2$, and $abc$.\ Hence the orthogonal $ \ov{T}_4 ^{\bot}$ is monomial, of dimension~$6$, contained in
\begin{equation*}
\la a^3,a^2c,ac^2,b^3,b^2c,bc^2,c^3\ra.
\end{equation*}
A direct check shows that the above subspace contains no monomial subspace of dimension $6$ on which $\sigma_0$ vanishes.\ It follows that $[T_4^{\bot}]$ is \emph{not} in $K_{\sigma_0}$.

Suppose now that $T_4$ is as in the third line.\ We   prove by contradiction that $a\mid g$ (once that is known, we might need to rename $b,c$).\ 
Let $\lambda$ be a $1$-parameter subgroup, diagonal in the basis  $(a,b,c)$, given by $\diag(1,t^r,t^s)$, where $r>3s$.\ 
Then $\ov{T}_4:=\lim_{t\to 0}\lambda(t)T_4$  is monomial and by our assumption $a\nmid g$ , there exist $i,j$ such that  
$\ov{T}_4=\la a^3,a^2b,a^2c,b^i c^j\ra$.\ Hence $ \ov{T}_4 ^{\bot}$ contains $\la ab^2,abc,ac^2\ra$ and  is therefore   not  in $K_{\sigma_0}$.\ It follows that $[T_4^{\bot}]$  is not  in $K_{\sigma_0}$.
\end{proof}

\subsection{Orbit and stabilizer} \label{sect74}
Recall that   $V_{10}= \Sym^3\! W_3$.\
 Since $\gsl(3)=\Gamma_{1,1} $ and
  \begin{equation*}
\End(V_{10} )=\Gamma_{3,3} \oplus \Gamma_{2,2}\oplus \Gamma_{1,1}\oplus \Gamma_{0,0}
,
 \end{equation*}
 it follows from the decomposition~\eqref{deco} that
there is an exact sequence
  \begin{equation}\label{sl}
0\to \gsl(3)\to \End(V_{10})\to \bw3V_{10}^\vee \xrightarrow{\ a\ } \Gamma_{0,6}\to 0.
 \end{equation}
We prove below that the stabilizer of $[\sigma_0]$ is $\SL(3)$.\ The normal space at $[\sigma_0]$ to the $\SL(V_{10})$-orbit of $[\sigma_0]$   is therefore $ \Gamma_{0,6}=H^0(\P(W_3),\cO_{\P(W_3)}(6))$.\ The map $a$ was given a geometric interpretation in~\eqref{isomo}.

\begin{prop}\label{stabver}
The stabilizer of $[\sigma_0]$ in $\SL(V_{10})$ is equal to the image of $\SL(W_3)\to \SL(V_{10})$ and the point
 $[\sigma_0]\in \P(\bw3V_{10}^\vee)$ is polystable for the $\SL(V_{10})$-action.
\end{prop}

\begin{proof}
The stabilizer contains  $\SL(W_3)$ by choice of $\sigma_0$.\ Conversely, if $g\in\SL(V_{10})$ stabilizes $[\sigma_0]$,   it maps $X_{\sigma_0}$ to itself,   hence the singular locus of $X_{\sigma_0}$ to itself.\  By Proposition~\ref{singeffe}, this singular locus  is equal to  $\gv\subset\PP(\Sym^3\! W_3)$.\ Thus $g$ maps to itself the  subvariety of $\PP(\Sym^3\! W_3)$ swept out by projective tangent planes to the Veronese surface ${\mathsf V}$.\ Since the singular locus of this subvariety is~${\mathsf V}$, the automorphism $g$ maps ${\mathsf V}$ to itself,  hence  belongs to $\SL(W_3)$.

It follows from Proposition~\ref{normal} that this stabilizer has finite index in its normalizer, hence~$[\sigma_0]$ is polystable by  \cite[Corollaire~3]{luna}.
 \end{proof}

\subsection{Degenerations}\label{dimodeg2}

The following theorem is the main  result of Section~\ref{div2}.\  We consider a general $1$-parameter deformation  $(\sigma_t)_{t\in \Delta}$ of our trivector $\sigma_0$.\ By the exact sequence~\eqref{sl}, we obtain a general element of $H^0(\P(W_3),\cO_{\P(W_3)}(6))$, hence a double cover $S\to \P(W_3)$ branched along the \rem{sextic} curve that it defines, where $S$ is a K3 surface of degree~$2$.\ The moduli  space  $\cM_S(0,L,1)$, a \hk\ fourfold birationally isomorphic to $\SS$, was defined in Remark~\ref{rema33}.
 
\begin{theo}\label{degenera2}  
Let $(\sigma_t)_{t\in \Delta}$ be  a  general  $1$-parameter deformation.\
Over a finite cover $\Delta'\to \Delta$, there is a family of smooth polarized \hk\ fourfolds
$\cK'\to \Delta'$ such that a general fiber
$\cK'_{t'}$ is isomorphic to $K_{\sigma_t}$ and the central fiber
  is   isomorphic  to
 $\cM_S(0,L,1)$, where $S$ is a general $K3$ surface of degree $2$, with the polarization $6\L2-5\delta$.
\end{theo}

The  proof will be given at the very end of this section.

Set $\cG:= \Gr(3,V_{10})\times\aff$ and consider the blow up  $$\varphi\colon\wt\cG:=\Bl_{ \gv \times\{0\}} \cG  \lra \cG$$ 
(see \eqref{mappagauss} for the definition of the surface $\gv$)
.\ The exceptional divisor   $E\to \gv$ is a bundle of $19$-dimensional projective spaces.\ We view $\wt\cG\to\aff$ as a degeneration of $\Gr(3,V_{10})$  with central fiber 
$\Bl_{\gv}\Gr(3,V_{10}) \cup E$.\ 

Write the deformation  in~Theorem~\ref{degenera2} as $\sigma_t=\sigma_0+t\sigma+O(t^2)$, where, by the analysis of Section~\ref{sect74}, we may assume that      $\sigma$ is very general in $\Sym^6 \!W_3^{\vee}\subset  \bw3 V_{10}^{\vee}$.\   Consider   the strict transform $\wt\cX\subset \wt\cG$
of
\begin{equation}\label{t2}
\{([U_3],t)\in \cG\mid \sigma_{t^2}\vert_{U_3}\equiv 0\},
\end{equation}
with projection $\pi\colon \wt\cX\to \aff$.\ By~\eqref{isomo}, the hypersurface $X_{\sigma}$ intersects  transversely $\gv$ and $\div(\sigma)$ is identified with $C:=X_{\sigma}\cap \gv$.\ Hence
\begin{equation*}
\wt\cX_t:=\pi^{-1}(t)\cong
\begin{cases}
X_{\sigma_{t^2}} & \text{if $t\ne0$,} \\
\Bl_{\gv} X_{\sigma_0} \cup Q & \text{if $t=0$,}
\end{cases}
\end{equation*}
where $Q\subset E$ is a  bundle of $18$-dimensional quadrics  over $\gv$, with  smooth fibers over $\gv\setminus C$ and fibers of corank $1$ over $C$ (this follows from Lemma~\ref{normeffe} and holds  because we performed a degree-$2$ base change in \eqref{t2}).
 
We identify $K_{\sigma}$ with the   closed subset of the Hilbert scheme of  $X_{\sigma}$ defined by
\begin{equation*}
\{  [U_6]\in\Gr(6,V_{10})\mid\Gr(3,U_6)\subset X_{\sigma}
\}.
\end{equation*}
This defines a subscheme $\cK\to \aff^\star  $ of the relative Hilbert scheme $\Hilb(\wt\cX/\aff)$, with fiber $K_{\sigma_0+t^2\sigma}$ at $t$, and we take 
its   schematic closure  $\rho\colon \wt\cK\to\aff$.

\begin{prop}\label{elletilde}
There exists an irreducible component $K'_L$ of  $\wt\cK_0$ which is birationally isomorphic to $S^{[2]}$, where $S$ is the degree-$2$ K3 surface of Theorem~\ref{degenera2}.
\end{prop}

\begin{proof}
Let $[U_6]\in  K_L \setminus  K_M $.\ By Remark~\ref{interver}, the scheme-theoretic intersection $\Gr(3,U_6)\cap\gv$ is   two reduced points
 $p_1,p_2$.\ Let $\wt{\Gr}(3,U_6)\subset \wt\cX_0$ be the strict transform of $\Gr(3,U_6)$, that is,~the blow up of $\Gr(3,U_6)$ at $p_1,p_2$.\  We have  $\wt{\Gr}(3,U_6)\cap Q=\{A_1,A_2\}$, where $A_i$, for $i\in\{1,2\}$,  is an $8$-dimensional linear subspace of the fiber $E_{p_i}$ of $E$ over $p_i$, contained in  the fiber $Q_{p_i}$ of $Q$ over~$p_i$.\ 
 Every subscheme of $\wt\cX_0$ given by
 \begin{equation}\label{controrho}
\wt{\Gr}(3,U_6)\cup R_1\cup R_2, \qquad  A_i\subset R_i\subset Q_{p_i},\quad [R_i]\in\Gr(9,E_{p_i})
\end{equation}
corresponds to a point of  $\wt\cK_0$.\ Moreover, by Proposition~\ref{kappapi}, these subschemes are parametrized by an \emph{open} subset   of the fiber   $\Hilb(\wt\cX/\aff)_0$, whose   closure  in  $\Hilb(\wt\cX/\aff)$ (equivalently, in  $\wt\cK$) is therefore an irreducible component of $\wt\cK_0$; we denote it by $K'_L$.\ Now 
 $Q_{p_i}$ is an $18$-dimensional quadric, either smooth or of corank $1$, which is smooth at each point of~$A_i$ (Lemma~\ref{normeffe}).\ It follows that there are exactly two $9$-dimensional linear subspaces of $Q_{p_i}$ containing $A_i$ if $Q_{p_i}$ is smooth (that is,~if $p_i\notin C$) and  one such  linear subspace if $Q_{p_i}$ is singular (that is,~if $p_i\in C$).\ 
 
 {By construction, an open dense subset $K_L^{\prime0}$ of $K'_L$ parametrizes subschemes as in~\eqref{controrho}, where $[U_6]\in K_L$ is such that 
$\Gr(3,U_6)\cap\gv$ is  reduced (of length $2$).\ The set of such $[U_6]$ is exactly $K_L\setminus K_M$.\ We have a forgetful map
\begin{equation}
\begin{aligned} \label{oublier}
K_L^{\prime0}& \lra  K_L\setminus K_M \\
\wt{\Gr}(3,U_6)\cup R_1\cup R_2 & \longmapsto   [U_6].
\end{aligned}
\end{equation}
Let  $\rho\colon S^{(2)}\to\PP(W_3)^{(2)}$ be the map induced by the double cover $S\to\PP(W_3)$.\ By definition of $R_1,R_2$, the map in~\eqref{oublier} can be identified with the map
\begin{equation*}
S^{(2)}\setminus\{\rho^{-1}(2x) \mid x\in \PP(W_3)\}  \lra  \PP(W_3)^{(2)}\setminus\{2x \mid x\in\PP(W_3)\} 
\end{equation*}
obtained by restricting  $\rho$.\ In particular, 
 $K'_L$ is birationally isomorphic to $S^{[2]}$ and the forgetful map $K'_L\to K_L$ has degree $4$.}
\end{proof}

\begin{shownto}{long}
\green{\begin{rema}
Let  $[U_6]\in  K_L \setminus  K_M $.\ Then $\Gr(3,U_6)\cap\gv=\emptyset$ by Remark~\ref{interver},  hence we may view 
$\Gr(3,U_6)$ as subscheme of $\Bl_{\gv} X_{\sigma_0} $.\ Clearly, $\Gr(3,U_6)$ belongs to $\wt\cK_0$.\ By Proposition~\ref{kappapi},
we get an irreducible component  of $\wt\cK_0$
which is  birationally isomorphic to  $K_M$,  hence also  to $\PP(W_3)\times\PP(W^{\vee}_3)$.
 \end{rema}}
   \end{shownto}

\begin{prop}\label{multelletilde}
 The  irreducible component $K'_L$ has  multiplicity  one in   $\wt\cK_0$.
\end{prop}

\begin{proof}
A  point $x$ of $K_L^{\prime0}$ (notation as in the proof of Proposition~\ref{elletilde}) parametrizes a scheme $Z:=\wt{\Gr}(3,U_6)\cup R_1\cup R_2$ as in~\eqref{controrho}, where
the scheme-theoretic intersection 
$\Gr(3,U_6)\cap\gv$ is the union of two reduced points
 $p_1=[U_{3,1}]$ and $p_2= [U_{3,2}]$, neither of which is contained in $X_{\sigma}$.

The scheme  $Z$  is locally a complete intersection in $Y_0$.\ Hence there is a well-defined normal bundle $N_{Z/Y_0}$ and it suffices to prove   $H^1(Z,N_{Z/Y_0})=0$ (because  $ K'_L$ is an open neighborhood of~$x$ in    the fiber   $\Hilb_P(\wt\cX/\aff)_0 $).\ In order to simplify notation, set $X_0:=X_{\sigma_0}$ and  $\wt{X}_0:=\Bl_{\gv}X_0$.\ We have
\begin{equation*}
N_{Z/Y_0}\vert_{\wt{\Gr}(3,U_6)}\cong N_{\wt{\Gr}(3,U_6)/\wt{X}_0},\quad N_{Z/Y_0}\vert_{R_i}\cong N_{R_i/Q_{p_i}}.
\end{equation*}
One easily checks   $H^1(R_i,N_{R_i/Q_{p_i}}(-1))=0$.\ In order to prove   $H^1(Z,N_{Z/Y_0})=0$, it therefore suffices to show  
\begin{equation}\label{svaporaz}
H^1(\wt{\Gr}(3,U_6),N_{\wt{\Gr}(3,U_6)/\wt{X}_0})=0.
\end{equation}
Let $\wt{\Gr}(3,V_{10}):=\Bl_{\gv}\Gr(3,V_{10})$.\
We have the normal exact sequence
\begin{equation}\label{normnorm}
0 \to N_{\wt{\Gr}(3,U_6)/\wt{X}_0}\to N_{\wt{\Gr}(3,U_6)/\wt{\Gr}(3,V_{10})}\to \cO_{\wt{\Gr}(3,V_{10})}(\wt{X}_0)\vert_{\wt{\Gr}(3,U_6)}\to 0.
\end{equation}
We claim that
\begin{equation}\label{nosezioni}
H^0(\wt{\Gr}(3,U_6), \cO_{\wt{\Gr}(3,V_{10})}(\wt{X}_0)\vert_{\wt{\Gr}(3,U_6)})= 0.
\end{equation}
In fact, the natural map $\psi\colon \wt{\Gr}(3,V_{10})\to \Gr(3,V_{10})$ is the blow up of the  points $p_1$ and $p_2$.\ Let $A=A_1+A_2$ be the exceptional divisor of $\psi$ and let $\cO_{\Gr}(1)$ be the Pl\"ucker line bundle on $\Gr(3,V_{10})$.\ Since $X_0$ is a divisor in $|\cO_{\Gr}(1)|$ with multiplicity $2$ along $\gv$, we have
\begin{equation}\label{restnorm}
\cO_{\wt{\Gr}(3,V_{10})}(\wt{X}_0)\vert_{\wt{\Gr}(3,U_6)}\isom \cO_{\wt{\Gr}(3,U_6)}(\psi^{*}\cO_{\Gr}(1)-2A).
\end{equation}
Let  $x$ be a general point in $ \wt{\Gr}(3,U_6)$ 
and set $[U_3]:=\psi(x)\in \Gr(3,U_6)$.\ We may assume that $U_3$ is transverse to $U_{3,1} $ and $U_{3,2} $, hence there exists a Segre embedding $\Phi\colon\PP^1\times\PP^2\hra\PP(U_6)$ such that $\Phi(\{(0,1)\}\times \P^2)=\P(U_{3,1}) $, 
$\Phi(\{(1,0)\}\times \P^2)=\P(U_{3,2} )$, and $\Phi(\{(1,1)\}\times \P^2)= \PP(U_3)$.\ 
Let $\phi\colon\PP^1\to\Gr(3,U_6)$ be the map defined by $\Phi$ and let  $\Gamma\subset\wt{\Gr}(3,U_6)$ be the strict transform of $\phi(\PP^1)$.\ Then  $\Gamma\cdot\psi^{*}\cO_{\Gr}(1)=3$ and 
$\Gamma\cdot A=2$,  hence  $\Gamma\cdot(\psi^{*}\cO_{\Gr}(1)-2A)=-1$.\ It follows that any section of the right side of~\eqref{restnorm} vanishes at  general points of  $\wt{\Gr}(3,U_6)$ hence is the zero section.\
This proves~\eqref{nosezioni}.

By~\eqref{nosezioni} and~\eqref{normnorm},  it suffices, in order to prove~\eqref{svaporaz}, to prove 
\begin{equation*}
H^1(\wt{\Gr}(3,U_6), N_{\wt{\Gr}(3,U_6)/\wt{\Gr}(3,V_{10})})= 0.
\end{equation*}
The differential of $\psi$ defines an exact sequence
\begin{equation*}
0\to N_{\wt{\Gr}(3,U_6)/\wt{\Gr}(3,V_{10})} \to \psi^{*}N_{\Gr(3,U_6)/\Gr(3,V_{10})}  \xrightarrow{\ a\ } \cO^{\oplus 10}_{A_1}\oplus\cO^{\oplus 10}_{A_2} \to 0.
\end{equation*}
The map  induced by $a$ on global sections is surjective, because  the subspaces of $U_6$ corresponding to $p_1$, $p_2$ are transverse.\ Since $H^1(\Gr(3,U_6),N_{\Gr(3,U_6)/\Gr(3,V_{10})})=0$, the desired vanishing follows from the long exact sequence  associated with this exact sequence.
\end{proof}

\begin{proof}[Proof of Theorem~\ref{degenera2}] 
By Propositions~\ref{elletilde} and~\ref{multelletilde}, and by (the proof of)~\cite[Theorem~(0.1)]{KLSV}, we obtain, as in the proof of Theorem~\ref{theopourgenre6},
after a suitable finite base change, a smooth family of polarized \hk\ fourfolds with (smooth) central fiber birationally isomorphic to $\SS$ with the polarization $6L-5\delta$.\ It follows from~Remark~\ref{rema33} that this central fiber is isomorphic to $(\cM_S(0,L,1),6\L2-5\delta)$.
\end{proof}

\section{The   divisor $\cD_{30}$}\label{ch8}

Let $(S,L)$ be a general polarized K3 surface of degree $30$.\ Unfortunately, little geometric information on $S$ is available and we 
were not able to find a trivector on some $10$-dimensional vector space
$V_{10}$ to relate $\hilbS$ to Debarre--Voisin varieties, \rem{nor were we able to decide whether~$\cD_{30} $ is an HLS divisor.}\ We will however   construct on $\hilbS$ a canonical rank $4$-vector bundle with the same numerical invariants as the restriction of the
tautological quotient bundle of $\Gr(6,V_{10})$ to a Debarre--Voisin variety.

\subsection{The rank-$4$ vector bundle $\quotauto_4$ over $\hilbS$}\label{g16Q4}
By Mukai's work (\cite{mukai2}), there is a \rem{simple and} rigid \mbox{rank-$2$} vector bundle $\cF$ on $S$ with $c_1(\cF)=L$ and
Euler characteristic $\chi(S,\cF)=10$.\ Moreover,  $\cF$ is globally generated and  the vector space
$W_{10}:=H^0(S,\cF)$ has dimension~$10$.

With the notation of Section~\ref{sectaut}, we let $\Tauto{\cF}$ be the tautological rank-$4$ vector bundle on $\hilbS$
associated with $\cF$.\
We have $c_1(\Tauto{\cF})=\L2-2\delta$ and $H^0(S^{[2]},\Tauto{\cF})=W_{10}$.

Consider now the tautological rank-$6$ vector bundle $\Tauto{\Sym^2\! \cF}$ constructed
on $\hilbS$ from the rank-$3$ vector bundle $\Sym^2\! \cF$ over $S$.

\begin{lemm}\label{defE4}
  The natural evaluation map
$$
\ev^+\colon  \Sym^2\! \Tauto{\cF} \lra \Tauto{\Sym^2\!\cF}
$$
is \rem{surjective.}\ Its kernel $\quotauto_4$ is a rank-$4$ vector bundle over $\hilbS$
with $c_1(\quotauto_4) = 2\L2-7\delta$.
\end{lemm}

\begin{proof}
Consider as in Section~\ref{sectaut} the double cover  $p:\isospecS \to \hilbS $
defined by the blow up $\isospecS$ of $S\times S$ along its diagonal.\ Let $q_1$
be the first projection to $S$, so that  $\Tauto{\cF}=p_*(q_1^*\cF)$.\ Tensorize the canonical surjection $p^*\Tauto{\cF} \to q_1^* \cF$  by the vector bundle $q_1^* \cF$ to obtain the exact sequence
$$
(p^*\Tauto{\cF}) \otimes q_1^* \cF \to q_1^* (\cF \otimes  \cF) \to 0.
$$
Its pushforward by the finite morphism $p$ gives with the projection
formula a surjection
$$
\ev\colon \Tauto{\cF} \otimes \Tauto{\cF}  \thra \Tauto{\cF \otimes \cF} .
$$
The map $\ev^+$ being the invariant part of $\ev$,  it is also
surjective.\ Its kernel $\quotauto_4$ is therefore  a vector bundle of rank $4$ and
 we have
$c_1(\Sym^2\! \Tauto{\cF} )=5c_1 (\Tauto{\cF})=5\L2-10\delta$
and $c_1(\Tauto{\Sym^2\!\cF})=3\L2-3\delta,$
so $c_1(\quotauto_4)=2\L2-7\delta$.
\end{proof}

\rem{ \begin{rema}\label{analogieD18}
   If we replace in this construction $\cF$ by the Mukai bundle $\cE_2$ over a $K3$-surface of
   degree $18$, the antiinvariant part
 $
\ev^-\colon  \bw2\! \Tauto{\cE_2} \thra \Tauto{\bigwedge^2\!\cE_2}
$ of $\ev$ is  the surjection in sequence~\eqref{eq20}.\ So, in the
degree-$18$ case, $\quotauto_4$ was defined as the kernel of $\ev^-$.
 \end{rema}}

\begin{lemm}
The vector space  $H^0(\hilbS, \quotauto_4)$ has dimension {at least} $10$ and  is canonically isomorphic to the kernel
 $$
V_{10} := \Ker ( \Sym^2\!W_{10} \lra H^0(S,\Sym^2\!\cF )).
$$
\end{lemm}

 {We expect this map to be onto, so that   $ V_{10}$ would have dimension $10$}.

\begin{proof}
  By \cite[Theorem~1]{dan2} or \cite[Theorem~6.6]{krug}, the canonical maps
\begin{eqnarray}
H^0(S,\cF)& \lra& H^0(\hilbS, \Tauto{\cF})\nonumber \\
H^0(S,\Sym^2\! \cF) &\lra& H^0(\hilbS,\Tauto{\Sym^2 \!\cF})\label{l2}\\
     H^0(S,\cF) \otimes H^0(S,\cF)& \lra&  H^0(\hilbS, \Tauto{\cF}\otimes \Tauto{\cF}) \label{l3}
\end{eqnarray}
are isomorphisms.\ By definition of $\quotauto_4$, we have an exact sequence
\begin{equation*}
  0\to H^0(\SS,\quotauto_4)\to
H^0(\SS,\Sym^2\! \Tauto{\cF})\to H^0(\SS,\Tauto{\Sym^2\!\cF}).
\end{equation*}
\rem{Since \eqref{l3} is bijective, its middle space is  isomorphic to $ \Sym^2\! H^0(S,\cF)=\Sym^2\!W_{10} $; since \eqref{l2} is bijective, the rightmost space is   isomorphic to $H^0(S,\Sym^2\!\cF )$.\ We therefore conclude
 that $H^0(\hilbS, \quotauto_4) $ is isomorphic to $ V_{10} $.}

We will show
that $H^1(S,\cF\otimes \cF) =H^2(S,\cF\otimes \cF) = 0 $  on a specific   $K3$ surface $S$ of degree~$30$  introduced by Mukai in \cite[\S6]{mukai2}, hence on a general K3 surface.\ This surface has  an elliptic fibration $S\to\P^1$ with general fiber $A_1$
 and Mukai shows that $\cF$
fits in an exact sequence
\begin{equation}
  \label{seqFpinceau}
  0 \to \OO_S(A_1) \oplus \OO_S(A_1) \to \cF \to \OO_Z(5z) \to 0,
\end{equation}
where $Z\subset  S$ is  a smooth rational curve  {and $z$ is the class of a point
on $Z$}.\ \rem{Tensoring~\eqref{seqFpinceau} by~$\OO_S(A_1)$, we get $H^2(S,\cF(A_1))=0$, and tensoring it by~$\cF$, we get $H^2(S,\cF\otimes \cF) =0$.

Since $\cF$ is globally generated, we have $H^1(Z,\cF\otimes \OO_Z(5z))=0$ and, tensoring~\eqref{seqFpinceau} by~$\cF$, we get a surjection
\begin{equation}
  \label{surj}
  H^1(S,\cF(A_1))^{\oplus 2}\thra H^1(S,\cF\otimes \cF).
  \end{equation}
Mukai showed that on this particular surface, one has $H^1(S,\cF)=H^2(S,\cF)=0$, hence
\begin{equation}
  \label{iss}
  H^1(S,\cF(A_1))\isom H^1(S,\cF\vert_{A_1})\isom H^2(S,\cF(-A_1))\isom H^0(S,\cF(A_1-H))^\vee,
    \end{equation}
where $\OO_S(H) := \bw2\cF=L$ is the polarization.\ Moreover, we have $Z\lin H-2A_1$, $A_1 \cdot H =
8$, and~$H^2=30$, and the sequence~\eqref{seqFpinceau} gives an exact sequence
$$
  0 \to \OO_S(2A_1-H) \oplus \OO_S(2A_1-H) \to \cF(A_1-H) \to \OO_Z(-z) \to 0.
$$
This implies $H^0(S,\cF(A_1-H))=0$, hence $ H^1(S,\cF(A_1))=0$ by~\eqref{iss}.\ Finally, the surjection~\eqref{surj} implies 
$H^1(S,\cF\otimes \cF) =0$.\ 

Going back to a general K3 surface $S$, where   the   vanishings $H^1(S,\cF\otimes \cF) =H^2(S,\cF\otimes \cF) =0$ still hold, we get
 $$h^0(S,\Sym^2\!\cF)= \chi(S,\Sym^2\!\cF)= 45$$ and, by definition of $V_{10}$,
$$\dim (V_{10})  \ge\dim(\Sym^2\!W_{10})
-h^0(S,\Sym^2\!\cF)=10.$$
This finishes the proof of the lemma.}
\end{proof}

From the   previous two lemmas, we obtain the following result, where we use, as in Remark~\ref{chernn}, the package 
  Schubert2 of Macaulay2 (\cite{macaulay2}) to compute the numerical invariants of
 the vector bundle $\cQ_4$ on $\hilbS$ (the code can be found in \cite{comp}).

\begin{prop}
Let $(S,L)$ be a general polarized K3 surface of degree $30$.\
The vector bundle~$\cQ_4$ induces a  rational map  $\hilbS\dra\Gr(6,V_{10})$ which corresponds to  
the polarization given  in the last column of Table~\ref{enumero1}.\
Moreover,   
 the vector bundle $\cQ_4$  has the same  Segre numbers as the rank-$4$
 tautological quotient bundle on   Debarre--Voisin varieties $K_\sigma\subset
 \Gr(6,10)$.
\end{prop}

\subsection{Geometric interpretation}
Let $X$ be the image in $\P(W_{10}^{\dual})$ of the scroll $\P(\cF^{\dual})$ by the projection from $S\times \P(W_{10}^{\dual})$ to $\P(W_{10}^{\dual})$.

We have $V_{10} = H^0(\PP(W_{10}^\dual),\cI_X(2))$, where $\cI_X$ is the ideal sheaf of $X$ in
$\PP(W_{10}^\dual)$.\ We want to describe, for   general points $x,y\in 
S$, the $6$-dimensional vector space $\soustauto_{6,\{x,y\}}$  defined by the exact sequence
$$
0 \to\soustauto_{6,\{x,y\}}\to V_{10} \to \quotauto_{4,\{x,y\}}\to 0.
$$

\begin{prop}
The vector space  $\soustauto_{6,\{x,y\}} $ is the  space of quadratic forms vanishing on $X$ and on the projective subspace $\PP_3=\PP(\cF_x^\dual \oplus \cF_y^\dual)$ of $\P(W_{10}^{\dual})$.
\end{prop}

\begin{proof}
The fiber over $\{x,y\}$ of the evaluation map defined in
  Lemma~\ref{defE4} gives an exact sequence
  $$
0 \to \soustauto_{6,\{x,y\}} \to V_{10} \to \Sym^2(\cF_x \oplus \cF_y) \to \Sym^2\! \cF_x \oplus \Sym^2\!\cF_y \to 0,
$$
hence $\soustauto_{6,\{x,y\}}$ consists of elements of $V_{10} $ that also vanish on $\PP(\cF_x^\dual \oplus \cF_y^\dual)$.
\end{proof}

    \end{document}